\documentclass[smallextended]{svjour3}
\usepackage{fullpage}
\smartqed  
\usepackage{graphicx}
\usepackage{fix-cm}

\usepackage{epstopdf}
\usepackage{hyperref}
\usepackage{subfigure}
\usepackage{algorithm}
\usepackage{algorithmic}
\usepackage{url}
\usepackage{amsmath,amsfonts,amssymb}
\usepackage{graphicx}
\usepackage{algorithm,algorithmic,mathtools}
\usepackage{color,cases}
\usepackage{tikz}
\usetikzlibrary{calc,shapes}
\usepackage{subfigure}
\usepackage{nicefrac}

\newcommand{\norm}[1]{\left\| #1 \right\|}
\newcommand{\Exp}{\mathop{\mathbb E}\displaylimits}

\newcommand{\dd}{\textup{DA}}

\newtheorem{assumption}{Assumption}

\newcommand{\argmin}{\mathop{\mathrm{arg\,min}{}}}

\newcommand{\ba}{\begin{array}}
	\newcommand{\ea}{\end{array}}
\newcommand{\beq}{\begin{equation}}
\newcommand{\eeq}{\end{equation}}
\newcommand{\beqa}{\begin{eqnarray}}
\newcommand{\eeqa}{\end{eqnarray}}
\newcommand{\beqas}{\begin{eqnarray*}}
	\newcommand{\eeqas}{\end{eqnarray*}}
\newcommand{\bi}{\begin{itemize}}
	\newcommand{\ei}{\end{itemize}}

\def\eqref#1{(\ref{#1})}

\def\bx{{\bar x}}

\def\tx{{\tilde x}}

\def\shownotes{0}
\ifnum\shownotes=1
\newcommand{\authnote}[2]{{ $\ll$\textsf{\footnotesize #1 notes: #2}$\gg$}}
\else
\newcommand{\authnote}[2]{}
\fi

\begin{document}
\title{Distributed Stochastic Variance Reduced Gradient Methods and A Lower Bound for Communication Complexity}


\author{Jason D. Lee         \and
         Qihang Lin       \and
         Tengyu Ma 
         \and \\
         Tianbao Yang
}


\institute{
        Jason D. Lee \at
              Department of Electrical Engineering and Computer Science, UC Berkeley \\
              \email{jasondlee@berkeley.edu}           
           \and
           Qihang Lin \at
           Tippie College of Business, The University of Iowa\\
           \email{qihang-lin@uiowa.edu}
           \and
           Tengyu Ma \at
           Department of Computer Science, Princeton University\\
           \email{tengyu@cs.princeton.edu}
           \and
           Tianbao Yang \at
           Department of Computer Science, The University of Iowa\\
           \email{tianbao-yang@uiowa.edu}
}

\date{Received: date / Accepted: date}

\maketitle

\begin{abstract}
We study distributed optimization algorithms for minimizing the average of convex functions. The applications include empirical risk minimization problems in statistical machine learning where the datasets are large and have to be stored on different machines. We design a distributed stochastic variance reduced gradient algorithm   that, under certain conditions on the condition number, simultaneously achieves the optimal parallel runtime, amount of communication and rounds of communication among all distributed first-order methods up to constant factors. Our method and its accelerated extension also outperform existing distributed algorithms in terms of the rounds of communication as long as the condition number is not too large compared to the size of data in each machine. We also prove a lower bound for the number of rounds of communication for a broad class of distributed first-order methods including the proposed algorithms in this paper. We show that our accelerated distributed stochastic variance reduced gradient algorithm achieves this lower bound so that it uses the fewest rounds of communication among all distributed first-order algorithms.
\keywords{Distributed Optimization \and Communication Complexity \and Machine Learning \and First-Order Method}
\end{abstract}

\section{Introduction}
\label{intro}

In this paper, we consider the \emph{distributed optimization} problem of minimizing the average of $N$ convex functions in $\mathbb{R}^d$, i.e.,

\begin{equation}
\min_{x\in \mathbb{R}^d}\left\{f(x) := \frac{1}{N}\sum_{i=1}^Nf_i(x)\right\}
\label{eqn:obj}
\end{equation}
\normalsize
using $m$ machines. For simplicity, we assume $N=mn$ for an integer $n$ with $m\ll n$ but all of our results can be easily generalized for a general $N$. Here, $f_i:\mathbb{R}^d\rightarrow\mathbb{R}$ for $i=1,\dots,N$ is convex and $L$-smooth, meaning that $f_i$ is differentiable and its gradient $\nabla f_i$ is $L$-Lipschitz continuous\footnote{In this paper, the norm $\|\cdot\|$ represents Euclidean norm.}, i.e.,
$\|\nabla f_i(x)-\nabla f_i(y)\|\leq L\|x-y\|,~\forall x,y\in\mathbb{R}^d$,
and their average $f$ is $\mu$-strongly convex, i.e.,
$\|\nabla f(x)-\nabla f(y)\|\geq \mu\|x-y\|,~\forall x,y\in\mathbb{R}^d$.
We call
$\kappa = \frac{L}{\mu}$
the \emph{condition number} of function $f$. Note that the function $f$ itself can be $L_f$-smooth, namely,
$\|\nabla f(x)-\nabla f(y)\|\leq L_f\|x-y\|,~\forall x,y\in\mathbb{R}^d$,
for a constant $L_f\leq L$. Let $x^*$ be the unique optimal solution of \eqref{eqn:obj} and a solution $\hat x$ is called an \emph{$\epsilon$-optimal solution}\footnote{If $\hat x$ is a random variable generated by a stochastic algorithm, we call it an $\epsilon$-optimal solution if $\mathbb{E}[f(\hat x)-f(x^*)]\leq\epsilon$.} for \eqref{eqn:obj} if
$f(\hat x)-f(x^*)\leq\epsilon$.

One of the most important applications of problem \eqref{eqn:obj} is \emph{empirical risk minimization} (ERM) in statistics and machine learning. Suppose there exists a set of i.i.d. samples $\{\xi_1,\xi_2,\dots,\xi_N\}$ from an unknown distribution $D$ of a random vector $\xi$. An ERM problem can be formulated as

\begin{equation}
\min_{x\in \mathbb{R}^d}\frac{1}{N}\sum_{i=1}^N\phi(x,\xi_i)
\label{eqn:emr}
\end{equation}
\normalsize
where $x$ represents a group of parameters of a predictive model, $\xi_i$ is the $i$th data point, and $\phi(x,\xi)$ is a loss function. Note that \eqref{eqn:emr} has the form of \eqref{eqn:obj} with each function $f_i(x)$ being $\phi(x,\xi_i)$. Typically, the data point $\xi$ is given as a pair $(a,b)$ where $a\in\mathbb{R}^d$ is a feature vector and $b\in\mathbb{R}$ is either a continuous (in regression problems) or a discrete response (for classification problems). The examples of loss function $\phi(x,\xi)$ with $\xi=(a,b)$ include: square loss in linear regression where $a\in\mathbb{R}^d$, $b\in\mathbb{R}$, and $\phi(x,\xi)=(a^Tx-b)^2$; logistic loss in logistic regression where $a\in\mathbb{R}^d$, $b\in\{1,-1\}$, and $\phi(x,\xi)=\log(1+\exp(-b(a^Tx))$; smooth hinge loss where $a\in\mathbb{R}^d$, $b\in\{1,-1\}$, and

$$
\phi(x,\xi)=\left\{
\begin{array}{ll}
0&\text{ if }ba^Tx\geq1\\
\frac{1}{2}-ba^Tx&\text{ if }ba^Tx\leq0\\
\frac{1}{2}(1-ba^Tx)^2&\text{ otherwise}.
\end{array}
\right.
$$
\normalsize
To improve the statistical generalization properties of the model learned from \eqref{eqn:emr}, a regularization term $\frac{\lambda}{2}\|x\|^2$ is often added to \eqref{eqn:emr} and the problem becomes a regularized ERM problem

\begin{equation}
\min_{x\in \mathbb{R}^d}\frac{1}{N}\sum_{i=1}^N\phi(x,\xi_i)+\frac{\lambda}{2}\|x\|^2
\label{eqn:regemr}
\end{equation}
\normalsize
which still takes the form of \eqref{eqn:obj} with $f_i(x)=\phi(x,\xi_i)+\frac{\lambda}{2}\|x\|^2$.  The parameter $\lambda$ is called an regularization parameter. As argued by~\cite{shamir2014distributed,shamir2014communication,shalev2009stochastic,zhang2015communication}, for ERM problem,  the value of $\lambda$ is typically in the order of $\Theta(1/\sqrt{N}) = \Theta(1/\sqrt{mn})$.

We consider a situation where all $N$ functions are initially stored in the same large storage space that has limited computation power. We assume that each of the $m$ machines we use to solve \eqref{eqn:obj} has a limited memory space of $C$ so that it can load at most $C$ of the $N$ functions in \eqref{eqn:obj}. In the case of ERM, this means each machine can load at most $C$ data points among $\{\xi_1,\xi_2,\dots,\xi_N\}$ in its memory. Since the data point $\xi_i$ uniquely defines $f_i$ in ERM, in the rest of the paper, we will call $f_i$ a data point $i$ or a function $i$ interchangeably.

Throughout the whole paper, we assume that
\begin{assumption}
\label{assumption_C}
The memory space $C$ of each machine satisfies
$n<C< N$
and the quantity
$\tilde n\equiv C-n$
satisfies $\tilde n\geq cn$ for a universal constant $c>0$.
\end{assumption}
The inequality $C<N$ forces us to use more than one, if not all, the machines for solving \eqref{eqn:obj}.
The quantity $\tilde n$ represents the remaining space in each machine after we evenly allocate $N$ data points onto $m$ machines. The inequality $\tilde n\geq cn$  means each machine still has $\Omega(n)$ memory space after such an allocation of data. This can happen when either the machine capacity $C$ or the number of machines $m$ is large enough.

We also assume that we can load the same function to multiple machines so that different machines may share some functions, so the sets of functions in all machines do not necessarily form a partition of $\{f_i\}_{i\in[N]}$. Since no machine can access all $N$ functions, we have to solve \eqref{eqn:obj} by distributed algorithms that alternate between a local computation procedure at each machine, and a round of communication to synchronize and share information among the machines.

\subsection{Communication efficiency and runtime}

To facilitate the theoretical study, we use the following simplified message passing model  from the distributed computation literature~\cite{gropp1996high,dean2008mapreduce}: We assume the communication occurs in rounds -- in each round, (a subset of) machines exchanges messages and, between two rounds, the machines only compute based on their local information (local data points and messages received before).


Given this state of affairs, we study the distributed optimization problem with three performance metrics in mind.
\begin{itemize}
\item \textbf{Local parallel runtime}: The longest running time of $m$ machines spent in local computation, measured in the number of gradient computations, i.e., computing $\nabla f_i(x)$ for any $i$. We also refer it as ``runtime'' for simplicity.
\item \textbf{The amount of communication}:  The total amount of communication among $m$ machines and the center, measured by the number of vectors\footnote{We will only consider communicating data points, or iterates $x$. For simplicity, we assume that the data point and iterates are of the same dimension, but this can be easily generalized.} of size $d$ transmitted.  
\item \textbf{Rounds of communication}: How many times all machines have to pause their local computation and exchange messages. We also refer it as ``rounds'' for simplicity.
\end{itemize}
We will study these performance metrics for the algorithms we propose and compare with other existing techniques. However, the  main focus of this paper is the rounds of communication.

\subsection{Summary of contributions}
\label{summary}

In this paper, we first propose a \emph{distributed stochastic variance reduced gradient} (DSVRG) method, which is simple and easy to implement -- it is essentially a distributed implementation of a well-known  single-machine stochastic variance reduced gradient (SVRG) method~\cite{JohnsonZhang13,XiaoZhang14,Konecny:15b}. We show that the proposed DSVRG algorithm requires $O((1+\frac{\kappa}{n})\log(1/\epsilon))$ rounds of communication to find an $\epsilon$-optimal solution for \eqref{eqn:obj} under Assumption \ref{assumption_C}. The corresponding parallel runtime is $O((n+\kappa)\log(1/\epsilon))$ and the associated amount of communication is $O((m+\frac{\kappa}{n})\log(1/\epsilon))$.

Given these performance metrics of DSVRG, we further ask a key question:

\textit{How can we achieve the optimal parallel runtime, the optimal amount of communication, and the optimal number of rounds of communication simultaneously for solving \eqref{eqn:obj}? }

This paper answers this seemingly ambitious question affirmatively in a reasonable situation: When $\kappa =\Theta( n^{1-2\delta})$ with a constant $0<\delta<\frac{1}{2}$, with an appropriate choices for the parameters in DSVRG (shown in Corollary
\ref{thm:constantrounds}), DSVRG finds an $\epsilon$-optimal solution for~\eqref{eqn:obj} with a parallel runtime of $O(n)$, an $O(m)$ amount of communication and $O(1)$ rounds of communication for any $\epsilon =\frac{1}{n^s}$ where $s$ is any positive constant. Here, the notation $O$ hides a logarithmic term of the optimality gap of an initial solution for DSVRG, which is considered as a constant in the whole paper.

We want to point out that $\kappa =\Theta( n^{1-2\delta})$ is a typical setting for machine learning applications. For example, as argued by~\cite{shamir2014distributed,shamir2014communication,shalev2009stochastic,zhang2015communication}, for ERM, the condition number $\kappa$ is typically in the order of $\Theta(\sqrt{N}) = \Theta(\sqrt{mn})$. Therefore, when the number of machines $m$ is not too large, e.g., when $m\le n^{0.8}$\footnote{If $n=10^5$, then $n^{0.8}= 10^4$ which is already much more than the number of machines in most clusters.}, we have that $\kappa = \Theta(\sqrt{mn})\le n^{0.9}$ (so that $\delta=0.05$). Moreover, $\epsilon = n^{-10}$ (so that $s=10$) is certainly a high enough accuracy for most machine learning applications since it exceeds the machine precision of real numbers, and typically people choose $\epsilon =\Theta(\frac{1}{N})$ in empirical risk minimization.


These performance guarantees of DSVRG, under the specific setting where $\kappa =O( n^{1-2\delta})$ and $\epsilon =O(\frac{1}{n^s})$, are optimal up to constant factors among all distributed first-order methods. First, to solve \eqref{eqn:obj}, all $m$ machines together need to compute at least $\Omega(N)$ gradients~\cite{argarwal14} in total so that each function in $\{f_i\}_{i=1,\dots,N}$ can be accessed at least once. Therefore, at least one machine needs to compute at least $\Omega(n)$ gradients in parallel given any possible allocation of functions. Second, the amount of communication is at least $\Omega(m)$ for even simple Gaussian mean estimation problems~\cite{BGMNW15}, which is a special case of \eqref{eqn:obj}. Third, at least $O(1)$ rounds of communication is needed to integrate the computation results from machines into a final output.

Furthermore, using the generic acceleration techniques developed in~\cite{frostigicml15} and~\cite{lin2015universal}, we propose a \emph{distributed accelerated stochastic variance reduced gradient} (DASVRG) method that further improves the theoretical performance of DSVRG. Under Assumption \ref{assumption_C}, we show that  DASVRG requires only $\tilde O((1+\sqrt{\frac{\kappa}{n}})\log(1/\epsilon))$ rounds of communication to find an $\epsilon$-optimal solution, leading to better theoretical performance than DSVRG. Also, we show that the runtime and the amount of communication for DASVRG are $\tilde O((n+\sqrt{n\kappa})\log(1/\epsilon))$ and $\tilde O(m+m\sqrt{\frac{\kappa}{n}})\log(1/\epsilon))$, respectively. We also prove a lower bound on the rounds of communication that shows any first-order distributed algorithm needs $\tilde \Omega( \sqrt{\frac{\kappa}{n}})\log(1/\epsilon))$ rounds of communication. It means DASVRG is optimal in that it uses the least number of rounds of communication. Since our lower bound indeed can be applied to a broad class of distributed first-order algorithms, it is interesting by itself. Here, and in the rest of the paper, $\tilde O$ and $\tilde \Omega$ hide some logarithmic terms of $\kappa$, $N$, $m$ and $n$.

The rest of this paper is organized as follows. In Section~\ref{sec:related}, we compared the theoretical performance of our methods with some existing work in distributed optimization.  In Section~\ref{sec:dsvrg} and Section~\ref{sec:dppasvrg}, we propose our DSVRG and DASVRG algorithms, respectively, and discuss their theoretical guarantee. In Section \ref{sec:lowerbound}, we prove a lower bound on the number of rounds of communication that a distributed algorithm needs, which demonstrates that DASVRG is optimal.  Finally, we present the numerical experiments in Section~\ref{sec:exp}, and conclude the paper in Section~\ref{sec:conclusion}.

\section{Related Work}
\label{sec:related}

Recently, there have been several distributed optimization algorithms proposed for problem~\eqref{eqn:obj}. We list several of them, including a distributed implementation of the accelerated gradient method (Accel Grad) by Nesterov~\cite{Nesterov04book}\footnote{This is the accelerated gradient method by Nesterov~\cite{Nesterov04book} except that the data points are distributed in $m$ machines to parallelize the computation of $\nabla f(x)$.}, in Table \ref{tab:dist-opt-algs} and present their rounds and runtime for a clear comparison. The algorithms proposed in this paper are DSVRG and DASVRG.

\begin{small}
\begin{table}
\centering
	\begin{tabular}{ |l | c | c | c|}
		\hline
		Algorithm & Rounds & Parallel Runtime & Assumptions \\
		\hline
    		DSVRG & $(1+\frac{\kappa}{n}) \log \frac{1}{\epsilon}$ & $G(n+\kappa) \log \frac{1}{\epsilon}$   & Assumption \ref{assumption_C}\\
		\hline
		DASVRG & $ ( 1+ \sqrt{\frac{\kappa}{n}})  \log ( 1+ \frac{\kappa}{n} ) \log \frac{1}{\epsilon}$ & $G (n+\sqrt{n\kappa})  \log ( 1+ \frac{\kappa}{n} ) \log \frac{1}{\epsilon}$ & Assumption \ref{assumption_C}\\
		\hline
		DISCO (quad) & $(1+\frac{\sqrt{\kappa}}{n^{.25}}) \log \frac{1}{\epsilon}$ &$Q(1+ \frac{\sqrt{\kappa}}{n^{.25}}) \log \frac{1}{\epsilon}$& \eqref{eqn:regemr}, $\xi_i \overset{iid}{\sim} D$\\
		\hline
		DISCO (non-quad) & $d^{.25} \left( (1+\frac{\sqrt{\kappa}}{n^{.25}}) \log \frac{1}{\epsilon} + \frac{\kappa^{1.5} }{n^{.75}} \right)$ &$Q d^{.25} \left( (1+\frac{\sqrt{\kappa}}{n^{.25}}) \log \frac{1}{\epsilon} + \frac{\kappa^{1.5} }{n^{.75}} \right)$ & \eqref{eqn:regemr}, $\xi_i \overset{iid}{\sim} D$\\
		\hline
		DANE (quad) & $(1+\frac{\kappa^2}{n}) \log \frac{1}{\epsilon}$ &$Q(1+\frac{\kappa^2}{n}) \log \frac{1}{\epsilon}$ & \eqref{eqn:regemr}, $\xi_i \overset{iid}{\sim} D$\\
		\hline
		CoCoA+  &$\kappa \log \frac{1}{\epsilon}$ & $G(n+\sqrt{\kappa n})\kappa \log \frac{1}{\epsilon}$& \eqref{eqn:regemr} \\
		\hline
		Accel Grad& $\sqrt{\kappa_f} \log \frac{1}{\epsilon}$& $Gn\sqrt{\kappa_f} \log \frac{1}{\epsilon}$& \\
		\hline

		
	\end{tabular}
	\caption{\textbf{Rounds and runtime of different distributed optimization algorithms in a general setting.}
Let $G$ be the computation cost of the gradient of $f_i$. For typical problems such as logistic regression, we have $G=O(d)$. Let $Q$ be the cost of solving a linear system. We have $Q = O(d^3)$ if exact matrix inversion is used and have $Q = O(dn\sqrt{\kappa} \log \frac{1}{\epsilon})$ if an $\epsilon$-approximate inverse is found by an accelerated gradient method~\cite{Nesterov04book}. } 
	\label{tab:dist-opt-algs}
\end{table}
\end{small}


The distributed dual coordinate ascent method, including DisDCA~\cite{Yang:13}, CoCoA~\cite{Jaggi14} and CoCoA+~\cite{ma2015adding}, is a class of distributed coordinate optimization algorithms which can be applied to the conjugate dual formulation of~\eqref{eqn:regemr}. In these methods, each machine only updates $n$ dual variables contained in a local problem defined on the $n$ local data points. Any optimization algorithm can be used as a subroutine in each machine as long as it reduces the optimality gap of the local problem by a constant factor. According to~\cite{ma2015adding,Jaggi14}, CoCoA+ requires $O(\kappa\log(1/\epsilon))$ rounds of communication to find an $\epsilon$-optimal solution\footnote{CoCoA+ has a better theoretical performance than CoCoA. According to~\cite{ma2015adding}, CoCoA+ is equivalent to DisDCA with ``practical updates''~\cite{Yang:13} under certain choices of parameters.}. If the accelerated SDCA method~\cite{SSZhang13SDCA,SSZhang13acclSDCA} is used as the subroutine in each machine, the total runtime for CoCoA+ is $O((n+\sqrt{\kappa n})\kappa \log (1/\epsilon))$. Therefore, both DSVRG and DASVRG have lower runtime and communication than CoCoA+, and the other distributed dual coordinate ascent variants.

Assuming the problem \eqref{eqn:obj} has the form of \eqref{eqn:regemr} with $\xi_i$'s i.i.d. sampled from a distribution $D$ (denoted by $\xi_i\overset{iid}{\sim} D$), the DANE \cite{shamir2014communication} and DISCO \cite{zhang2015communication} algorithms require $O((1+\frac{\sqrt{\kappa}}{n^{.25}}) \log (1/\epsilon))$ and $O((1+\frac{\kappa^2}{n})  \log (1/\epsilon))$ rounds of communication, respectively. Hence, DSVRG uses fewer rounds of communication than DANE and fewer than DISCO when $\kappa\leq n^{1.5}$. DASVRG always uses fewer rounds of communication than  DISCO and DANE. Note that, for these four algorithms, the rounds can be very small in the ``big data'' case of large $n$. Indeed, as $n$ increases, all four methods require only $O(\log \frac{1}{\epsilon})$ rounds which is independent of the condition number $\kappa$.

Moreover, DISCO and DANE have large running times due to solving a linear system each round, which is not practical for problems of large dimensionality. As an alternative, Zhang and Xiao~\cite{zhang2015communication} suggest solving the linear system with an inexact solution using another optimization algorithm, but this still has large runtime for ill-conditioned problems. The runtimes of DISCO and DANE are shown in Table~\ref{tab:dist-opt-algs} with $Q = O(d^3)$ to represent the time for taking matrix inverse and $Q = O(dn\sqrt{\kappa} \log \frac{1}{\epsilon})$  to represent the the time when an accelerated first-order method is used for solving the linear system. In both case, the runtimes of DISCO and DANE can be higher than those of DSVRG or DASVRG when $d$ is large. Furthermore, DANE only has the theoretical guarantee mentioned above when it is applied to quadratic problems, for example, regularized linear regression. Also, DISCO only applies to self-concordant functions with easily computed Hessian\footnote{The examples in \cite{zhang2015communication} all take the form of $f_i (x) = g( a_i^Tx)$ for some function $g$ on $\mathbb{R}^1$, which is more specific than $\phi(x,\xi_i)$. Under this form, it is relatively easy to compute the Hessian of $f_i$.}, and makes strong statistical assumptions on the data points.
On the contrary, DSVRG and DASVRG works for a more general problem \eqref{eqn:obj} and do assume  $\xi_i\overset{iid}{\sim} D$ for \eqref{eqn:regemr}.


We also make the connection to the recent lower bounds~\cite{Shamir15} for the rounds of communication needed by distributed optimization. Arjevani and Shamir~\cite{Shamir15} prove that, for a class of $\delta$-related functions (see~\cite{Shamir15} for the definition) and, for a class of algorithms, the rounds of communication achieved by DISCO is optimal. However, as mentioned above, DASVRG needs fewer rounds than DISCO. This is not a contradiction since DASVRG does not fall into the class of algorithms subject to the lower bound in~\cite{Shamir15}. In particular, the algorithms concerned by~\cite{Shamir15} can only use the $n$ local data points from the initial partition to update the local solutions while DASVRG samples and utilizes a second set of data points in each machine in addition to those $n$ data points.

Building on the work of \cite{Shamir15}, we prove a new lower bound showing that any distributed first-order algorithm requires $\tilde O(\sqrt{\frac{\kappa}{n}} \log \frac{1}{\epsilon})$ rounds. This lower bound combined with the convergence analysis of DASVRG shows that DASVRG is optimal in the number of rounds of communication.



In Table \ref{tab:dist-opt-algs-risk}, we compare the rounds and runtime of DSVRG, DASVRG and DISCO in the case where $\kappa = \Theta(\sqrt{N})=\Theta(\sqrt{mn})$, which is a typical setting for ERM problem as justified in \cite{zhang2015communication,shamir2014communication,shamir2014distributed,shalev2009stochastic}. We only compare our methods against DISCO, since it uses the fewest rounds of communication among other related algorithms. Let us consider the case where $n>m$, which is true in almost any reasonable distributed computing scenario. We can see from Table \ref{tab:dist-opt-algs-risk} that the rounds needed by DSVRG is lower than that of DISCO.  In fact, both DSVRG and DASVRG use $O(\log \frac{1}{\epsilon})$ rounds of communication, which is almost a constant for many practical machine learning applications\footnote{Typically $\epsilon \in (10^{-6}, 10^{-2})$, so $\log \frac{1}{\epsilon}$ is always less than $20$. }.
\begin{table}
	\centering
	\begin{tabular}{ |l | c | c | c|}
		\hline
		Algorithm & Rounds & Runtime & Assumptions \\
		\hline
		DSVRG & $(1+\sqrt{\frac{m}{n}}) \log \frac{1}{\epsilon}$ & $G(n+\sqrt{mn}) \log \frac{1}{\epsilon}$   & Assumption \ref{assumption_C} \\
        \hline
		DASVRG & $(1+(\frac{m}{n})^{.25}) \log(1+ \sqrt{\frac{m}{n}}) \log \frac{1}{\epsilon}$& $G (n+ n^{.75} m^{.25} ) \log(1+ \sqrt{\frac{m}{n}}) \log \frac{1}{\epsilon}$& Assumption \ref{assumption_C} \\
		\hline
		DISCO (quad) & $m^{.25} \log \frac{1}{\epsilon}$ &$Q m^{.25} \log \frac{1}{\epsilon}$& \eqref{eqn:regemr}, $\xi_i \overset{iid}{\sim} D$\\
		\hline

		
	\end{tabular}
	\caption{Rounds and runtime of DSVRG, DASVRG and DISCO when $\kappa = \Theta(\sqrt{N})$. The coefficients $Q$ and $G$ are defined as in Table~\ref{tab:dist-opt-algs}.} 
	\label{tab:dist-opt-algs-risk}
\end{table}

\section{Distributed SVRG}
\label{sec:dsvrg}

In this section, we consider a \emph{distributed stochastic variance reduced gradient} (DSVRG) method that is based on a parallelization of SVRG~\cite{JohnsonZhang13,XiaoZhang14,Konecny:15b}. SVRG works in multiple stages and, in each stage, one batch gradient is computed using all $N$ data points and $O(\kappa)$ iterative updates are performed with only one data point processed in each. Our distributed algorithm randomly partitions the $N$ data points onto $m$ machines with $n$ local data points on each to parallelize the computation of the batch gradient in SVRG. Then, we let the $m$ machines conduct the iterative update of SVRG in serial in a ``round-robin'' scheme, namely, let all machine stay idle except one machine that performs a certain steps of iterative updates of SVRG using its local data and pass the solution to the next machine. However, the only caveat in this idea is that the iterative update of SVRG requires an unbiased estimator of $\nabla f(x)$ which can be constructed by sampling over the whole data set. However, the unbiasedness will be lost if each machine can only sample over its local data. To address this issue, we use the remaining $\tilde n=C-n$ memory space of each machine to store a second set of data which is uniformly sampled from the whole data set before the algorithm starts. If each machine samples over this dataset, the unbiased estimator will be still available so that the convergence property can be inherited from the single-machine SVRG.

To load the second dataset mentioned above onto each machine, we design an efficient data allocation scheme which reuses the randomness of the first partitioned dataset to construct this second one. We show that this method helps to increase the overlap between the first and second dataset so that it requires smaller amount of communication than the direct implementation.

\subsection{An efficient data allocation procedure} \label{subsec:dataallocationplus}
To facilitate the presentation, we define a \emph{multi-set} as a collection of items where some items can be repeated. We allow taking the union of a regular set $S$ and a multi-set $R$, which is defined as a regular set $S\cup R$ consisting of the item in either $S$ or $R$ without repetition.

We assume that a random partition $S_1,\dots,S_m$ of $[N]$ can be constructed efficiently. A straightforward data allocation procedure is to prepare a partition $S_1,\dots,S_m$ of $[N]$ and then sample a sequence of $Q$ i.i.d. indices $r_1,\dots,r_Q$ uniformly with replacement from $[N]$. After partitioning $r_1,\dots,r_Q$ into $m$ multi-sets $R_1,\dots, R_m\subset[N]$, we allocate data $\{f_i\mid i\in S_j\cup R_j\}$ to machine $j$. Since $S_1,\dots,S_m$ has occupied $n$ of the memory in each machine, $Q$ can be at most $\tilde n m$. Note that the amount of communication in distributing $S_1,\dots,S_m$ is exactly $N$ which is necessary for almost all distributed algorithms. However, this straightforward procedure requires an extra $O(Q)$ amount of communication for distributing $R_j\backslash S_j$.


To improve the efficiency of data allocation, we propose a procedure which reuses the randomness of $S_1,\dots,S_m$ to generate the indices $r_1,\dots,r_Q$ so that the the overlap between $S_j$ and $R_j$ can be increased which helps reduce the additional amount of communication for distributing $R_1,\dots, R_m$. The key observation is that the concatenation of $S_1,\dots,S_m$ is a random permutation of $[N]$ which has already provided enough randomness needed by $R_1,\dots, R_m$. Hence, it will be easy to build the i.i.d. samples $r_1,\dots,r_Q$ by adding a little additional randomness on top of $S_1,\dots,S_m$. With this observation in mind, we propose our data allocation procedure in Algorithm \ref{alg:prepare}.

\setlength{\textfloatsep}{6pt}
\begin{algorithm}\caption{Data Allocation : $\dd(N,m,Q)$}\label{alg:prepare}
	\begin{algorithmic}[1]
		\REQUIRE Index set $[N]$, the number of machines $m$, and the length of target sequence $Q$.
		\ENSURE A random partition $S_1,\dots,S_m$ of $[N]$, indices $r_1,\dots,r_Q\in[N]$, multi-sets $R_1,\dots, R_m\subset [N]$, and data $\{f_i\mid i\in S_j\cup R_j\}$ stored on machine $j$ for each $j\in [m]$. \\
\vspace{0.1in}
		\noindent\textbf{Center samples $r_1,\dots,r_Q$ and $R_1,\dots,R_m$ as follows:}
				\vspace{0.1in}
            \STATE Randomly partition $[N]$ into $m$ disjoint sets $S_1,\dots,S_m$ of the same size $n=\frac{N}{m}$.
			\STATE Concatenate the subsets $S_1,\dots,S_m$ into a random permutation  $i_1,\dots,i_N$ of $[N]$ so that $S_j=\{i_{(j-1)n+1},\dots,i_{jn}\}$.
			\FOR{$\ell=1$ to $Q$}
				\STATE Let
        $$
        r_\ell=\left\{
        \begin{array}{ll}
        i_\ell&\text{ with probability }1-\frac{\ell-1}{N}\\
        i_{\ell'}&\text{ with probability }\frac{1}{N}\text{ for }\ell'=1,2,\dots,\ell-1.
        \end{array}\right.
        $$
			\ENDFOR
			\STATE Let \begin{eqnarray}
        \label{Randr}
        R_j=\left\{
        \begin{array}{ll}
        \{r_{(j-1)\tilde n+1},\dots,r_{j\tilde n}\}&\text{ if }j=1,\dots,\lceil Q/\tilde n\rceil-1\\
        \{r_{(\lceil Q/\tilde n\rceil-1)\tilde n+1},\dots,r_{Q}\}&\text{ if }j=\lceil Q/\tilde n\rceil\\
        \emptyset&\text{ if }j=\lceil Q/\tilde n\rceil+1,\dots,m.
        \end{array}\right.
        \end{eqnarray}\\
        		\vspace{0.1in}
			\textbf{Distribute data points to machines: } \\
			\vspace{0.1in}
			\STATE Machine $j$ acquires data points in $\{f_i|i\in S_j\cup R_j\}$ from the storage center.
	\end{algorithmic}
\end{algorithm}

The correctness and the expected amount of communication of Algorithm \ref{alg:prepare} are characterized as follows.
\begin{lemma}\label{lem:shuffling}
	The sequence $r_1,\dots,r_Q$ generated in Algorithm~\ref{alg:prepare} has the same joint distribution as a sequence of i.i.d. indices uniformly sampled with replacement from $[N]$.
Moreover, the expected amount of communication for distributing $\cup_{i=1}^m\{f_i|i\in R_j\backslash S_j\}$
is at most $\frac{Q^2}{N}$. 
\end{lemma}

\begin{proof}
	Conditioned on $i_1,\dots,i_{\ell-1}$ and $r_1,\dots,r_{\ell-1}$, the random index $i_{\ell}$ has uniform distribution over $[N]\setminus\{i_1,\dots,i_{\ell-1}\}$. Therefore, by Line 3 in Algorithm \ref{alg:prepare}, the conditional distribution of the random index $r_{\ell}$, conditioning on $i_1,\dots,i_{\ell-1}$ and $r_1,\dots,r_{\ell-1}$, is a uniform distribution over $[N]$. Hence, we complete the proof of the first claim of the lemma.
	
	To analyze the amount of communication, we note that $i_{\ell}\neq r_{\ell}$ with probability $\frac{\ell-1}{N}$. Suppose $i_{\ell}\in S_j$ for some $j$. We know that the data point $f_{r_{\ell}}$ needs to be transmitted to machine $j$ separately from $S_j$ only if $i_{\ell}\neq r_{\ell}$. Therefore, the expected amount of communication for distributing $\cup_{i=1}^m\{f_i|i\in R_j\backslash S_j\}$ is upper bounded by $\sum_{\ell=1}^{Q}\frac{\ell-1}{N} \le \frac{Q^2}{N}$.
\qed
\end{proof}

According to Lemma~\ref{lem:shuffling}, besides the (necessary) $N$ amount of communication to distribute $S_1,\dots,S_m$,  Algorithm \ref{alg:prepare} needs only $\frac{Q^2}{N}$ additional amount of communication to distribute $R_1,\dots,R_m$ thanks to the overlaps between $S_j$ and $R_j$ for each $j$. This additional amount is less than the $O(Q)$ amount required by the straightforward method when $Q\leq N$.

Note that, in the DSVRG algorithm we will introduce later, we need $Q=O(\kappa\log(1/\epsilon))$.  When $\kappa=\Theta(\sqrt{N})$ in a typical ERM problem, we have $\frac{Q^2}{N}=\frac{\kappa^2(\log(1/\epsilon))^2}{N}=O((\log(1/\epsilon))^2)$ which is typically much less than the $N$ amount of communication in distributing $S_1,\dots,S_m$. In other words, although DSVRG does require additional amount of communication to allocate the data than other algorithms, this additional amount is nearly negligible.

%

\subsection{DSVRG algorithm and its theoretical guarantees}\label{subsec:alg}

With the data $\{f_i\mid i\in S_j\cup R_j\}$ stored on machine $j$ for each $j\in [m]$ after running Algorithm~\ref{alg:prepare}, we are ready to present the distributed SVRG algorithm in Algorithm~\ref{alg:serialrr} and Algorithm~\ref{alg:ssvrg}.



We start SVRG in machine $k$ with $k=1$ initially at an initial solution $\tx^{0}\in\mathbb{R}^d$. At the beginning of stage $\ell$ of SVRG, all $m$ machines participate in computing a batch gradient $h^\ell$ in parallel using the data indexed by $S_1,\dots,S_m$. Within stage $\ell$, in each iteration, machine $k$ samples one data $f_i$ from its local data indexed by $R_k$ to construct a stochastic gradient $\nabla f_i(x_t)-\nabla f_i(\tx^{\ell})+h^{\ell}$ and performs the iterative update. Since $R_k$ is a multi-set that consists of indices sampled with replacement from $[N]$, the unbiasedness of $\nabla f_i(x_t)-\nabla f_i(\tx^{\ell})+h^{\ell}$, i.e., the property
\begin{eqnarray}
\label{eq:unbias}
\mathbb{E}_{i\sim R_k}[f_i(x_t)-\nabla f_i(\tx^{\ell})+h^{\ell}]=\mathbb{E}_{i\sim N}[f_i(x_t)-\nabla f_i(\tx^{\ell})+h^{\ell}]=\nabla f(x_t)
\end{eqnarray}
is guaranteed. After this iteration, $i$ is removed from $R_k$.

The $m$ machines do the iterative updates in the order from machine 1 to machine $m$. Once the current active machine, says machine $k$, has removed all of its samples in $R_k$ (so that $R_k=\emptyset$), then it must pass the current solution and the running average of all solutions generated in the current stage to machine $k+1$. At any time during the algorithm, there is only one machine updating the solution $x_t$ and the other $m-1$ machines only contribute in computing the batch gradient $h^\ell$. We want to emphasis that it is important that machines should never use any samples in $R_j$'s more than once since, otherwise, the stochastic gradient $\nabla f_i(x_t)-\nabla f_i(\tx^\ell)+h^{\ell}$  will lose its unbiasedness.  We describe formally each stage of this algorithm in Algorithm~\ref{alg:serialrr} and the iterative update in Algorithm~\ref{alg:ssvrg}.


\setlength{\textfloatsep}{6pt}

\begin{algorithm}
	\caption{Distributed SVRG (DSVRG)}
	\label{alg:serialrr}
	\begin{algorithmic}[1]
		\REQUIRE An initial solution $\tx^{0}\in\mathbb{R}^d$, data $\{f_i\}_{i=1,\dots,N}$, the number of machine $m$, a step length $\eta<\frac{1}{4L}$, the number of iterations $T$ in each stage, the number of stages $K$, and a sample size $Q=TK$.
		\ENSURE $\tx^{K}$
		\STATE Use $\dd$ to generate (a) a random partition $S_1,\dots,S_m$ of $[N]$, (b) $Q$ i.i.d. indices $r_1,\dots,r_Q$ uniformly sampled with replacement form $[N]$, (c) $m$ multi-sets $\mathcal{R}=\{R_1,\dots, R_m\}$ defined as \eqref{Randr}, and (d) data $\{f_i\mid i\in S_j\cup R_j\}$ stored on machine $j$.
		\STATE $k\leftarrow 1$
		\FOR{$\ell=0,1,2,\ldots,K-1$}
		\STATE Center sends $\tx^{\ell}$ to each machine
		\FOR{machine $j=1,2,\dots,m$ \textbf{in parallel}}
		\STATE Compute $h_j^{\ell}=\sum_{i\in S_j}\nabla f_i(\tx^{\ell})$ and send it to center
		\ENDFOR
		\STATE Center computes $h^{\ell}=\frac{1}{N}\sum_{j=1}^mh_j^{\ell}$ and send it to machine $k$
		\STATE $(\tx^{\ell+1},\mathcal{R},k)\leftarrow\text{SS-SVRG}(\tx^{\ell},h,\mathcal{R},k,\eta,T)$
		\ENDFOR
	\end{algorithmic}
\end{algorithm}

\begin{algorithm}
\caption{Single-Stage SVRG: $\text{SS-SVRG}(\tx,h,\mathcal{R},k,\eta,T, \{f_i\}_{i=1,\dots,N})$}
\label{alg:ssvrg}
\begin{algorithmic}[1]
\REQUIRE A solution $\tx\in\mathbb{R}^d$,  data $\{f_i\}_{i=1,\dots,N}$, a batch gradient $h$, $m$ multi-sets $\mathcal{R}=\{R_1,R_2,\dots,R_m\}$,
the index of the active machine $k$, a step length $\eta<\frac{1}{4L}$, and the number of iterations $T$.
\ENSURE The average solution $\bx_T$, the updated multi-sets $\mathcal{R}=\{R_1,R_2,\dots,R_m\}$, and the updated index of active machine  $k$.
\STATE $x_0=\tx$ and $\bx_0=\mathbf{0}$
\FOR{$t=0,1,2,\ldots,T-1$}
\STATE Machine $k$ samples an instance $i$ from $R_k$ and computes
\begin{eqnarray*}
\label{eq:svrgstep}
x_{t+1}=x_t-\eta\left(\nabla f_i(x_t)-\nabla f_i(\tx)+h\right),\quad
\bx_{t+1}=\frac{x_{t+1}+t\bx_t}{t+1},\quad
R_k\leftarrow R_k\backslash\{i\}
\end{eqnarray*}
\IF{$R_k=\emptyset$}
\STATE $x_{t+1}$ and $\bx_{t+1}$ are sent to machine $k+1$
\STATE $k\leftarrow k+1$
\ENDIF
\ENDFOR
\end{algorithmic}
\end{algorithm}


Note that an implicit requirement of Algorithm \ref{alg:serialrr} is
$TK=Q\leq\tilde n m$.
Because one element in $R_k$ is removed in each \emph{iterative update} (Line 3) of Algorithm~\ref{alg:ssvrg}, this update cannot be performed any longer once each $R_k$ becomes empty. Hence, the condition $TK=Q$ ensures that the number of iterative updates matches the total number of indices contained in $R_1,R_2,\dots,R_m$. Note that, in the worse case, $S_j$ and $R_j$ may not overlap so that the size of data $\{f_i\mid i\in S_j\cup R_j\}$ stored on machine $j$ can be $|S_j|+|R_j|$. That is why the size of $R_j$ is only $\tilde n=C-n$. The condition $TK=Q\leq\tilde n m$ is needed here so the total amount of data stored in all machines do not exceed the total capacity $Cm$. The convergence of Algorithm~\ref{alg:serialrr} is established by the following theorem.

\begin{theorem}
\label{thm:svrgrr2}
Suppose $0<\eta<\frac{1}{4L}$ and $TK=Q\leq\tilde n m$.
 Algorithm~\ref{alg:serialrr} guarantees 
 
\begin{eqnarray}
\label{eq:svrg_conv}
\mathbb{E}\left[f(\tx^{K})-f(x^*)\right]
\leq\left(\frac{1}{\mu\eta(1-4L\eta)T}+\frac{4L\eta(T+1)}{(1-4L\eta)T}\right)^{K}\left[f(\tx^0)-f(x^*)\right].
\end{eqnarray}
\normalsize
In particular, when $\eta=\frac{1}{16L}$, $T=96\kappa$ and $TK= Q\leq\tilde n m$, Algorithm~\ref{alg:serialrr}
needs
$
K= \frac{1}{\log(9/8)}\log\left(\frac{f(\tilde x^0) -f(x^*)}{\epsilon}\right)
$
stages to find an $\epsilon$-optimal solution.
\end{theorem}

\begin{proof}
In the iterative update given in Line 3 of Algorithm~\ref{alg:ssvrg}, a stochastic gradient $\nabla f_i(x_t)-\nabla f_i(\tx)+h$ is constructed with $h$ being the batch gradient $\nabla f(\tx)$ and $i$ sampled from $R_k$ in the active machine $k$. Since $i$ is one of the indices $r_1,\dots,r_Q$, each of which is sampled uniformly from $[N]$, this stochastic gradient is unbiased estimator of $\nabla f(x_t)$. Therefore, the path of solutions $\tx^0,\tx^1,\tx^2,\dots$ generated by Algorithm~\ref{alg:serialrr} has the same distribution as the ones generated by single-machine SVRG so that the convergence result for the single-machine SVRG can be directly applied to Algorithm~\ref{alg:serialrr}. 
The inequality \eqref{eq:svrg_conv} has been shown in Theorem 1 in \cite{XiaoZhang14} for single-machine SVRG, which now also holds for Algorithm~\ref{alg:serialrr}.

When $\eta=\frac{1}{16L}$ and $T=96\kappa$, it is easy to show that

\begin{eqnarray}
\label{eq:8over9}
\frac{1}{\mu\eta(1-4L\eta)T}+\frac{4L\eta(T+1)}{(1-4L\eta)T}
\leq\frac{1}{\mu\eta(1-4L\eta)T}+\frac{8L\eta}{(1-4L\eta)}
=\frac{2}{9}+\frac{2}{3}=\frac{8}{9}
\end{eqnarray}
\normalsize
so that Algorithm~\ref{alg:serialrr}
needs
$
K= \frac{1}{\log(9/8)}\log\left(\frac{f(\tilde x^0) -f(x^*)}{\epsilon}\right)
$
stages to find an $\epsilon$-optimal solution.
\qed
\end{proof}

By Theorem~\ref{thm:svrgrr2}, DSVRG can find an $\epsilon$-optimal solution for~\eqref{eqn:obj} after $K=O(\log(1/\epsilon))$ stages with $T=O(\kappa)$ iterative updates (Line 3 of Algorithm~\ref{alg:ssvrg}) in each stage.
Therefore, there are $O(\kappa \log(1/\epsilon))$ iterative updates in total so that $Q$ must be $TK=O(\kappa \log(1/\epsilon))$.
Since the available memory space requires $Q\leq\tilde n m=(C-n)m$. We will need at least $C=\Omega(n+\frac{\kappa}{m} \log(1/\epsilon))$ in order to implement DSVRG.

Under the assumptions of Theorem~\ref{thm:svrgrr2}, including Assumption~\ref{assumption_C} (so $\tilde n=\Omega(n)$), we discuss the theoretical performance of DSVRG as follows.
\begin{itemize}
\item \textbf{Local parallel runtime}: Since one gradient is computed in each iterative update and $n$ gradients are computed in parallel to construct the batch gradient, the total local parallel runtime for DSVRG to find an $\epsilon$-optimal solution is $O((n+T )K)=O((n+\kappa) \log(1/\epsilon))$
\item \textbf{Communication}: There is a fixed amount of $O(N)$ communication needed to distribute the partitioned data $S_1,\dots,S_m$ to $m$ machines. When Algorithm \ref{alg:prepare} is used to generate $R_1,\dots,R_m$, the additional amount of communication to complete the data allocation step of DSVRG is $O(Q^2/N) = O(\kappa^2\log^2(1/\epsilon)/N)$ in expectation according to Lemma \ref{lem:shuffling}. During the algorithm, the batch gradient computations and the iterative updates together require $O((m+\frac{T}{\tilde n})K)=O((m+\frac{\kappa}{n})\log(1/\epsilon))$ amount of communication.
\item \textbf{Rounds of communication}: Since DSVRG needs one round of communication to compute a batch gradient at each stage (one call of $\text{SS-SVRG}$) and one round after every $\tilde n$ iterative update (when $R_k=\emptyset$), it needs $O(K+\frac{TK}{\tilde n})=O((1+\frac{\kappa}{n})\log(1/\epsilon))$ rounds of communication in total to find an $\epsilon$-optimal solution.  We note that, if the memory space $C$ in each machine is large enough, the value of $\tilde n$ can larger than $\kappa$ so that the rounds of communication needed will be only $O(\log(1/\epsilon))$.
\end{itemize}



\subsection{Regimes where DSVRG is Optimal}
\label{sec:DSVRGopt}

In this subsection, we consider a scenario where $\kappa =\Theta( n^{1-2\delta})$ with a constant $0<\delta<\frac{1}{2}$ and $\epsilon =\frac{1}{n^s}$ with positive constant $s$. We show that with a different choices of $\eta$, $T$ and $K$, DSVRG can find an $\epsilon$-optimal solution solution for~\eqref{eqn:obj} with the optimal parallel runtime, the optimal amount of communication, and the optimal number of rounds of communication simultaneously.
\begin{proposition}
\label{thm:constantrounds} Suppose $\kappa\leq \frac{n^{1-2\delta}}{32} $ with a constant $0<\delta<\frac{1}{2}$ and we choose $\eta=\frac{1}{16n^\delta L}$, $T=n$ and $K= \frac{1}{\log(n^\delta/2)}\log\left(\frac{f(\tilde x^0) -f(x^*)}{\epsilon}\right)$ in Algorithm~\ref{alg:serialrr}. Also, suppose Assumption~\ref{assumption_C} holds and $TK=Q\leq\tilde n m$. Algorithm~\ref{alg:serialrr} finds $\epsilon$-optimal solution for~\eqref{eqn:obj} with $O(\frac{\log(1/\epsilon)}{\delta \log n})$ rounds of communications, $O(\frac{n\log(1/\epsilon)}{\delta \log n})$ total parallel runtime, and $O(\frac{m\log(1/\epsilon)}{\delta \log n})$ amount of communication.

In particular, when $\epsilon =\frac{1}{n^s}$ with a positive constant $s$, Algorithm~\ref{alg:serialrr} finds $\epsilon$-optimal solution for~\eqref{eqn:obj} with $O(1)$ rounds of communications, $O(n)$ total parallel runtime, and $O(m)$ amount of communication.
\end{proposition}

\begin{proof}
Since  $\eta=\frac{1}{16n^\delta L}<\frac{1}{4L}$, Algorithm~\ref{alg:serialrr} guarantees \eqref{eq:svrg_conv} according to  Theorem \ref{thm:svrgrr2}. With $T=n$ and $\eta=\frac{1}{16n^\delta L}$, we have

\begin{eqnarray*}
\frac{1}{\mu\eta(1-4L\eta)T}+\frac{4L\eta(T+1)}{(1-4L\eta)T}
&\leq&\frac{1}{\mu\eta(1-4L\eta)T}+\frac{8L\eta}{(1-4L\eta)}
=\frac{16n^\delta\kappa}{(1-1/(4n^\delta))n}+\frac{1}{2n^\delta(1-1/(4n^\delta))}\\
&\leq&\frac{1}{2n^\delta(1-1/(4n^\delta))}+\frac{1}{2n^\delta(1-1/(4n^\delta))}\leq\frac{2}{n^\delta}\\
\end{eqnarray*}
\normalsize
Hence, $\mathbb{E}\left[f(\tx^{K})-f(x^*)\right]\leq\epsilon$ can be implied by the inequality \eqref{eq:svrg_conv} as $K= \frac{1}{\log(n^\delta/2)}\log\left(\frac{f(\tilde x^0) -f(x^*)}{\epsilon}\right)$.

Under Assumption~\ref{assumption_C}, we have $\tilde n=\Omega(n)$ so that
the number of rounds of communication  Algorithm~\ref{alg:serialrr} needs is
$O(K+\frac{TK}{\tilde n})\leq O(K+\frac{TK}{n})=O(\frac{\log(1/\epsilon)}{\delta \log n})$. Moreover, the total parallel needed is $O((n+T)K)=O(\frac{n\log(1/\epsilon)}{\delta \log n})$ and the amount of communication needed is
$O((m+\frac{T}{\tilde n})K)=O(\frac{m\log(1/\epsilon)}{\delta \log n})$.
The second conclusion can be easily derived by replacing $\epsilon$ with $O(\frac{1}{n^s})$.
\qed
\end{proof}

The justification for the scenario where $\kappa =\Theta( n^{1-2\delta})$ and $\epsilon =\frac{1}{n^s}$ and why these performance guarantees are optimal have been discussed in Section~\ref{summary}.


\section{Accelerated Distributed SVRG}\label{sec:dppasvrg}
In this section, we use the generic acceleration techniques in~\cite{frostigicml15} and~\cite{lin2015universal} to further improve the theoretical performance of DSVRG and obtain a \emph{distributed accelerated stochastic variance reduced gradient} (DASVRG) method. 

\subsection{DASVRG algorithm and its theoretical guarantees}

Following \cite{frostigicml15} and~\cite{lin2015universal}, we define a \emph{proximal function} for $f(x)$ as

\begin{eqnarray}
\label{eq:proxobj}
f_\sigma(x;y) \equiv f(x) + \frac{\sigma}{2} \norm{x-y}^2=\frac{1}{N}\sum_{i=1}^N\tilde f_i(x;y),
\end{eqnarray}
\normalsize
where $\tilde f_i(x;y)=f_i(x) + \frac{\sigma}{2} \norm{x-y}^2$, $\sigma\geq0$ is a constant to be determined later and $y\in\mathbb{R}^d$ is a \emph{proximal point}. The condition number of this proximal function is
$
\kappa( f_\sigma) \equiv \frac{ L +\sigma}{ \mu +\sigma }
$
which can be smaller than $\kappa$ when $\sigma$ is large enough.

Given an algorithm, denoted by $\mathcal{A}$, that can be applied to \eqref{eqn:obj}, the acceleration scheme developed in \cite{frostigicml15} and~\cite{lin2015universal} is an iterative method that involves inner and outer loops and uses $\mathcal{A}$ as a sub-routine in its outer loops. In particular, in $p$-th outer iteration of this acceleration scheme, the algorithm $\mathcal{A}$ is applied to find a solution for the $p$-th proximal problem defined on a proximal point $y_{p-1}$, namely,

\begin{eqnarray}
\label{eq:proxproblem}
f^*_p\equiv \min_{x\in\mathbb{R}^d}f_\sigma(x;y_{p-1})\text{ for }p=1,2,\dots,P.
\end{eqnarray}
\normalsize

The algorithm $\mathcal{A}$ does not need to solve \eqref{eq:proxproblem} to optimality but only needs to generate an approximate solution $\hat x_p$ with an accuracy $\epsilon_p$ in the sense that

\begin{eqnarray}
\label{eq:proxaccuracy}
f_\sigma(\hat x_p;y_{p-1})-f^*_p\leq\epsilon_p.
\end{eqnarray}
\normalsize
When $\kappa( f_\sigma)$ is smaller than $\kappa$, finding such an $\hat x_p$ is easier than finding an $\epsilon$-optimal solution for \eqref{eqn:obj}.

Then, the acceleration scheme uses $\hat x_p$ to construct a new proximal point $y_{p}$ using an extrapolation update as $y_{p}=\hat x_p+\beta_p(\hat x_p-\hat x_{p-1})$, where $\beta_p\geq0$ is an extrapolation step length. After that,
the $p+1$-th proximal problem is constructed based on $y_p$ which will be solved in the next outer iteration. With an appropriately chosen value for $\sigma$, it is shown by \cite{frostigicml15} and~\cite{lin2015universal} that, for many existing $\mathcal{A}$ including SAG~\cite{schmidt2013minimizing,LeRouxSchmidtBach12}, SAGA~\cite{defazio2014saga}, SDCA~\cite{SSZhang13SDCA}, SVRG~\cite{JohnsonZhang13} and Finito/MISO~\cite{FINITO,MISO:15}, this acceleration scheme needs a smaller runtime for finding an $\epsilon$-optimal solution than applying algorithm $\mathcal{A}$ directly to \eqref{eqn:obj}.



Given the success of this acceleration scheme in the single-machine setting, it will be promising to also apply this scheme to the DSVRG to further improve its theoretical performance. Indeed, this can be done by choosing $\mathcal{A}$ in this aforementioned acceleration scheme to be DSVRG. Then, we can obtain the DASVRG algorithm. In particular, in the $p$-th outer iteration of the aforementioned acceleration scheme, we use DSVRG to solve the proximal problem \eqref{eq:proxproblem} in a distributed way up to an accuracy $\epsilon_p$. We present DASVRG in Algorithm \ref{alg:dppasvrg} where $\tilde f_i(x;y)=f_i(x) + \frac{\sigma}{2} \norm{x-y}^2$.

\begin{algorithm}
	\caption{Distributed Accelerated SVRG (DASVRG)}
	\label{alg:dppasvrg}
	\begin{algorithmic}[1]
		\REQUIRE An initial solution $\hat x_0\in\mathbb{R}^d$, data $\{f_i\}_{i=1,\dots,N}$, the number of machine $m$, a step length $\eta<\frac{1}{4L}$, the number of iterations $T$ in each stage of DSVRG, the number of stages $K$ of DSVRG, the number of outer iterations $P$ in the acceleration scheme, a sample size $TKP=Q\leq\tilde n m$, and a parameter $\sigma\geq0$.
		\ENSURE $\hat x_P$
        \STATE Use  $\dd$ to generate (a) a random partition $S_1,\dots,S_m$ of $[N]$, (b) $Q$ of i.i.d. indices $r_1,\dots,r_Q$ uniformly sampled with replacement form $[N]$, (c) $m$ multi-sets $\mathcal{R}=\{R_1,\dots, R_m\}$ defined as \eqref{Randr}, and (d) data $\{f_i\mid i\in S_j\cup R_j\}$ stored on machine $j$. 
		\STATE $k\leftarrow 1$
        \STATE Initialize $q=\frac{\mu}{\mu+\sigma}$, $y_0 = \hat x_0$ and $\alpha_0 = \sqrt{q}$
		\FOR{$p=1,2,\ldots,P$}
        \STATE Center computes $\tx^{0}=\hat x_{p-1}$ and sends $y_{p-1}$ to each machine
		\FOR{$\ell=0,1,2,\ldots,K-1$}
		\STATE Center sends $\tx^{\ell}$ to each machine
		\FOR{machine $j=1,2,\dots,m$ \textbf{in parallel}}
		\STATE Compute $h_j^{\ell}=\sum_{i\in S_j}\nabla \tilde f_i(\tx^{\ell};y_{p-1})$ and send it to center
		\ENDFOR
		\STATE Center computes $h^{\ell}=\frac{1}{N}\sum_{j=1}^mh_j^{\ell}$ and sends it to machine $k$
		\STATE $(\tx^{\ell+1},\mathcal{R},k)\leftarrow\text{SS-SVRG}(\tx^{\ell},h^{\ell},\mathcal{R},k,\eta,T,\{\tilde f_i(x;y_{p-1})\}_{i=1,\dots,N})$
		\ENDFOR
        \STATE Machine $k$ computes $\hat x_p=\tx^{K}$ and sends $\hat x_p$ to center
\STATE Center computes $\alpha_p \in(0,1) $ from the equation $ \alpha_p ^2 = (1-\alpha_p ) \alpha_{p-1}^2 + q\alpha_p$.
		\STATE Center computes $y_p =\hat x_p + \beta_p (\hat x_p - \hat x_{p-1})$ , where $\beta_p= \frac{\alpha_{p-1} ( 1- \alpha_{p-1}) }{ \alpha_{p-1}^2 + \alpha_p } $.
		\ENDFOR
	\end{algorithmic}
\end{algorithm}


\begin{proposition}\label{thm:DASVRG}
Suppose $\eta=\frac{1}{16L}$, $T=96\kappa( f_\sigma)$,

\begin{eqnarray}
\label{eq:asvrg_K}
K= \frac{1}{\log(9/8)}\log\left(\frac{4}{2-\sqrt{q}}+\frac{10368\sigma}{\mu q(1-\frac{\sqrt{q}}{2})^2}\right)
=O\left(\log\left(1+\frac{\sigma}{\mu}\right)\right).
\end{eqnarray}
\normalsize
and $TKP=Q\leq\tilde n m$. The solution $\hat x_p$ generated in Algorithm \ref{alg:dppasvrg} satisfies \eqref{eq:proxaccuracy} with

\begin{eqnarray}
\label{eq:prox_epsilon_p}
\epsilon_p=\frac{2}{9}[f(\hat x_0)-f(x^*)]\left(1-\frac{\sqrt{q}}{2}\right)^p, \text{ for }p=1,2,\dots,P,
\end{eqnarray}
\normalsize
Moreover, Algorithm \ref{alg:dppasvrg} finds an $\epsilon$-optimal solution for \eqref{eqn:obj} after
$P=O\left(\sqrt{\frac{\mu+\sigma}{\mu}}\log(1/\epsilon)\right)$outer iterations.
\end{proposition}
\begin{proof}Due to the unbiasedness of $\nabla f_i(x_t)-\nabla f_i(\tx^{\ell})+h^{\ell}$~\eqref{eq:unbias}, conditioning on $\hat x_{p-1}$, the solution path $\tx^0,\tx^1,\tx^2,\dots$ generated within the $p$-th outer loop of Algorithm~\ref{alg:dppasvrg} has same distribution as the solution path generated by applying single-machine SVRG to \eqref{eq:proxproblem} with an initial solution of $\hat x_{p-1}$. Hence, all the convergence results of SVRG can be applied.


According to \eqref{eq:8over9} in the proof of Theorem \ref{thm:svrgrr2}, the choices of $\eta$ and $T$ ensure
$
\mathbb{E}\left[f_\sigma(\hat x_p;y_{p-1})-f^*_p\right] \leq\left(\frac{8}{9}\right)^{K}\left[f_\sigma(\hat x_{p-1};y_{p-1})-f^*_p\right].
$
Using this result and following the analysis in Section B.2 in~\cite{lin2015universal}, we can show that $\hat x_p$ satisfies \eqref{eq:proxaccuracy} with $\epsilon_p$ given by \eqref{eq:prox_epsilon_p} if $K$ is set to \eqref{eq:asvrg_K}.

Therefore, according to Theorem 3.1 in~\cite{lin2015universal} with $\rho=\frac{\sqrt{q}}{2}$, Algorithm~\ref{alg:dppasvrg} guarantees that
$\mathbb{E}\left[f(\hat x_P)-f(x^*)\right]
\leq\frac{32}{q}\left(1-\frac{\sqrt{q}}{2}\right)^{P+1}\left[f(\hat x_0)-f(x^*)\right]$
by choosing $K$ as \ref{eq:asvrg_K}.
This means
Algorithm~\ref{alg:dppasvrg} finds an $\epsilon$-optimal solution for \eqref{eqn:obj} after $P= O(\frac{1}{\sqrt{q}}\log(1/\epsilon))=O\left(\sqrt{\frac{\mu+\sigma}{\mu}}\log(1/\epsilon)\right)$ outer loops.

The condition $TKP=Q\leq\tilde n m$ here is only to guarantee that we have enough samples in $R_1,R_2,\dots,R_m$ to finish a total of $TKP$ iterative updates.
\qed
\end{proof}

We will choose $\sigma=\frac{L}{n}$ in DASVRG and obtain the following theorem.

\begin{theorem}
\label{thm:dasvrg}
Suppose $\eta=\frac{1}{16L}$, $T=96\kappa( f_\sigma)$, $K$ is chosen as \eqref{eq:asvrg_K}, $TKP=Q\leq\tilde n m$ and $\sigma =\frac{L}{n}$. Algorithm \ref{alg:dppasvrg} finds an $\epsilon$-optimal solution for \eqref{eqn:obj} after
$O((1+\sqrt{\frac{\kappa}{n}})\log(1+\frac{\kappa}{n})\log(1/\epsilon))$ calls of SS-SVRG.
\end{theorem}

\begin{proof}
When $\sigma =\frac{L}{n}$, according to Proposition \ref{thm:DASVRG}, DASVRG needs $P=O(\sqrt{\frac{\mu+\sigma}{\mu}}\log(1/\epsilon))=O((1+\sqrt{\frac{\kappa}{n}})\log(1/\epsilon))$ outer iterations with $K=O(\log(1+\frac{\sigma}{\mu}))=O(\log(1+\frac{\kappa}{n}))$ calls of SS-SVRG in each according to \eqref{eq:asvrg_K}. Hence, a total of $KP=O((1+\sqrt{\frac{\kappa}{ n}})\log(1+\frac{\kappa}{ n})\log(1/\epsilon))$ calls of SS-SVRG are needed.
\qed
\end{proof}

When $\sigma =\frac{L}{n}$, we have
$
\kappa( f_\sigma)= \frac{ L +\sigma}{ \mu +\sigma }= \frac{ n L +L}{ n\mu +L }\leq n+1.
$
According to Theorem~\ref{thm:dasvrg}, DASVRG can find an $\epsilon$-optimal solution for~\eqref{eqn:obj} after $\tilde O((1+\sqrt{\frac{\kappa}{n}})\log(1/\epsilon))$ calls of SS-SVRG. Note that each call of SS-SVRG involves $T=O(\kappa( f_\sigma))=O(n)$ iterative updates. Therefore, there are $Q=TKP=\tilde O(( n+\sqrt{ n\kappa })\log(1/\epsilon))$ iterative updates in total. Since the available memory space requires $Q\leq\tilde n m=(C-n)m$, we can derive from the inequality
$
\tilde O((n+\sqrt{n\kappa })\log(1/\epsilon))\leq (C-n)m
$
that we need at least $C=\tilde \Omega(n+(\frac{n}{m}+\frac{\sqrt{n\kappa}}{m})\log(1/\epsilon))$ in order to implement DASVRG.

Under the assumptions of Theorem~\ref{thm:dasvrg}, including Assumption~\ref{assumption_C} (so $\tilde n=\Omega(n)$), we summarize the theoretical performance of DASVRG as follows.
\begin{itemize}
\item \textbf{Local parallel runtime}: Since each call of SS-SVRG involves a batch gradient computation and $T=O(n)$ iterative update,  the total runtime of DASVRG is $O((n+T)KP)=\tilde O(n+\sqrt{n\kappa})\log(1/\epsilon))$.
\item \textbf{The amount of communication}: Similar to DSVRG, we need a fixed amount $O(N)$ communication to distribute $S_1,\dots,S_m$ and an additional amount of $O(\kappa^2\log^2(1/\epsilon)/N) $ communication to distribute $R_1,\dots,R_m$ to $m$ machines. During the algorithm, the batch gradient computations and the iterative updates together require $O((m+\frac{T}{\tilde n})KP)=\tilde O((m+\frac{T}{\tilde n})(1+\sqrt{\frac{\kappa}{n}})\log(1/\epsilon))=\tilde O((m+m\sqrt{\frac{\kappa}{ n}})\log(1/\epsilon))$ amounts of communication.
\item \textbf{The number of rounds of communication}: Since each call of SS-SVRG needs $1+\frac{T}{\tilde n}$ rounds communication,  the rounds  of DASVRG is $O((1+\frac{T}{\tilde n})KP)=\tilde O((1+\sqrt{\frac{\kappa}{n}})\log(1/\epsilon))$.
\end{itemize}

Recall that DSVRG needs $O((1+\frac{\kappa}{n})\log(1/\epsilon))$ rounds of communication which is more than DASVRG. The crucial observation here is that, although in the single-machine setting, the acceleration scheme of~\cite{frostigicml15,lin2015universal} only helps when $\kappa \geq N$, in the distributed setting, it helps to reduce the rounds  as long as $\kappa$ is larger than the number of local samples $n$.


\section{Lower Bounds on Rounds of Communication}\label{sec:lowerbound}

In this section, we prove that, under Assumption~\ref{assumption_C}, any distributed first-order method will require at least $\tilde \Omega(\sqrt{\frac{\kappa}{n}}\log(1/\epsilon))$ rounds to find an $\epsilon$-optimal solution for \eqref{eqn:obj} with both partitioned data and i.i.d. sampled data in each machine. This lower bound is matched by the upper bound of the rounds needed by DASVRG in Section~\ref{sec:dppasvrg} up to some logarithmic terms.  We note that we are working under different scenarios than~\cite{Shamir15}: In~\cite{Shamir15}, the authors assumed that the only property of the data that an algorithm can exploit is that the local sums are $\delta$-related for a $\delta \approx 1/\sqrt{n}$, and proved that the number of rounds is at least $\Omega(\sqrt{\kappa}/n^{1/4})$. The DASVRG algorithm exploits the fact that data is randomly partitioned and outperforms their lower bound. This suggests that $\delta$-relatedness shouldn't be the only property that an algorithm exploits, and motivates us to prove a new (matching) lower bound by assuming the data is randomly partitioned.


\subsection{A lower bound for rounds of communication}

We first consider a family of algorithms which consist of a data distribution stage where the functions $\{f_i\}_{i\in[N]}$ are distributed onto $m$ machines, and a distributed computation stage where, in each round, machines can not only use first-order (gradient) information of the functions stored locally but also apply preconditioning using local second-order information (Hessian matrix).

\begin{definition}[Distributed (extended) first-order algorithms $\mathcal{F}_{\alpha}$]\label{def:dfo}We say an algorithm $\mathcal{A}$ for solving \eqref{eqn:obj} with $m$ machines belongs to the family $\mathcal{F}_{\alpha}$ ($\mathcal{A}\in\mathcal{F}_{\alpha}$) of distributed first-order algorithms if it distributes $\{f_i\}_{i\in[N]}$ to $m$ machines only once at the beginning such that:
	\begin{enumerate}
		\item The index set $[N]$ is randomly and evenly partitioned into $S_1,S_2,\dots,S_m$ with $|S_j|=n$ for $j=1,\dots,m$.
		\item A multi-set $R_j$ of size $\alpha n$ is created by sampling with replacement from $[N]$ for $j=1,\dots,m$, where  $\alpha\geq0$ is a constant.
        \item Let $S'_j=S_j\cup R_j$. Machine $j$ acquires functions in $\{f_i|i\in S'_j\}$ for $j=1,\dots,m$.
	\end{enumerate}
and let the machines do the following operations in rounds:
	\begin{enumerate}
		\item Machine $j$ maintains a local set of vectors $W_j$ initialized to be $W_j = \{\mathbf{0}\}$.
		\item In each round, for arbitrarily many times, machine $j$ can add any $w$ to $W_j$ if $w$ satisfies (c.f.~\cite{Shamir15})

	\begin{align}
	&\gamma w + \nu \nabla F_j(w) \in \textrm{span}\left\{w', \nabla F_j(w'), (\nabla^2 F_j(w')+D)w'', (\nabla^2 F_j(w')+D)^{-1}w''\mid\right.\nonumber\\
	&\left. w',w''\in W_j, D \textrm{ diagonal}, \nabla^2 F_j(w') \textrm{ and } (\nabla^2 F_j(w')+D)^{-1} \textrm{ exists}\right\} \label{eqn:udpate-rule}
	\end{align}
	for some $\gamma,\mu$ such that $\gamma\nu\neq  0$, where $F_j=\sum_{i\in U_j\subset S'_j} f_i$ with an arbitrary subset $U_j$ of $S'_j$.
		\item At the end of the round, all machines can simultaneously send any vectors in $W_j$ to any other machines, and machines can add the vectors received from other machines to its local working set.
		\item The final output is a vector in the linear span of one $W_j$.
	\end{enumerate}
We define $\mathcal{A}(\{f_i\}_{i\in[N]},H)$ as the output vector of $\mathcal{A}$ when it is applied to \eqref{eqn:obj} for $H$ rounds with the inputs $\{f_i\}_{i\in[N]}$.
\end{definition}

Besides the randomness due to the data distribution stage, the algorithm $\mathcal{A}$ itself can be a randomized algorithm. Hence, the output $\mathcal{A}(\{f_i\}_{i\in[N]},H)$ can be a random variable.

We would like to point out that, although the algorithms in $\mathcal{F}_{\alpha}$ can use the local second-order information like $\nabla^2 f_i(x)$ in each machine, Newton's method is still not contained in $\mathcal{F}_{\alpha}$ since Newton's method requires the access to the global second-order information such as $\nabla^2 f(x)$ (machines are not allowed to share matrices with each other). That being said, one can still use a distributed iteration method which can multiply $\nabla^2 f_i(x)$ to a local vector in order to solve the inversion of $\nabla^2 f(x)$ approximately. This method will lead to a distributed inexact Newton method such as DISCO \cite{zhang2015communication}. In fact, both DANE \cite{shamir2014communication} and DISCO \cite{zhang2015communication} belong to $\mathcal{F}_{\alpha}$ with $\alpha=0$. Suppose, in Assumption~\ref{assumption_C}, the capacity of each machine $C$ is given such that $\tilde n=cn$. The DSVRG and DASVRG algorithms proposed in this paper belong to $\mathcal{F}_{c}$ with $\alpha=c$.


 We are ready to present the lower bounds for the rounds of communications.
\vspace{-.01in}
\begin{theorem}\label{thm:lowerbound_main}
	Suppose $\kappa \ge n$ and there exists an algorithm $\mathcal{A}\in\mathcal{F}_{\alpha}$ with the following property:
	
	``For any $\epsilon>0$ and any $N$ convex functions $\{f_i\}_{i\in[N]}$ where each $f_i$ is $L$-smooth and $f$ defined in \eqref{eqn:obj} is $\mu$-strongly convex, there exists $H_{\epsilon}$ such that the output $\hat{x}=\mathcal{A}(\{f_i\}_{i\in[N]},H_{\epsilon})$ satisfies $\Exp[f(\hat{x})-f(x^*)]\le \epsilon$.''

Then, when $m\geq \max\{\exp(\frac{\alpha}{\max\{1,\alpha\}}e^{\frac{2}{\max\{1,\alpha\}}+1}),(e+\max\{1,\alpha\})^2\}$, we must have 
\small
$$
H_{\epsilon} \geq\left(\frac{\sqrt{\kappa/n-1}}{4\sqrt{2}((e+\max\{1,\alpha\})\log m)^{3/2}}\right)\log\left(\left(1-\frac{1}{(e+e^\alpha) (e+\max\{1,\alpha\})^2}\right)\frac{\mu n\|w^*\|^2}{4\epsilon}\right)
\geq\Omega(\sqrt{\frac{\kappa}{n(\log m)^3}}\log(\frac{\mu n \|w^*\|^2}{\epsilon})).
$$
\normalsize
\end{theorem}

We want to emphasis that Theorem~\ref{thm:lowerbound_main} holds without assuming Assumption~\ref{assumption_C}.\footnote{If the capacity $C$ is such that $\tilde n=0$ and $c=0$, the lower bound given by Theorem~\ref{thm:lowerbound_main} still applies to the algorithms in $\mathcal{F}_{\alpha}$ with $\alpha=0$. } In the definition of $\mathcal{F}_{\alpha}$, we allow the algorithm to access both randomly partitioned data and independently sampled data, and allow the algorithm to use local Hessian for preconditioning. This makes our lower bounds in Theorem~\ref{thm:lowerbound_main} {\bf stronger}: Even with an algorithm more powerful than first-order methods (in terms of the class of operations it can take) and with more options in distributing data, the number of rounds needed to find  an $\epsilon$-optimal solution still cannot be reduced.

We note that the condition $\kappa\geq n$ in Theorem~\ref{thm:lowerbound_main} is necessary. Recall that, when $\kappa\leq \frac{n^{1-2\delta}}{32}< n $ for a constant $0<\delta<\frac{1}{2}$, we showed in Proposition~\ref{thm:constantrounds} in Subsection~\ref{sec:DSVRGopt} that $O(\frac{\log(1/\epsilon)}{\delta \log n})$ rounds is enough for DSVRG. Therefore, the lower bound $H\ge \tilde\Omega(\sqrt{\frac{\kappa}{n}}\log(\frac{1}{\epsilon}))=\tilde\Omega({\frac{1}{n^{\delta}}}\log(\frac{1}{\epsilon}))$ won't be true for the case when $\kappa\leq \frac{n^{1-2\delta}}{32}< n $.


\subsection{Proof for the lower bound}

Given a vector $x\in\mathbb{R}^d$ and a set of indices $D\subset[d]$, we use $x_D$ to represent the sub-vector of $x$ that consists of the coordinates of $x$ indexed by $D$. 

\begin{definition}\label{decomposable}
	A function $f:\mathbb{R}^d\rightarrow \mathbb{R}$ is \emph{decomposable} with respect to a partition $D_1,\dots,D_{r}$ of coordinates $[d]$ if $f$ can be written as $f(x) =  g^l(x_{D_1})+\dots +  g^l(x_{D_r})$, where $ g^l:\mathbb{R}^{|D_{i}|}\rightarrow \mathbb{R}$ for $l=1,\dots,r$. A set of functions $\{f_i\}_{i\in[N]}$ is \emph{simultaneously decomposable} w.r.t a partition $D_1,\dots,D_{r}$ if each $f_i$ is decomposable w.r.t $D_1,\dots,D_{r}$.
\end{definition}

It follows the Definition~\ref{def:dfo} and Definition~\ref{decomposable} straightforwardly that:

\begin{proposition}\label{prop:decomposable}
Suppose the functions $\{f_i\}_{i\in[N]}$ in \eqref{eqn:obj} are simultaneously decomposable with respect to a partition $D_1,\dots,D_r$ so that $f_i(x)=\sum_{l=1}^rg_i^l(x_{D_l})$ with functions $g_i^l:\mathbb{R}^{|D_l|}\rightarrow\mathbb{R}$ for $i=1,\dots,N$ and $l=1,\dots,r$. We have

\begin{eqnarray}
\label{eq:decomposedobj}
x_{D_l}^*=\argmin_{w\in\mathbb{R}^{|D_l|}}\left\{\bar g^l(w)\equiv \frac{1}{N}\sum_{i=1}^Ng_i^l(w)\right\},\text{ for }l=1,2,\dots,r.
\end{eqnarray}
\normalsize
where $x^*$ is the optimal solution of \eqref{eqn:obj}.

Moreover, any algorithm $\mathcal{A}\in \mathcal{F}_{\alpha}$, when applied to $\{f_i\}_{i\in[N]}$, becomes decomposable with respect to the same partition $D_1,\dots,D_{r}$ in the following sense: For $l=1,\dots,r$, there exists an algorithm 
$\mathcal{A}_l\in \mathcal{F}_{\alpha}$ such that, after any number of rounds $H$,

$$
\mathbb{E}[f(\hat x)-f(x^*)]=\sum_{l=1}^r\mathbb{E}[\bar g^l(\hat{w}^l)-\bar g^l(x_{D_l}^*)],
$$
\normalsize
where $\hat{x}=\mathcal{A}(\{f_i\}_{i\in[N]},H)\in\mathbb{R}^d$ and $\hat{w}^l=\mathcal{A}_l(\{g_i^l\}_{i\in[N]},H)\in\mathbb{R}^{|D_l|}$ for $l=1,2,\dots,r$.
\end{proposition}

\begin{remark}
	We note a subtlety that might be important for careful readers:  here and throughout the paper, by slight abuse of terminology,   we consider an ``algorithm"  as a sequence of operations that satisfies the requirement of Definition~\ref{def:dfo}. In this sense, an algorithm doesn't have to be describable by a Turing machine and it can access any information (e.g., even the minimizer of the sum of functions) as long as the operations that it takes satisfies the rules in Definition~\ref{def:dfo}. This different interpretation of ``algorithm'' makes Theorem~\ref{thm:lowerbound_main} even stronger and Proposition~\ref{prop:decomposable} (which is essentially a reduction statement) true and trivial. 
\end{remark}

\begin{proof}
The proof of this proposition is straightforward. Since $\{f_i\}_{i\in[N]}$ in \eqref{eqn:obj} are simultaneously decomposable with respect to a partition $D_1,\dots,D_r$, we have

$$
f(x)=\frac{1}{N}\sum_{i=1}^N\sum_{l=1}^rg_i^l(x_{D_l})=\sum_{l=1}^r\bar g^l(x_{D_l})
$$
\normalsize
so that the problem \eqref{eqn:obj} can be solved by solving \eqref{eq:decomposedobj} for each $l$ separately and $x^*_{D_l}$ must be the solution of the $l$-th problem in \eqref{eq:decomposedobj}.

In addition, the function $F_j$ in Definition~\ref{def:dfo} is also decomposable with respect to the same partition $D_1,\dots,D_r$. As a result, its gradient $\nabla F_j(x)$ also has a decomposed structure in the sense that $[\nabla F_j(x)]_{D_l}$ only depends on $x_{D_l}$ for $l=1,2,\dots,r$. Similarly, its Hessian matrix $\nabla^2 F_j(x)$ is a block diagonal matrix with $r$ blocks and the $l$-th block only depends on $x_{D_l}$. These properties ensure that each operation as \eqref{eqn:udpate-rule} conducted by $\mathcal{A}$ can be decomposed into $r$ independent operations as \eqref{eqn:udpate-rule} and applied on $x_{D_1},x_{D_2},\dots,x_{D_r}$ separately. The data distribution and the sequence of operations conducted by $\mathcal{A}$ on $x_{D_l}$ can be viewed as an algorithm $\mathcal{A}_l\in \mathcal{F}_{\alpha}$ applied to $\{g_i^l\}_{i\in[N]}$ so that $\hat{w}^l=\mathcal{A}_l(\{g_i^l\}_{i\in[N]},H)$ is indeed the subvector $\hat{x}_{D_l}$ of the vector $\hat{x}=\mathcal{A}(\{f_i\}_{i\in[N]},H)$ for $l=1,2,\dots,r$ and any $H$.
\qed
\end{proof}

Now we are ready to give the proof for Theorem~\ref{thm:lowerbound_main}.
\begin{proof}[\textbf{Theorem~\ref{thm:lowerbound_main}}]


Let $\mu' = \mu n$, $\kappa'=\frac{L}{\mu'}=\frac{\kappa}{n}$ and $k = (e+\max\{1,\alpha\})\log m$. Since $m\geq \max\{\exp(\frac{\alpha}{\max\{1,\alpha\}}e^{\frac{2}{\max\{1,\alpha\}}+1}),(e+\max\{1,\alpha\})^2\}$, we can show that $v\equiv\frac{N}{nk}=\frac{m}{(e+\max\{1,\alpha\}) \log m}\geq\frac{e+\max\{1,\alpha\}}{2\log (e+\max\{1,\alpha\})} \geq\frac{e}{2}>1$ for any $\alpha\geq 0$.\footnote{Here, we use the factor that $\frac{x}{\log x}$ is monotonically increasing on $[e,+\infty)$.} For the simplicity of notation, we will only prove Theorem~\ref{thm:lowerbound_main} when $k$ and $v$ are both integers. The general case can be proved by a very similar argument only with more sophisticated notations.

We first use the machinery developed by~\cite{Shamir15,Nesterov04book,Lan:15a} to construct $k$ functions on $\mathbb{R}^{b}$ where $b=uk$ for any integer $u\geq 1$. In particular, for $i,j=1,\dots,b$, let $\delta_{i,j}$ be an $b\times b$ matrix with its $(i,j)$ entry being one and others being zeros. Let $M_0,M_1,\dots,M_{b-1}$ be  $b\times b$ matrices defined as

\begin{eqnarray*}
M_i=\left\{
\begin{array}{ll}
\delta_{1,1}&\text{ for }i=0\\
\delta_{i,i} - \delta_{i,i+1}-\delta_{i+1,i} + \delta_{i+1,i+1}&\text{ for }1\leq i\leq b-2\\
\delta_{b-1,b-1} - \delta_{b-1,b}-\delta_{b,b-1} + \frac{\sqrt{\kappa'+k-1}+3\sqrt{k}}{\sqrt{\kappa'+k-1}+\sqrt{k}}\delta_{b,b}&\text{ for }i=b-1.
\end{array}
\right.
\end{eqnarray*}
\normalsize
For $s \in [k]$, let $\Sigma_s=\sum_{i=0}^{u-1}M_{ik+s-1}$. For example, when $u=2$ and $k=3$ (so $b=6$), the matrices $\Sigma_s$'s are given as follows.

	\begin{eqnarray*}\label{Sigmas}
\Sigma_1=\left[
\begin{array}{rrrrrr}
1 & 0& 0 &0 & 0& 0\\
0 & 0& 0 &0 & 0& 0\\
0 & 0& 1 &-1 & 0& 0\\
0 & 0& -1 &1 & 0& 0\\
0 & 0& 0 &0 & 0& 0\\
0 & 0& 0 &0 & 0& 0
\end{array}
\right],
\Sigma_2=\left[
\begin{array}{rrrrrr}
1 & -1& 0 &0 & 0& 0\\
-1 & 1& 0 &0 & 0& 0\\
0 & 0& 0 &0 & 0& 0\\
0 & 0& 0& 1 &-1 &  0\\
0 & 0& 0& -1 &1 &  0\\
0 & 0& 0 &0 & 0& 0
\end{array}
\right],
\Sigma_3=\left[
\begin{array}{rrrrrc}
0 & 0& 0 &0 & 0& 0\\
0& 1 & -1& 0 &0 &  0\\
0& -1 & 1& 0 &0 &  0\\
0 & 0& 0 &0 & 0& 0\\
0 & 0& 0&  0& 1 &-1 \\
0 & 0& 0 &  0& -1 &\frac{\sqrt{\kappa'+k-1}+3\sqrt{k}}{\sqrt{\kappa'+k-1}+\sqrt{k}}
\end{array}
\right].
	\end{eqnarray*}
\normalsize

We define $k$ functions $p_1,\dots,p_{k}:\mathbb{R}^{b}\rightarrow\mathbb{R}$ as follows

	\begin{eqnarray}\label{pfunction}
p_s(w) =\left\{
\begin{array}{ll}
\frac{L}{4}\left[\left(1-\frac{\mu'}{L}\right)\frac{w^T\Sigma_1w}{2} - \left(1-\frac{\mu'}{L}\right)e_1^{\top}w\right] + \frac{\mu'}{2}\|w\|^2&\text{ for }s=1\\
\frac{L}{4}\left[\left(1-\frac{\mu'}{L}\right)\frac{w^T\Sigma_sw}{2}\right]+ \frac{\mu'}{2}\|w\|^2&\text{ for }s=2,\dots,k,
\end{array}
\right.
	\end{eqnarray}
\normalsize
where $e_1=(1,0,\dots,0)^T\in\mathbb{R}^b$,  and denote their average by $\bar p = \frac{1}{k}\sum_{s=1}^{k} p_s$. According to the condition $\kappa\geq n$, we have $1-\frac{\mu'}{L}\geq0$ so that $p_s$ for any $s\in[k]$ and $\bar p$ are all $\mu'$-strongly convex functions. It is also easy to show that $\lambda_{\max}(\Sigma_s)\leq 4$ so that $\nabla p_s$ has a Lipschitz continuity constant of $L\left(1-\frac{\mu'}{L}\right)+\mu'=L$ and $p_s$ is $L$-smooth.

Next, we characterize the optimal solution of $\min_{w\in\mathbb{R}^b}\bar p(w)$.
\begin{lemma}\label{minp}
Let $h\in\mathbb{R}$ be the smaller root of the equation

$$
h^2-2\left(\frac{\kappa'-1+2k}{\kappa'-1}\right)h+1=0,
$$
\normalsize
namely,

$$
h=\frac{\sqrt{\kappa'+k-1}-\sqrt{k}}{\sqrt{\kappa'+k-1}+\sqrt{k}}.
$$
\normalsize
Then, $w^*=(w_1^*,w_2^*,\dots,w_b^*)^T\in\mathbb{R}^b$ with

\begin{eqnarray}
\label{wstar}
w^*_j=h^j,\text{ for }j=1,2,\dots,b
\end{eqnarray}
\normalsize
is the optimal solutions of $\min_{w\in\mathbb{R}^b}\bar p(w)$.
\end{lemma}
\begin{proof}
By definition, we provide the following explicit formulation of $\bar p(w)$

$$
\bar p(w) =\frac{L'-\mu'}{4k}\left[\frac{1}{2}w^T\left(\sum_{s=1}^k\Sigma_s\right)w - e_1^Tw\right] + \frac{\mu'}{2}\|w\|^2.
$$
\normalsize
Observing that

	\begin{eqnarray*}\label{Amatrix}
\sum_{s=1}^k\Sigma_s=\left[
\begin{array}{rrrrrc}
2& -2& 0 &0 & 0& 0\\
-1 & 2& -2 &0 & 0& 0\\
0 & -1& 2 &-1 & 0& 0\\
\vdots & \vdots& \ddots& \ddots &\ddots &  \vdots\\
0 & 0& 0& -1 &2 &  -1\\
0 & 0& 0 &0 & -1& \frac{\sqrt{\kappa'+k-1}+3\sqrt{k}}{\sqrt{\kappa'+k-1}+\sqrt{k}}
\end{array}
\right],
	\end{eqnarray*}
\normalsize
and following \cite{Lan:15a},
we can show that $w^*$ must satisfy the following optimality conditions

	\begin{eqnarray}\nonumber
w^*_2-2\left(\frac{\kappa'-1+2k}{\kappa'-1}\right)w^*_1+1&=&0\\\label{focminp}
w^*_{j+1}-2\left(\frac{\kappa'-1+2k}{\kappa'-1}\right)w^*_j+w^*_{j-1}&=&0,\text{ for }j=2,3,\dots,b-1\\\nonumber
-\left(\frac{\sqrt{\kappa'+k-1}+3\sqrt{k}}{\sqrt{\kappa'+k-1}+\sqrt{k}}+\frac{4k}{\kappa'-1}\right)w^*_b+w^*_{b-1}&=&0
	\end{eqnarray}
\normalsize
We can easily verify that $w^*_j=h^j$ for $j=1,2,\dots,b$ satisfy all equations~\eqref{focminp} and is the optimal solution of $\min_{w\in\mathbb{R}^b}\bar p(w)$.
\qed
\end{proof}

%
%

We claim that $\{p_s\}_{s\in[k]}$ has the following property which directly follows our construction.

\begin{lemma}\label{blockdiag}
		Suppose $U$ is a {\bf strict} subset of $\{p_s\}_{s\in[k]}$, and $q$ is an arbitrary linear combination of $p_s$ in $U$. The Hessian of $q$ is a block diagonal matrix where each block has a size of at most $k$.
\end{lemma}

\begin{proof}
Suppose $p_{s'}$ is not in $U$ for some $s'\in[k]$. Since $q$ is a linear combination of $p_s$'s in $U$, according to the construction in \eqref{pfunction}, the Hessian of $q$ is a linear combination of one diagonal matrix and all $\Sigma_s$'s except $\Sigma_{s'}$, which is a tridiagonal matrix. We note that $\Sigma_{s'}$ is the only matrix among all $\Sigma_s$'s that has non-zero entries in the positions $(s'-1+ik,s'+ik)$ and $(s'+ik,s'-1+ik)$ for $i=0,1,\dots,u-1$ and these positions are periodically repeated with a period of $k$. Therefore, without $\Sigma_{s'}$ involved in the linear combination, the tridiagonal Hessian becomes block diagonal with each block of a size at most $k$.
\qed
\end{proof}

%

To complete the proof of Theorem~\ref{thm:lowerbound_main}, the following lemma is critical. This lemma tells us that the property given by Lemma~\ref{blockdiag} forces the machines to perform a large number of rounds of communication in order to minimize $\bar p$ whenever $\{p_s\}_{s\in[k]}$ do not appear together in any machine.

\begin{lemma}\label{lem:1d}
Suppose $b$ (or $u$) is large enough. Let $\{g_i\}_{i\in[N]}$ be functions on $\mathbb{R}^b$ that consists of $v$ copies of $\{p_s\}_{s\in[k]}$ defined as \eqref{pfunction} and $(n-1)vk$ zero functions, that is,

	\begin{eqnarray}\label{gfunction}
g_i(w) =\left\{
\begin{array}{ll}
p_s(w)&\text{ if }i=s,s+k,s+sk,\dots,s+(v-1)k\\
0&\text{ if }i\geq vk+1.
\end{array}
\right.
	\end{eqnarray}
\normalsize
Let $\bar g=\frac{1}{N}\sum_{i=1}^Ng_i$. We have $w^*=\argmin_{w\in\mathbb{R}^b}\bar p(w)=\argmin_{w\in\mathbb{R}^b}\bar g(w)$ where $w^*$ is defined as \eqref{wstar}.

Suppose an algorithm $\mathcal{A}\in\mathcal{F}_{\alpha}$ is applied to $\{g_i\}_{i\in[N]}$. Let $\mathcal{E}$ be the random event that none of the $m$ machines has all functions in $\{p_s\}_{s\in[k]}$ (in either $S_j$ or $R_j$) after the data distribution stage of $\mathcal{A}$ and let $\hat{w}=\mathcal{A}(\{g_i\}_{i\in[N]},H)$. Then, to ensure $\mathbb{E}[g(\hat{w})-g(w^*)|\mathcal{E}]\leq\epsilon$, we need $H=\left(\frac{\sqrt{\kappa'+k-1}-\sqrt{k}}{4k\sqrt{k}}\right)\log\left(\frac{\mu\|w^*\|^2}{4\epsilon}\right)$.
\end{lemma}

\begin{proof}
Since $\bar g=\frac{1}{N}\sum_{i=1}^Ng_i=\frac{1}{nk}\sum_{s=1}^kp_s=\frac{\bar p}{n}$, we have $w^*=\argmin_{x\in\mathbb{R}^b}\bar g(x)=\argmin_{x\in\mathbb{R}^b}\bar p(x)$ by Lemma~\ref{minp}, where $w^*$ is defined as in \eqref{wstar}.

Let $E_0=\{\mathbf{0}\}$ and $E_t$ be the linear space spanned by the unit vectors $e_1,\dots, e_t$ for $t=1,\dots,b$. Suppose event $\mathcal{E}$ happens. Every machine will only have a strict subset $U$ of $\{p_s\}_{s\in[k]}$. Lemma~\ref{blockdiag} guarantees that, under algorithm $\mathcal{A}$, if machine $j$ starts one round with a set of working vectors $W_j\subset E_t$, then $W_j$ is always contained by the space $E_{t+k}$ after this round. Therefore, we can show that, at the beginning of round $\ell$ in algorithm $\mathcal{A}$, if $\cup_{j=1}^m W_j \subset E_t$, then at the end of round $\ell$ (and at the beginning of round $\ell+1$), we have $\cup_j W_j \subset E_{t+k}$. Using this finding and the fact that $\cup W_j = \{\mathbf{0}\}= E_0$ initially, we conclude that, after $H$ rounds in $\mathcal{A}$, $\cup_j W_j \subset E_{Hk}$. Let $t = Hk$.  Since $\hat{w}=\mathcal{A}(\{g_i\}_{i\in[N]},H)$, we must have $\hat{w}\in E_t$.

By \eqref{wstar}, we can show that

\begin{eqnarray}
\label{eq:lowerboundsparse}
\|w^*\|^2=\sum_{j=1}^b(w^*_j)^2=\sum_{j=1}^bh^{2j}=\frac{h^2(1-h^{2b})}{1-h^2}.
\end{eqnarray}
\normalsize
Following the analysis in~\cite{Shamir15,Nesterov04book,Lan:15a} and using the $\mu'$-strong convexity of $\bar p$, we have

\begin{eqnarray*}
\mathbb{E}[\bar p(\hat{w})-\bar p(w^*)|\mathcal{E}]
\geq\frac{\mu'}{2}\mathbb{E}[\|\hat{w}-w^*\|^2|\mathcal{E}]
\geq\frac{\mu'}{2}\sum_{j=t+1}^b\mathbb{E}[(w^*_j)^2|\mathcal{E}]
\geq\frac{\mu'}{2}\frac{h^{2t+2}(1-h^{2b-2t})}{1-h^2}
\geq\frac{\mu'\|w^*\|^2}{2}\frac{(h^{2t}-h^{2b})}{1-h^{2b}}
\end{eqnarray*}
\normalsize
where the second inequality is because $\hat{w}\in E_t$ and
the third inequality is due to \eqref{eq:lowerboundsparse}. When $b$ (or $u$) is large enough, the inequality above implies

$$
\mathbb{E}[\bar p(\hat{w})-\bar p(w^*)|\mathcal{E}]
\geq\frac{\mu'\|w^*\|^2h^{2t}}{4}
=\frac{\mu'\|w^*\|^2}{4}\left(\frac{\sqrt{\kappa'+k-1}-\sqrt{k}}{\sqrt{\kappa'+k-1}+\sqrt{k}}\right)^{2t}.
$$
\normalsize
Based on this inequality, when $\mathbb{E}[\bar g(\hat{w})-\bar g(w^*)|\mathcal{E}]\leq\epsilon$, or equivalently, when $\mathbb{E}[\bar p(\hat{w})-\bar p(w^*)|\mathcal{E}]\leq n\epsilon$, we must have

$$
\log\left(\frac{\mu'\|w^*\|^2}{4n\epsilon}\right)
\leq
2t\log\left(\frac{\sqrt{\kappa'+k-1}+\sqrt{k}}{\sqrt{\kappa'+k-1}-\sqrt{k}}\right)
\leq
2t\log\left(1+\frac{2\sqrt{k}}{\sqrt{\kappa'+k-1}-\sqrt{k}}\right)
\leq
\frac{4t\sqrt{k}}{\sqrt{\kappa'+k-1}-\sqrt{k}}
$$
\normalsize
which further implies

$$
H = \frac{t}{k} \ge \left(\frac{\sqrt{\kappa'+k-1}-\sqrt{k}}{4k\sqrt{k}}\right)\log\left(\frac{\mu\|w^*\|^2}{4\epsilon}\right)
$$
\normalsize
\qed
\end{proof}

We now complete the proof of Theorem~\ref{thm:lowerbound_main} by constructing $N$ special functions $\{f_i\}_{i\in[N]}$ on $\mathbb{R}^d$ with $d=nb$ for a sufficiently large $b$ (or $u$) based on $\{p_s\}_{s\in[k]}$, so that any algorithm $\mathcal{A}\in\mathcal{F}_{\alpha}$, when applied to $\{f_i\}_{i\in[N]}$, will need at least the targeted amount of rounds of communication.

We partition the set of indices $[d] = \{1,\dots, d\}$ into $n$ disjoint subsets $D_1,D_2,\dots, D_n$ with $|D_j|=b$ and $D_j = \{b(j-1)+1,\dots, bj\}$. For any $j\in [n]$ and $s\in [k]$, let $q_{j,s}(x)$ be a function on $\mathbb{R}^d$ such that  $q_{j,s}(x)=p_s(x_{D_j})$, which means $q_{j,s}(x)$ only depends on the $b$ coordinates of $x$ indexed by $D_j$. Therefore, we obtain $nk$ different functions $\{q_{j,s}\}_{j\in[n],s\in[k]}$. Finally, we define $\{f_i\}_{i\in[N]}$ to be a set that consists of $v$ copies of $\{q_{j,s}\}_{j\in[n],s\in[k]}$ (recall that $N=vkn$ and $v\geq 1$ is an integer). Because

\begin{eqnarray}
\label{fandp}
f(x) = \frac{1}{N}\sum_{i=1}^{N} f_i(x)=\frac{v}{vnk}\sum_{j=1}^{n}\sum_{s=1}^{k} q_{j,s}(x)=\frac{v}{vnk}\sum_{j=1}^{n}\sum_{s=1}^{k} p_s(x_{D_j})=\frac{1}{nk}\sum_{j=1}^{n}\sum_{s=1}^{k} p_s(x_{D_j}) = \frac{1}{n}\sum_{j=1}^{n}\bar p(x_{D_j})
\end{eqnarray}
\normalsize
and Lemma~\ref{minp}, the optimal solution $x^*$ for \eqref{eqn:obj} with $\{f_i\}_{i\in[N]}$ constructed as above is $x^*=(w^*,w^*,\dots,w^*)^T$ where $w^*\in\mathbb{R}^d$ is defined as \eqref{wstar} and is repeated for $n$ times.

Now, we want to verify that functions $\{f_i\}_{i\in[N]}$ satisfy our assumptions. In fact, we have shown that $p_s$ is $L$-smooth for each $s\in[k]$. Since $f_i$ is either an zero function or equals $p_s(x_{D_j})$ for some $j\in[n]$ and $s\in[k]$, the function $f_i$ is $L$-smooth for each $i\in[N]$ as well. Since $\bar p$ is $\mu'$-strongly convex (on $\mathbb{R}^b$) and $\mu'=n\mu$, the function $f$ defined in \eqref{eqn:obj} must be $\mu$-strongly convex (on $\mathbb{R}^d$) according to the relationship~\eqref{fandp}.

%
%


 \newcommand{\sJ}{\text{DOS}}

%


According to its construction, $\{f_i\}_{i\in[N]}$ are simultaneously decomposable with respect to a partition $D_1,\dots,D_n$ with $D_j = \{d(j-1)+1,\dots, dj\}$ (see Definition~\ref{decomposable}). In particular, for any $i\in[N]$, $f_i(x)=\sum_{l=1}^ng_i^l(x_{D_l})$ where $g_i^l\in\{p_s\}_{s\in[k]}$ for exactly one $l\in[n]$ and $g_i^l=0$ for other $l$'s. Moreover, for any $l\in[n]$,  $\{g_i^l\}_{i\in[N]}=\{g_i\}_{i\in[N]}$ where $\{g_i\}_{i\in[N]}$ are defined as \eqref{gfunction} such that   $\bar g^l\equiv\frac{1}{N}\sum_{i=1}^Ng_i^l=\frac{1}{N}\sum_{i=1}^Ng_i=\bar g$.

By Proposition~\ref{prop:decomposable}, $\mathcal{A}$ can be decomposed with respective to the same partition $D_1,D_2,\dots, D_n$ into $\mathcal{A}_1,\dots,\mathcal{A}_n\in\mathcal{F}_{\alpha}$ and $\mathcal{A}_l$ is applied to $\{g_i\}_{i\in[N]}$.
Following Definition~\ref{def:dfo}, let $ S_1, \dots, S_{m}$ be the random partition of $[N]$ and and $R_1,\dots, R_m$ be set of i.i.d. indices uniformly drawn from $[N]$ with $|R_j| = \alpha n $. Let $S'_j=S_j\cup R_j$. Then, the algorithm $\mathcal{A}_l$ will allocate $\{g_i|i\in S'_j\}$ to machine $j$ and start the computation in rounds.

We now focus on the solution generated by $\mathcal{A}_l$ for any $l$. For machine $j$ in $\mathcal{A}_l$, let $Y_{1,j}$ be the number of functions in $\{p_s\}_{s\in[k]}$ (repetitions counted) that are contained in $S_{j}$ and $Y_{2,j}$ be the number of functions in $\{p_s\}_{s\in[k]}$ (repetitions counted) that are contained in $R_{j}$.

Due to \eqref{gfunction}, the function $g_i$ is not a zero function if and only if $1\leq i\leq vk=m$. Hence, $Y_{1,j}$ has a hypergeometric distribution where $\text{Prob}(Y_{1,j}=r)$ equals the probability of $r$ successes in $n$ draws, without replacement, from a population of size $N$ that contains exactly $m$ successes. According to Chvatal~\cite{Chvatal1979}, we have

$$
\text{Prob}(Y_{1,j}\geq r)\leq\left(\frac{1}{r}\right)^{r}\left(\frac{n-1}{n-r}\right)^{n-r}
$$
\normalsize
which, when $r=e\log m$, implies

\begin{eqnarray}
\nonumber
\text{Prob}(Y_{1,j}\geq e\log m)&\leq&\left(\frac{1}{e\log m}\right)^{e\log m}\left(\frac{n-1}{n-e\log m}\right)^{n-e\log m}\\
\nonumber
&=&\left(\frac{1}{e\log m}\right)^{e\log m}\left(1+\frac{e\log m-1}{n-e\log m}\right)^{n-e\log m}\\
\nonumber
&<&\left(\frac{1}{e\log m}\right)^{e\log m}e^{e\log m-1}\\\nonumber
&=&\frac{1}{e}\left(\frac{1}{\log m}\right)^{e\log m}\\\nonumber
&\leq&\frac{1}{e}\left(\frac{1}{2\log(e+1)}\right)^{e\log m}\\\label{eq:Y1}
&\leq&\frac{1}{em^2}
\end{eqnarray}
\normalsize
where the second inequality is because $(1+\frac{1}{x})^x<e$ for any $x>0$, the third inequality is due to the assumption that $m\geq \max\{\exp(\frac{\alpha}{\max\{1,\alpha\}}e^{\frac{2}{\max\{1,\alpha\}}+1}),(e+\max\{1,\alpha\})^2\}\geq (e+1)^2$, and the last inequality is because $(2\log(e+1))^e>e^2$.

%

On the other hand, we can represent $Y_{2,j}  =\sum_{r \in R_j}\mathbf{1}_{r\leq vk}$ which is the sum of $\alpha n$ i.i.d. binary random variables $\mathbf{1}_{r\leq vk}$'s which equal one with a probability of $\frac{vk}{N}=\frac{1}{n}$ and zero with a probability of $1-\frac{1}{n}$.
By Chernoff inequality of multiplicative form, we have

\begin{eqnarray}
\nonumber
\text{Prof}(Y_{2,j}\geq \max\{1,\alpha\}\log m) 
&\leq&\left(\frac{e^{\frac{\max\{1,\alpha\}}{\alpha}\log m- 1}}{\left(\frac{\max\{1,\alpha\}}{\alpha}\log m\right)^{\frac{\max\{1,\alpha\}}{\alpha}\log m}}\right)^\alpha\\\nonumber
&=&\frac{e^{\max\{1,\alpha\}\log m- \alpha}}{\left(\frac{\max\{1,\alpha\}}{\alpha}\log m\right)^{\max\{1,\alpha\}\log m}}\\\nonumber
&=&\frac{1}{e^\alpha}\left(\frac{e\alpha}{\max\{1,\alpha\}\log m}\right)^{\max\{1,\alpha\}\log m}\\\nonumber
&\leq&\frac{1}{e^\alpha}\left(\frac{1}{e^{\frac{2}{\max\{1,\alpha\}}}}\right)^{\max\{1,\alpha\}\log m}\\\label{eq:Y2}
&=&\frac{1}{e^\alpha m^2},
\end{eqnarray}
\normalsize
where the second inequality is because  of the assumption that $m\geq \max\{\exp(\frac{\alpha}{\max\{1,\alpha\}}e^{\frac{2}{\max\{1,\alpha\}}+1}),(e+\max\{1,\alpha\})^2\}\geq \exp(\frac{\alpha}{\max\{1,\alpha\}}e^{\frac{2}{\max\{1,\alpha\}}+1})$.

Combining \eqref{eq:Y1} and \eqref{eq:Y2} for $j=1,2,\dots,m$ and using the union bound, we have

\begin{eqnarray}
\nonumber
\text{Prob}(Y_{1,j}\geq e\log m\text{ for some }j \text{ or }Y_{2,j}\geq \max\{1,\alpha\}\log m\text{ for some } j)\leq \frac{1}{em}+\frac{1}{e^\alpha m}
= \frac{1}{(e+e^\alpha) m},
\end{eqnarray}
\normalsize
which implies

\begin{eqnarray}
\nonumber
\text{Prob}(Y_{1,j}+Y_{2,j}< (e+\max\{1,\alpha\})\log m\text{ for }j=1,2,\dots,m)\geq 1-\frac{1}{(e+e^\alpha) m}.
\end{eqnarray}
\normalsize
Therefore, we have shown that, with a probability of at least $ 1-\frac{1}{(e+e^\alpha) m}$, all of the sets $S_1',\dots, S_m'$ contain fewer than $(e+\max\{1,\alpha\}) \log m =k$ functions from $\{p_s\}_{s\in[k]}$ (repetition counted). In other words, with a probability of at least $ 1-\frac{1}{(e+e^\alpha) m}$, none of the sets $S_1',\dots, S_m'$ contains all of the functions in $\{p_s\}_{s\in[k]}$. If the event that ``\emph{none of the sets $S_1',\dots, S_m'$ contains all of functions of $\{p_s\}_{s\in[k]}$}'' (same as the event $\mathcal{E}$ in Lemma~\ref{lem:1d}) indeed happens in $\mathcal{A}_l$, we call $\mathcal{A}_l$ \textit{bad}. Then, we  have actually proved

$$
\text{Prob}(\mathcal{A}_l\text{ is bad})\geq\left(1-\frac{1}{(e+e^\alpha) m}\right)\geq \left(1-\frac{1}{(e+e^\alpha) (e+\max\{1,\alpha\})^2}\right).
$$
\normalsize




By Proposition \ref{prop:decomposable} and $\{g_i^l\}_{i\in[N]}=\{g_i\}_{i\in[N]}$, after $H$ rounds, the solutions  $\hat{x}=\mathcal{A}(\{f_i\}_{i\in[N]},H)\in\mathbb{R}^d$ and $\hat{w}^l=\mathcal{A}_l(\{g_i^l\}_{i\in[N]},H)=\mathcal{A}_l(\{g_i\}_{i\in[N]},H)\in\mathbb{R}^{|D_l|}$ for $l=1,2,\dots,n$ satisfy

\begin{eqnarray}
\nonumber
\mathbb{E}[f(\hat x)-f(x^*)]&=&\sum_{l=1}^n\mathbb{E}[\bar g^l(\hat{w}^l)-\bar g^l(x_{D_l}^*)]\\\nonumber
&=&\sum_{l=1}^n\mathbb{E}[\bar g(\hat{w}^l)-\bar g(x_{D_l}^*)]\\\nonumber
&\geq &\sum_{l=1}^n\mathbb{E}[\bar g(\hat{w}^l)-\bar g(x_{D_l}^*)|\mathcal{A}_l\text{ is bad}]\text{Prob}(\mathcal{A}_l\text{ is bad})\\\nonumber
&\geq &\sum_{l=1}^n\mathbb{E}[\bar g(\hat{w}^l)-\bar g(x_{D_l}^*)|\mathcal{A}_l\text{ is bad}]\left(1-\frac{1}{(e+e^\alpha) (e+\max\{1,\alpha\})^2}\right).
\end{eqnarray}
\normalsize
Therefore, if $\mathbb{E}[f(\hat x)-f(x^*)]\leq\epsilon$, there must exist an $l\in[n]$ such that

\begin{eqnarray}
\label{eq:lmachinegap}
\mathbb{E}[\bar g(\hat{w}^l)-\bar g(x_{D_l}^*)|\mathcal{A}_l\text{ is bad}]\left(1-\frac{1}{(e+e^\alpha) (e+\max\{1,\alpha\})^2}\right) \le \frac{\epsilon}{n}.
\end{eqnarray}
\normalsize
When $\mathcal{A}_l$ is bad,  after the data distribution stage, none of the $m$ machines in $\mathcal{A}_l$ has all functions in $\{p_s\}_{s\in[k]}$. According to Lemma~\ref{lem:1d},  we know that to ensure \eqref{eq:lmachinegap}, $\mathcal{A}_l$ needs

\begin{eqnarray*}
H &\ge& \left(\frac{\sqrt{\kappa'+k-1}-\sqrt{k}}{4k\sqrt{k}}\right)\log\left(\left(1-\frac{1}{(e+e^\alpha) (e+\max\{1,\alpha\})^2}\right)\frac{\mu n\|w^*\|^2}{4\epsilon}\right)\\
&\geq&\left(\frac{\sqrt{\kappa'-1}}{4\sqrt{2}k\sqrt{k}}\right)\log\left(\left(1-\frac{1}{(e+e^\alpha) (e+\max\{1,\alpha\})^2}\right)\frac{\mu n\|w^*\|^2}{4\epsilon}\right)
\end{eqnarray*}
\normalsize
which is the desired lower bound after plugging in $k=(e+\max\{1,\alpha\})\log m$.

\qed
\end{proof}

\section{Numerical Experiments}
\label{sec:exp}

\begin{figure}[t]
    \begin{tabular}[h]{@{}c|ccc@{}}
    $\lambda$ & Covtype(RFF) & Million Song(RFF) & Epsilon \\
    \hline \\
    \raisebox{10ex}{$\frac{1}{N^{\frac{1}{2}}}$}
        &\quad \includegraphics[width=0.28\textwidth]{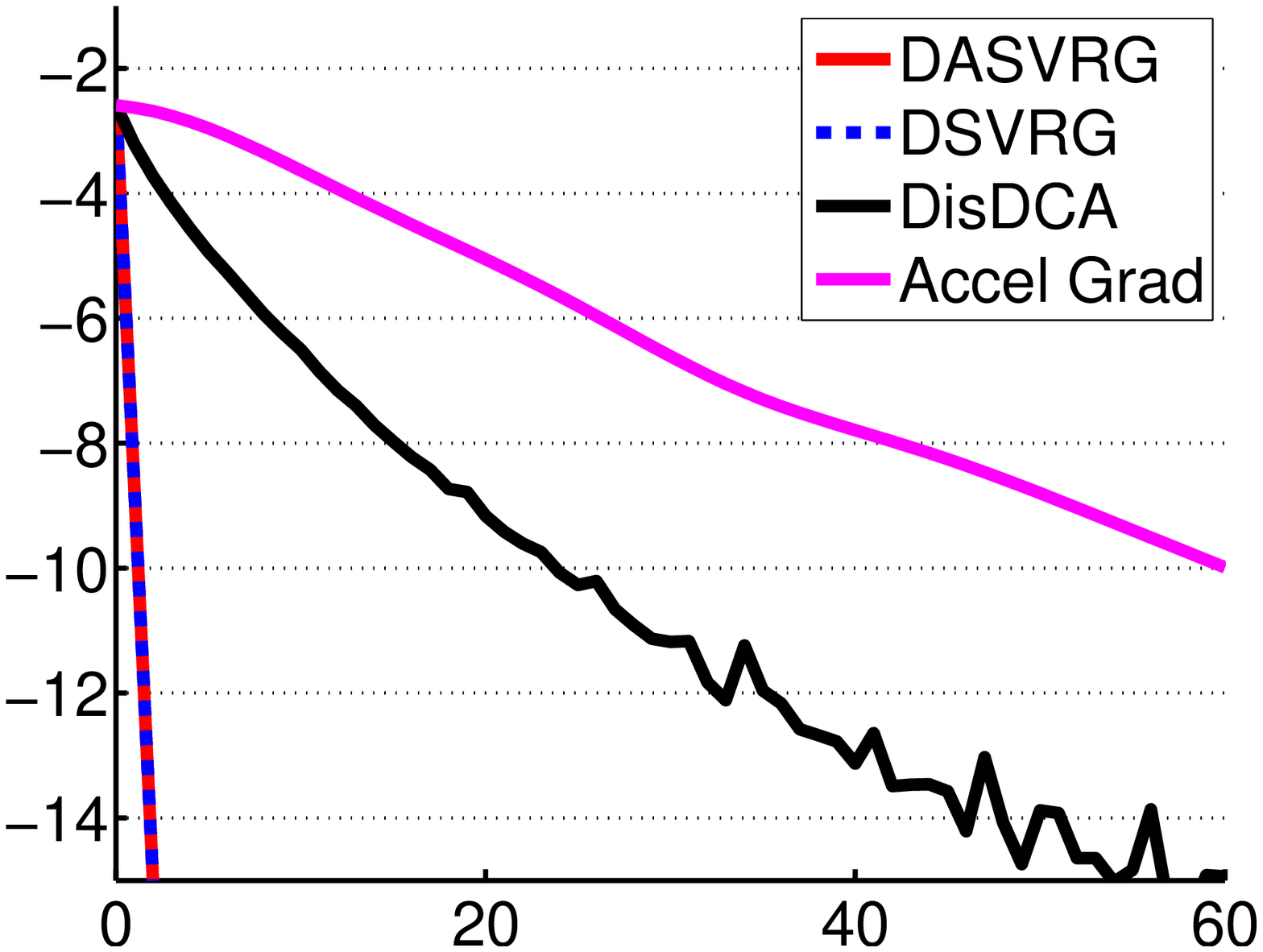}
        & \includegraphics[width=0.28\textwidth]{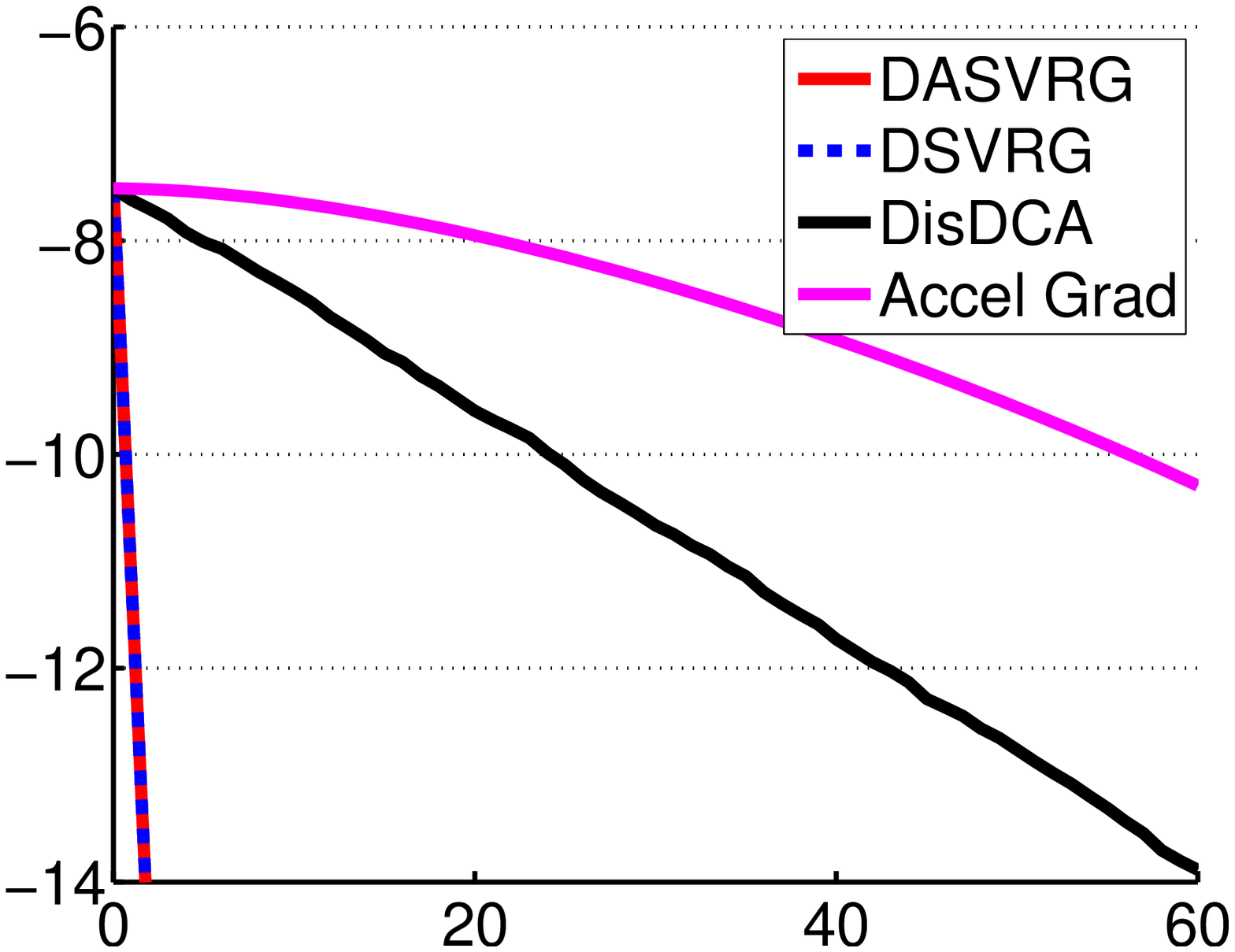}
        & \includegraphics[width=0.28\textwidth]{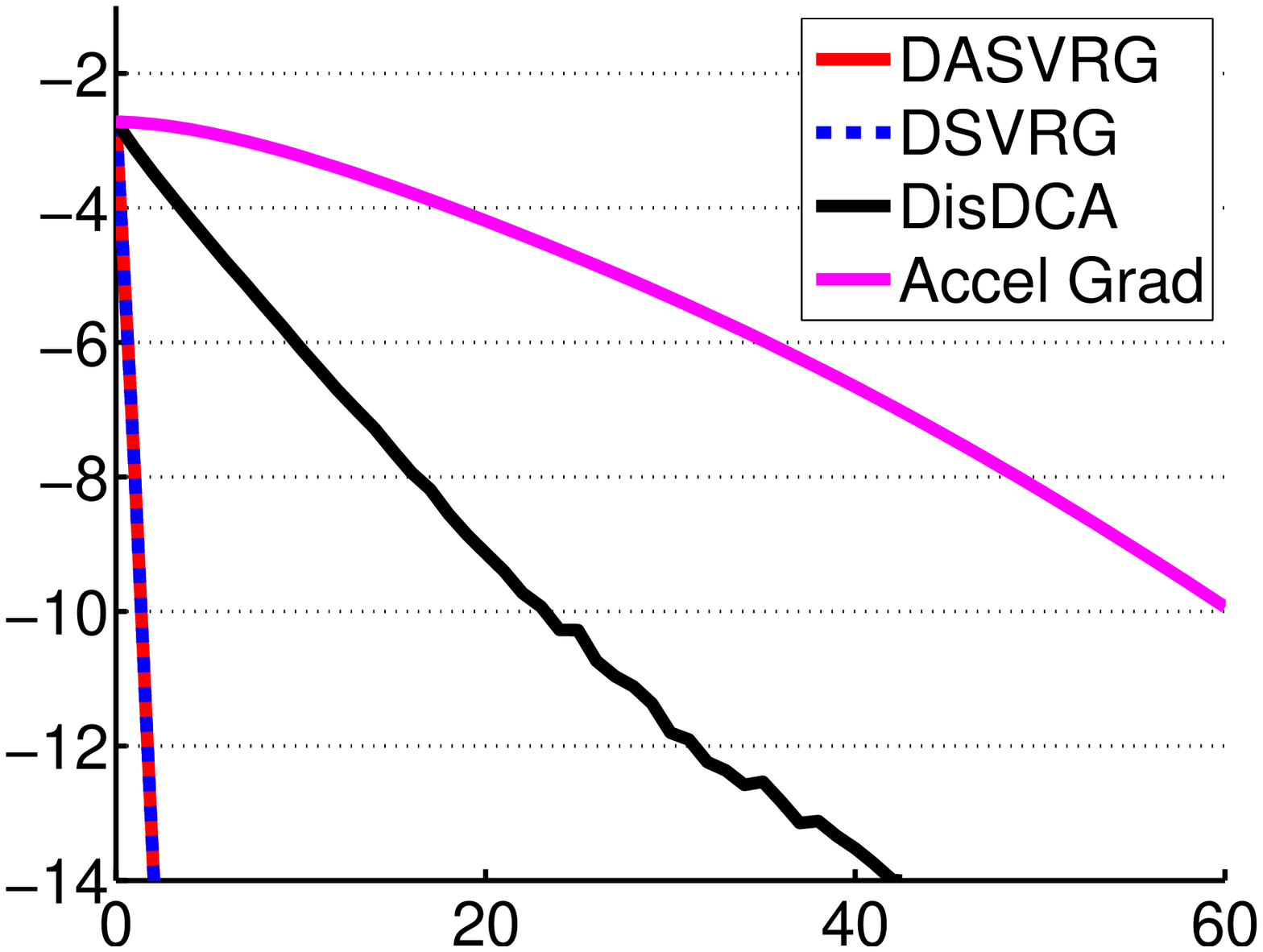}\\
    \raisebox{10ex}{$\frac{1}{N^{\frac{3}{4}}}$}
        &\quad \includegraphics[width=0.28\textwidth]{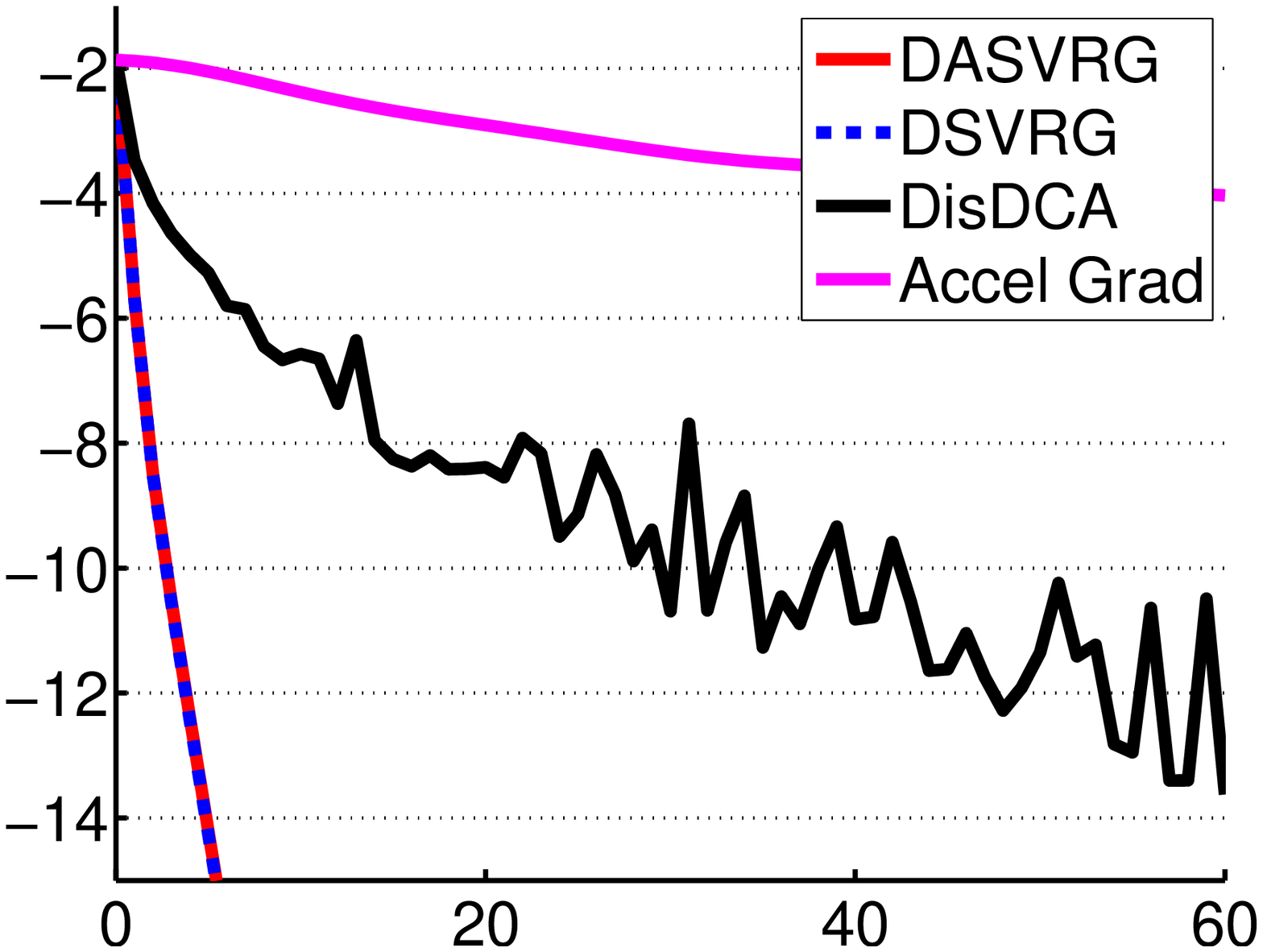}
        & \includegraphics[width=0.28\textwidth]{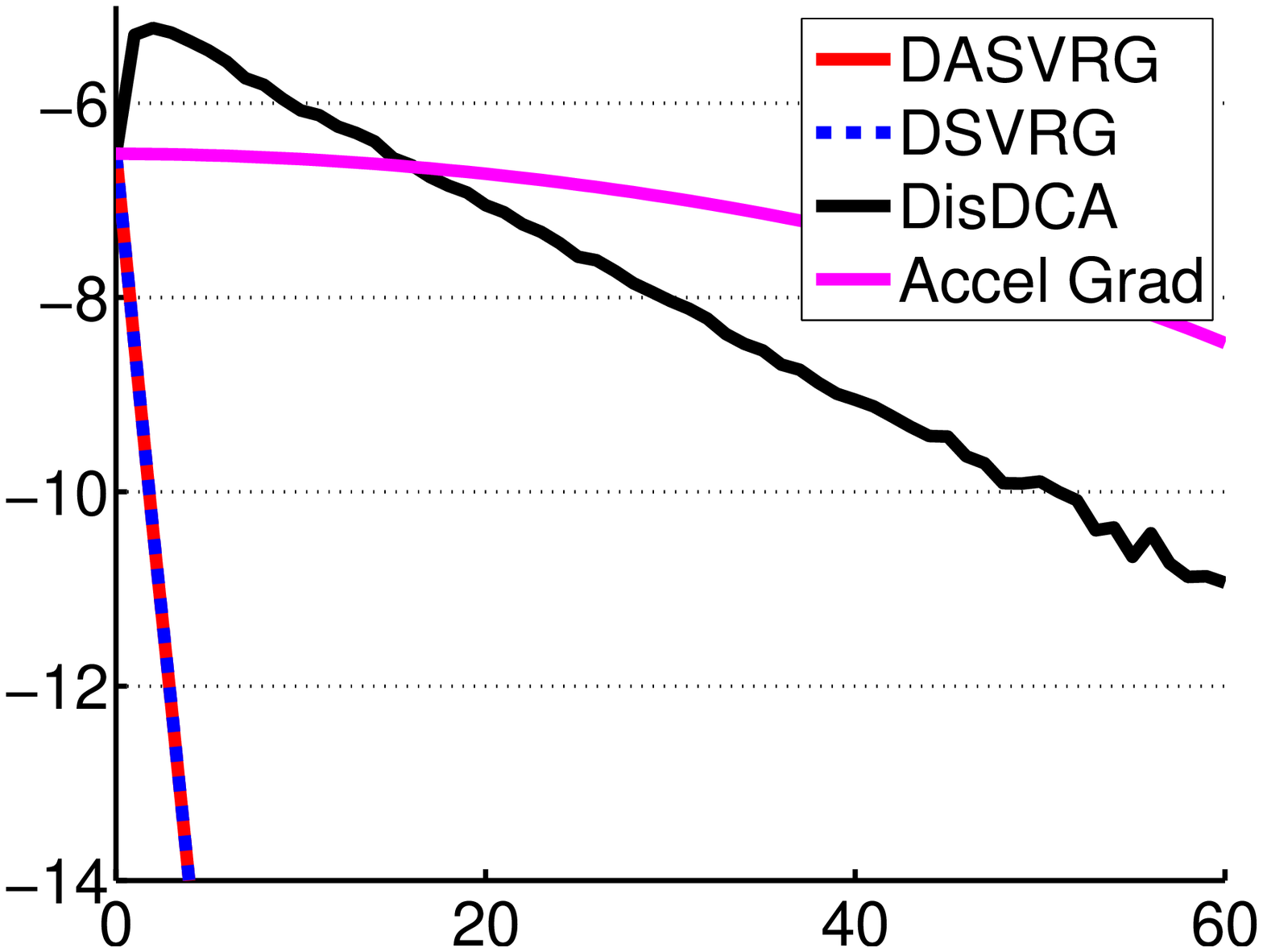}
        & \includegraphics[width=0.28\textwidth]{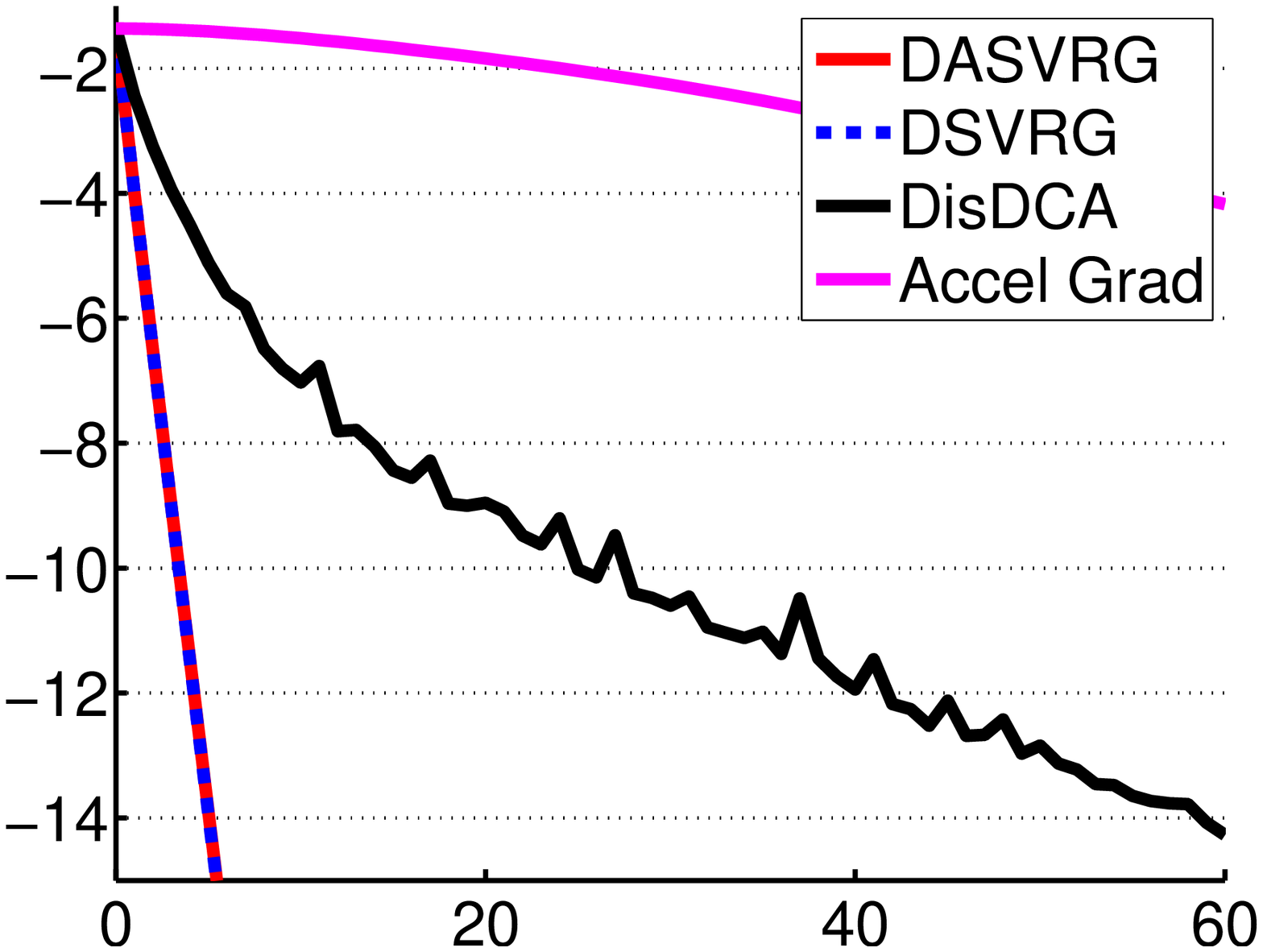}\\
    \raisebox{10ex}{$\frac{1}{N}$}
        &\quad \includegraphics[width=0.28\textwidth]{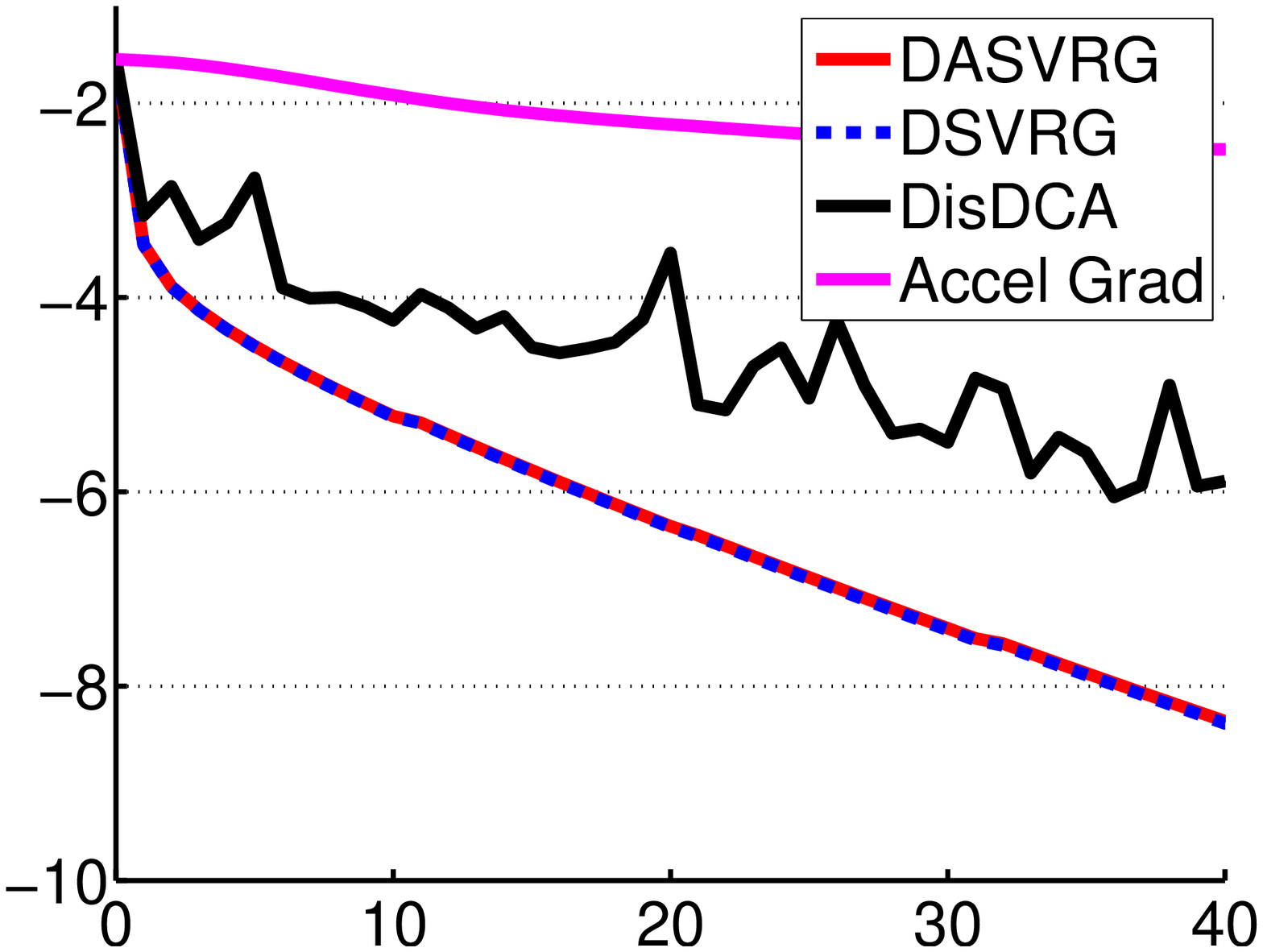}
        & \includegraphics[width=0.28\textwidth]{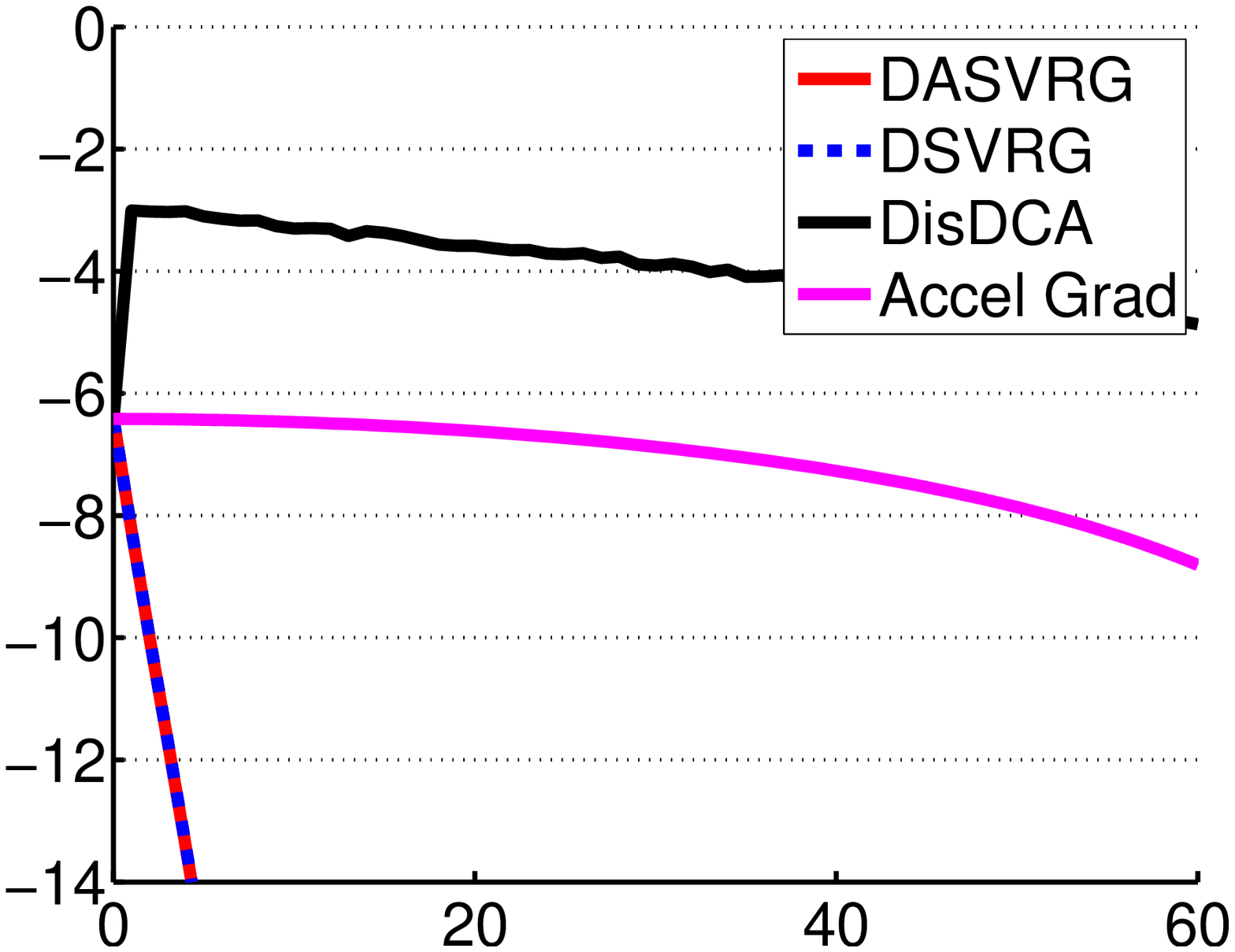}
        & \includegraphics[width=0.28\textwidth]{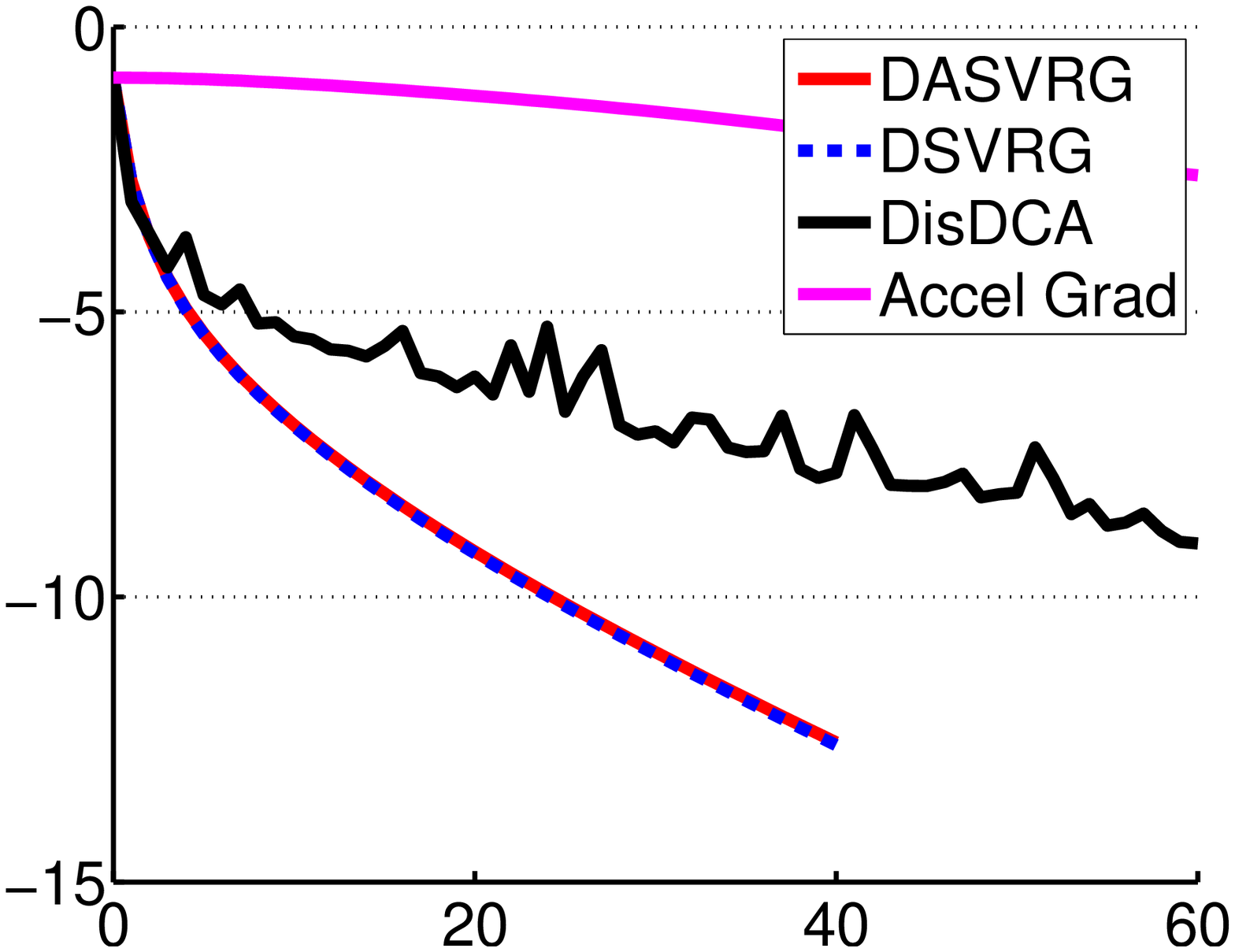}\\
\end{tabular}
\vspace{2ex}
\caption{Comparing the DSVRG and DASVRG methods with DisDCA and the accelerated gradient method (Accel Grad) in rounds.}
\label{fig:round}
\end{figure}

In this section, we conduct numerical experiments to compare our DSVRG and DASVRG algorithms with the DisDSCA~\cite{Yang:13} (with its practical updates) and a distributed implementation of the accelerated gradient method (Accel Grad) by Nesterov~\cite{Nesterov04book}. We apply these four algorithms to the ERM problem \eqref{eqn:regemr} with three  datasets\footnote{\url{http://www.csie.ntu.edu.tw/~cjlin/libsvmtools/datasets/binary.html}}: Covtype, Million Song and Epsilon. According to the types of data, the loss function $\phi(x,\xi)$ in \eqref{eqn:regemr} is chosen to be the square loss in ridge regression for Million Song and the logistic loss in logistic regression for the other two datasets. Following the previous work, we map the target variable of year from $1922 \sim 2011$ into $[0,1]$ for the Million Song data. We notice that the original Covtype and Million Song datasets are not very large (Covtype has 62M in the original size and Million Song has 450M in the original size), and both our algorithms and DisDCA can finish quickly on our server.  Therefore, to make comparison among these algorithms in a more challenging setting, we conduct experiments  using random Fourier features  (RFF)~\cite{rahimi2007random} on Covtype and Million Song datasets. The RFF is a popular method for solving large-scale kernel methods by generating finite dimensional features, of which the inner product approximate the kernel similarity. We generate RFF corresponding to RBF kernel.  Finally, Covtype data has $N=522,911$ examples $d=1,000$ features, Million Song data has $N=463,715$ examples and $d=2,000$ features. Since the original Epsilon data is large enough (12G), we use its original features.


The experiments are conducted on one server (Intel(R) Xeon(R) CPU E5-2667 v2 \@ 3.30GHz) with multiple processes with each process simulating one machine. We first choose the number of processes (machines) to be $m=5$. To test the performances of algorithms for different condition numbers, we choose the value of the regularization parameter $\lambda$ in \eqref{eqn:regemr} to be $1/N^{0.5}$, $1/N^{0.75}$ and $1/N$. For each setting, $L$ is computed as $\frac{\max_{i=1,\dots,N}\|a_i\|^2}{\gamma}+\lambda$ where $\frac{1}{\gamma}$ is the Lipschiz continuous constant of $\nabla_x \phi(x,\xi)$, and $\mu$ is equal to $\lambda$. We implement DSVRG by choosing $\eta=\frac{1}{L}$, $T=10,000$ and $K=\frac{N}{T}$. For DASVRG, we choose $\eta=\frac{1}{L}$, $T=10,000$, $K=1$ and $P=\frac{N}{T}$. In both DSVRG and DASVRG, we directly choose $R_j=S_j$ (so that $|R_j|=\frac{N}{m}$ and $Q=N$) since it saves the time for data allocation and, in practice, gives performances very similar to the performances when $R_j$ is sampled separately. For DisDCA, we use SDCA~\cite{SSZhang13SDCA} as the local solver so that it is equivalent to the implementation of CoCoA+ with $\sigma'=m$ and $\gamma=1$ as in the experiments in~\cite{ma2015adding}. We run SDCA for $T=10,000$ iterations in each round of DisDCA with $\frac{N}{T}$ rounds in total.

The numerical results are presented in Figure~\ref{fig:round} and Figure~\ref{fig:time}. The horizontal axis presents the number of rounds of communication conducted by algorithms in Figure~\ref{fig:round} and presents the parallel runtime (in seconds) used by the algorithms in Figure~\ref{fig:time}. In both figures, the vertical axis represents the logarithm of optimality gap. According to Figure~\ref{fig:round} and Figure~\ref{fig:time}, the performances of all algorithms get worsen when $\lambda$ decreases (so the condition number increases). We find that DSVRG and DASVRG have almost identical performances in rounds of communication and they both outperform the other two methods significantly. This shows the merit of our methods when applied to computer clusters with a high communication cost due to significant network delay. DSVRG and DASVRG have slightly different performances in runtime and they outperform the other two methods in Million Song data and obtain a comparable performance on Covtype data. DSVRG and DASVRG do not perform as good as DisDCA in runtime on Epsilon data.

\begin{figure}[t]
    \begin{tabular}[h]{@{}c|ccc@{}}
    $\lambda$ & Covtype(RFF) & Million Song(RFF) & Epsilon \\
    \hline \\
    \raisebox{10ex}{$\frac{1}{N^{\frac{1}{2}}}$}
        &\quad \includegraphics[width=0.28\textwidth]{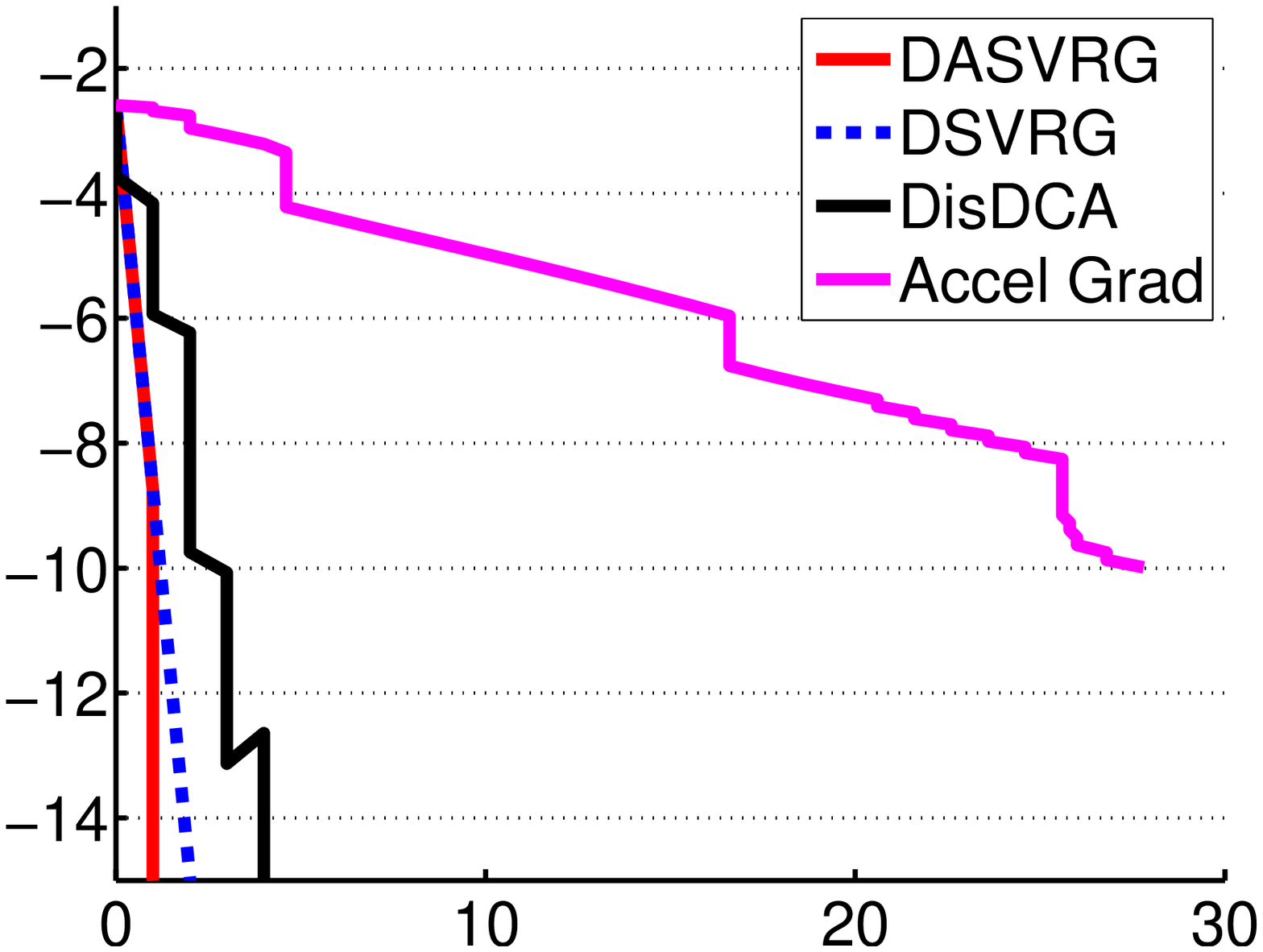}
        & \includegraphics[width=0.28\textwidth]{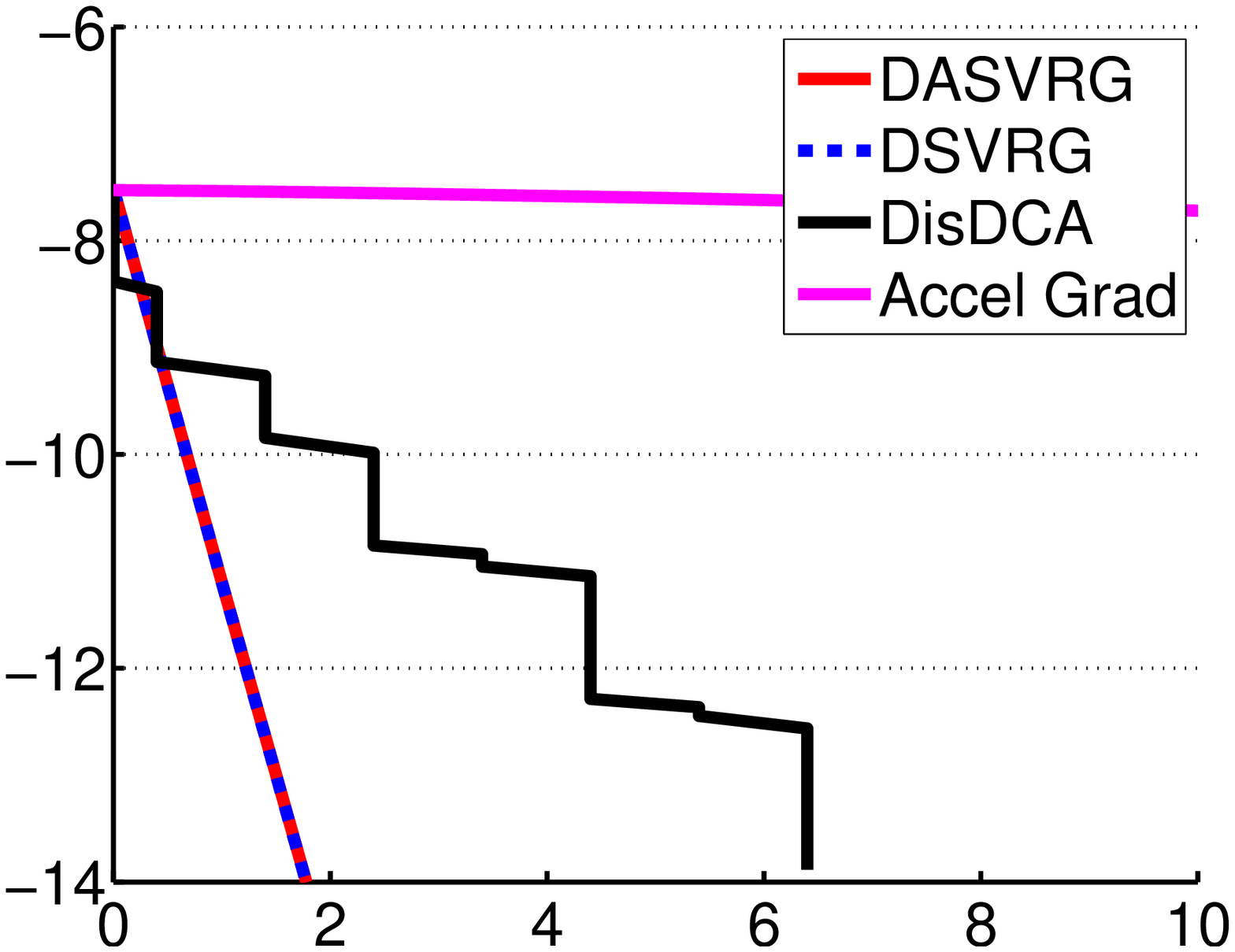}
        & \includegraphics[width=0.28\textwidth]{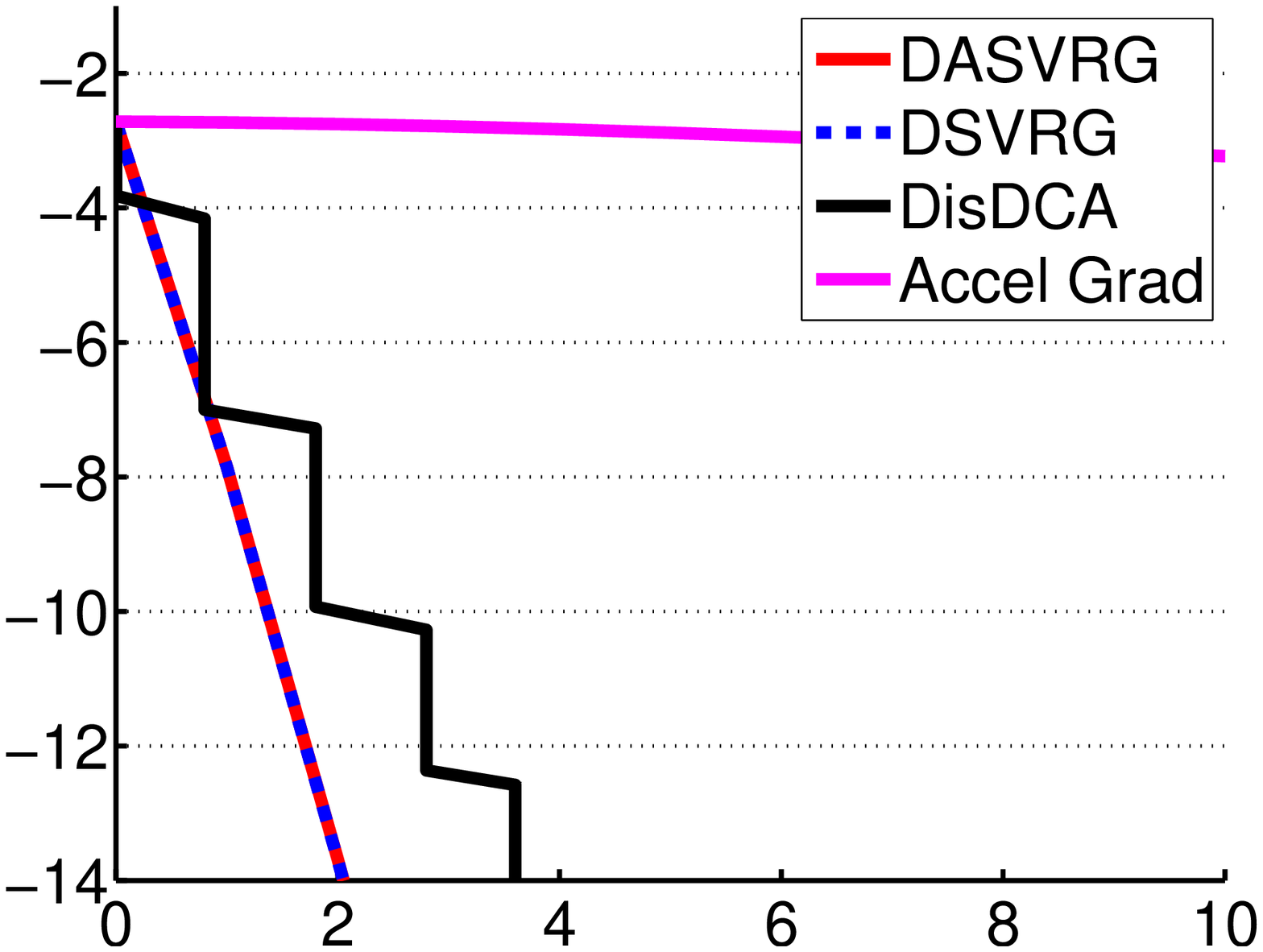}\\
    \raisebox{10ex}{$\frac{1}{N^{\frac{3}{4}}}$}
        &\quad \includegraphics[width=0.28\textwidth]{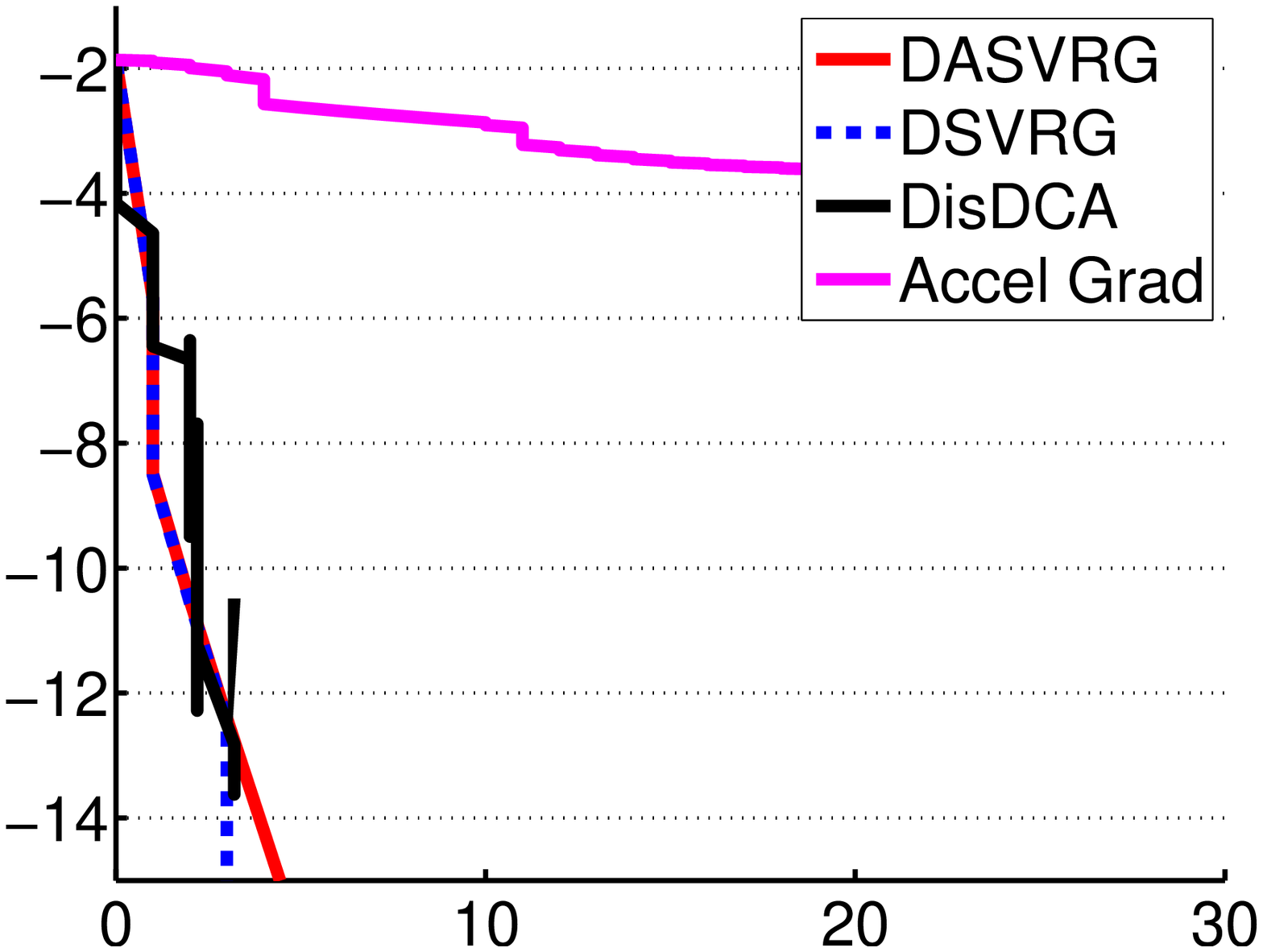}
        & \includegraphics[width=0.28\textwidth]{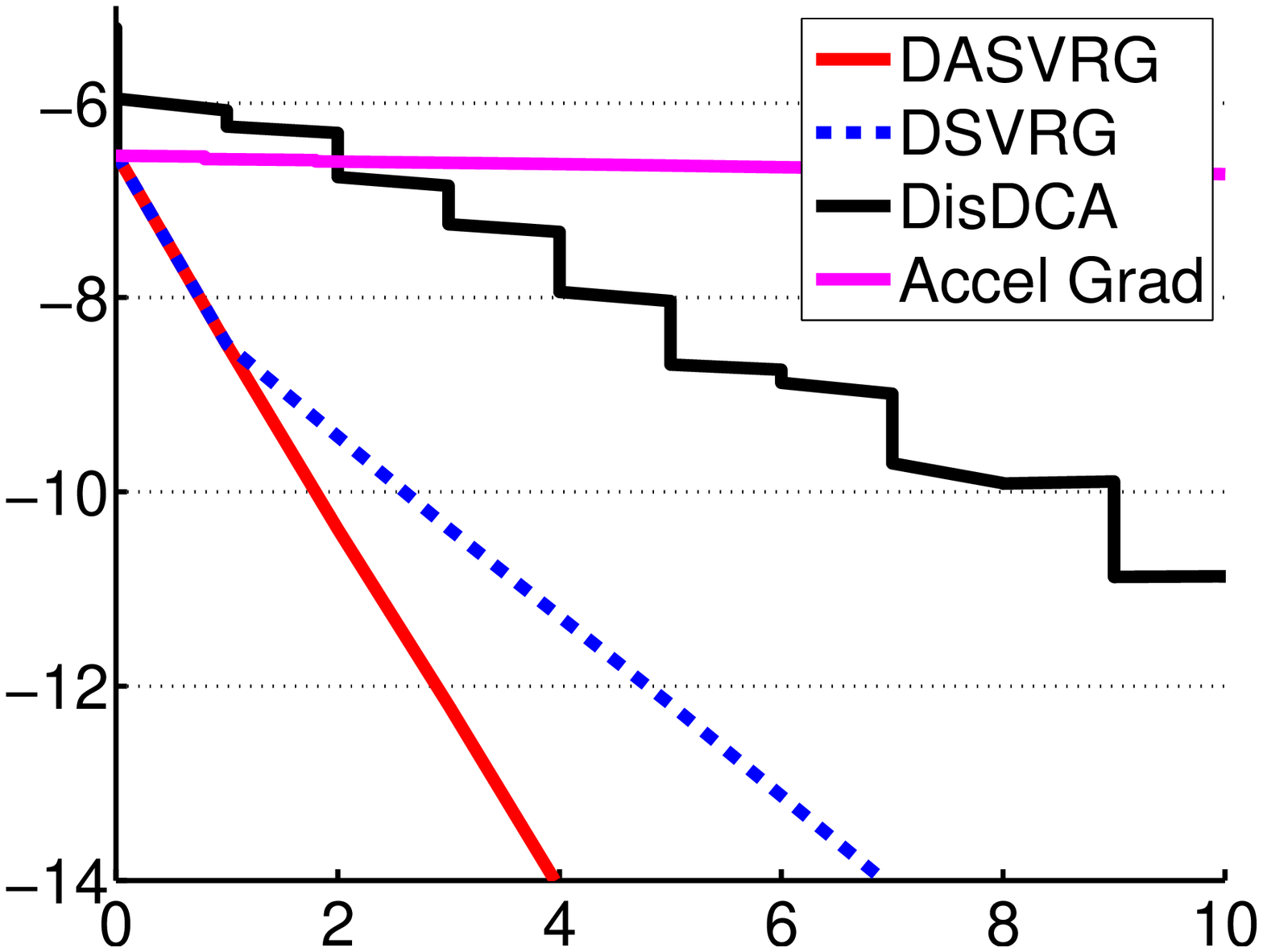}
        & \includegraphics[width=0.28\textwidth]{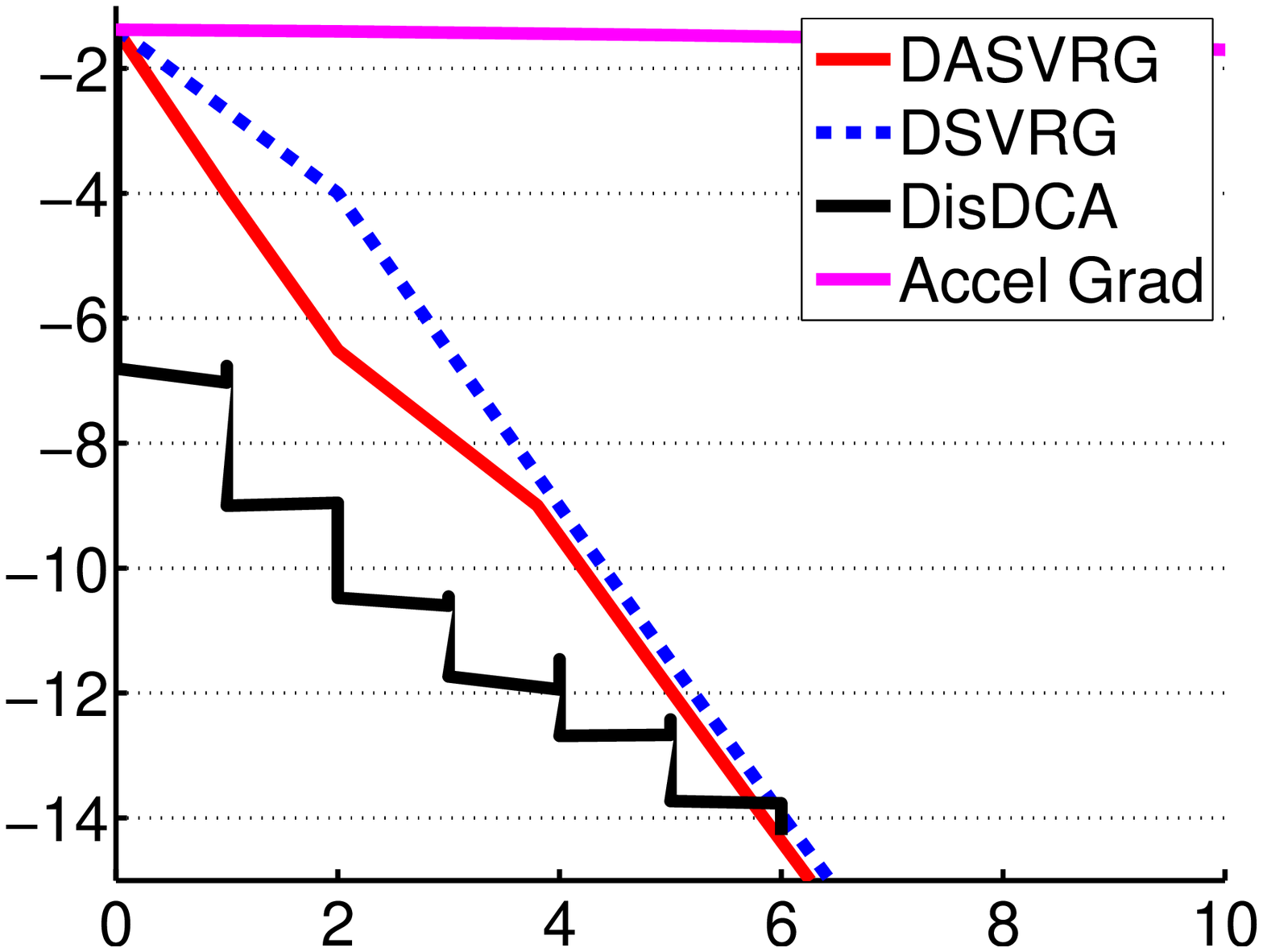}\\
    \raisebox{10ex}{$\frac{1}{N}$}
        &\quad \includegraphics[width=0.28\textwidth]{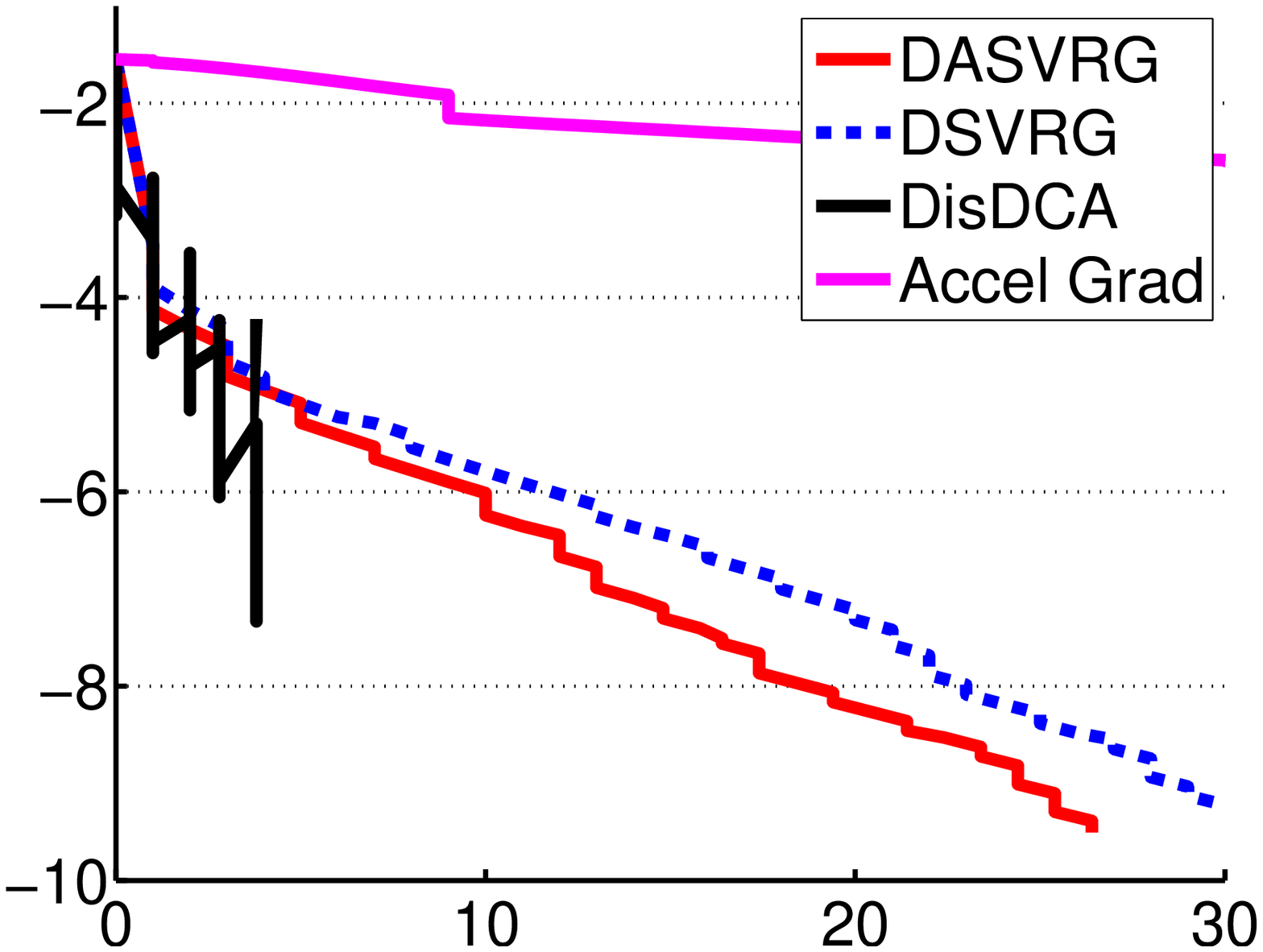}
        & \includegraphics[width=0.28\textwidth]{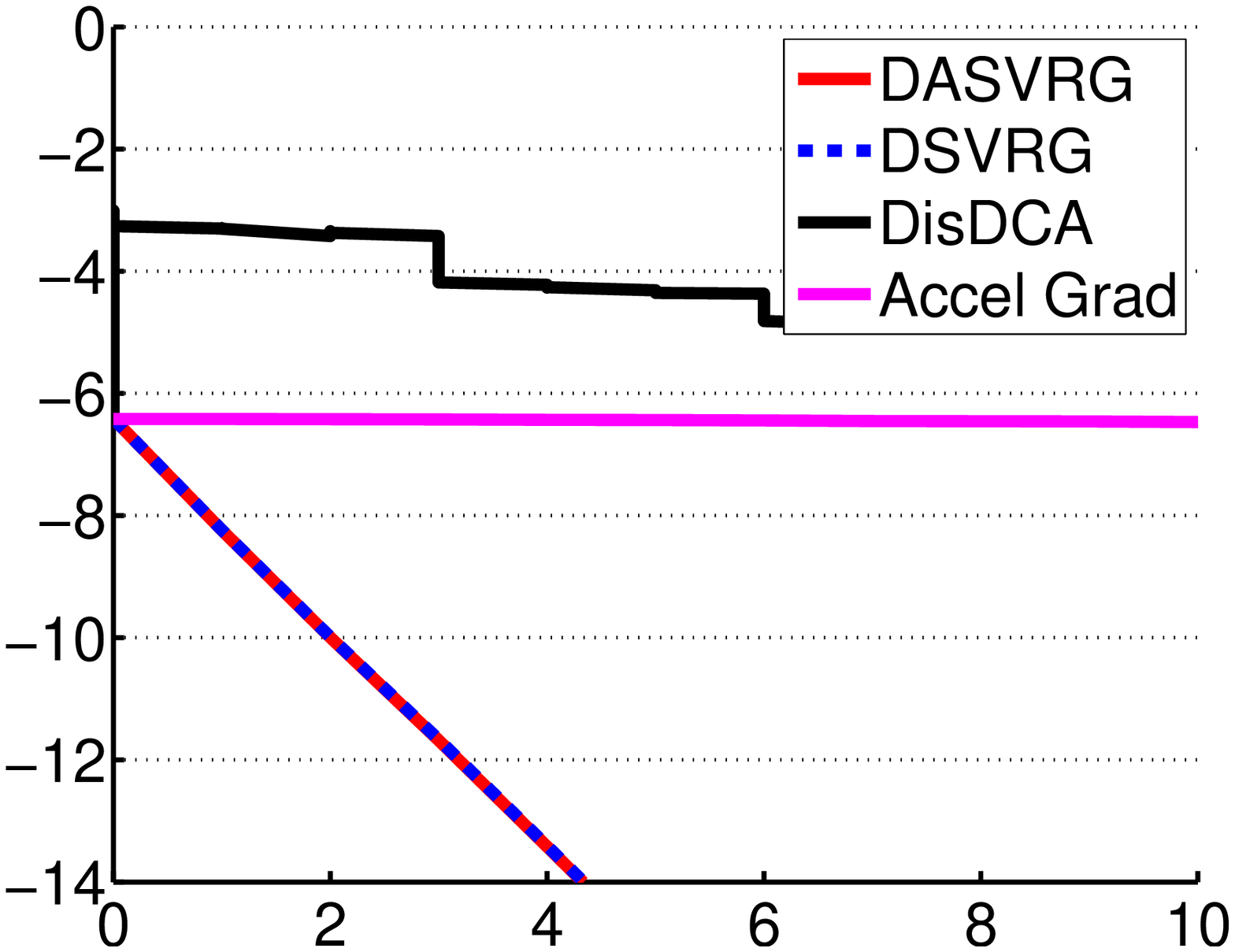}
        & \includegraphics[width=0.28\textwidth]{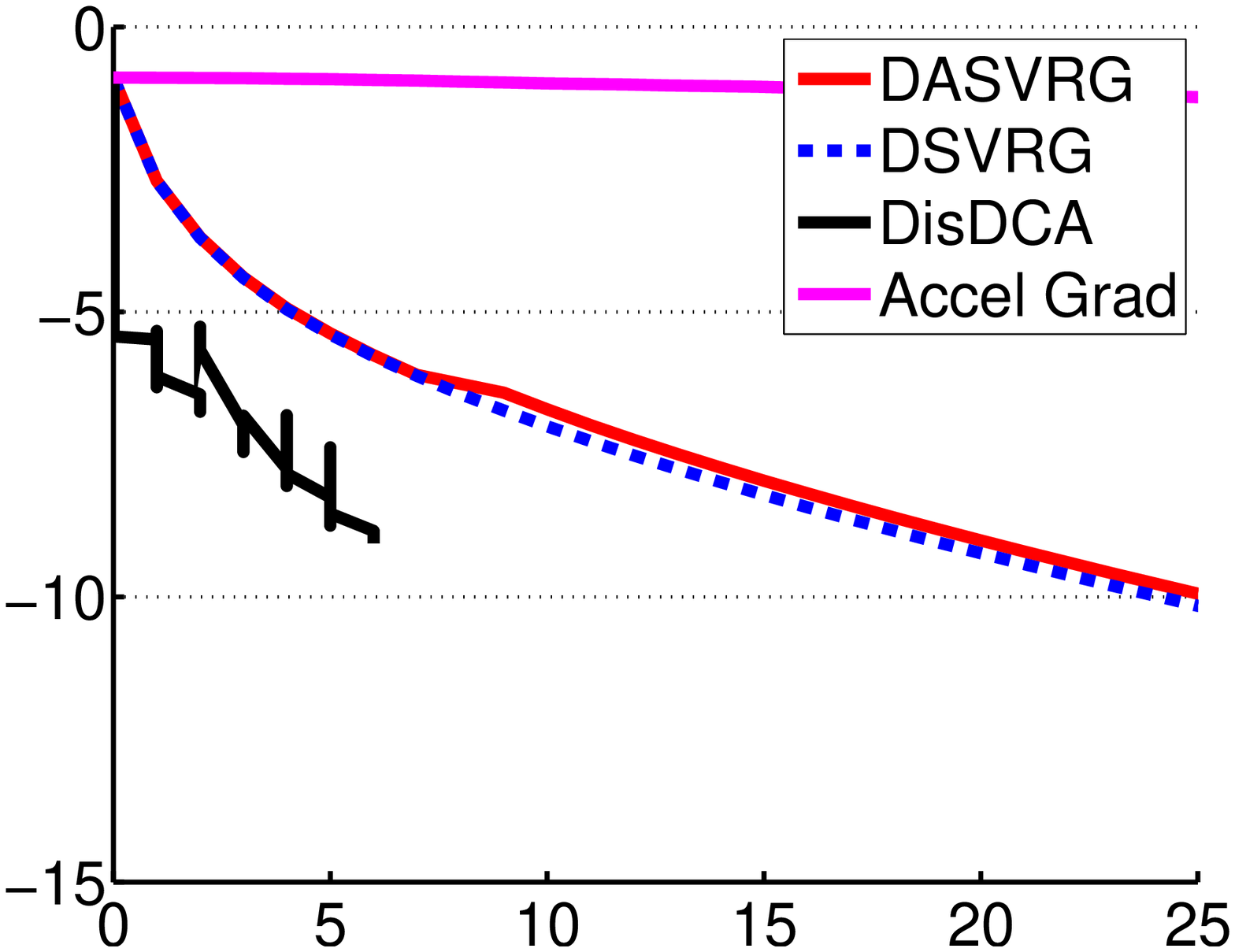}\\
\end{tabular}
\vspace{2ex}
\caption{Comparing the DSVRG and DASVRG methods with DisDCA and the accelerated gradient method (Accel Grad) in runtime.
}
\label{fig:time}
\end{figure}

To compare the performances of algorithms under different values of $m$. We choose the $m=10$ and $15$ and repeat the same experiments on Epsilon data. The numerical results based on the rounds and runtime are shown in Figure~\ref{fig:machine_round} and Figure~\ref{fig:machine_time} respectively. Similar to the case of $m=5$, our DSVRG and DASVRG requires fewer rounds to reach the same $\epsilon$-optimal solution but might require longer runtime on some dataset.

\begin{figure}[t]
    \begin{tabular}[h]{@{}c|ccc@{}}
    $m$ & $\lambda=\frac{1}{N^{\frac{1}{2}}}$ & $\lambda=\frac{1}{N^{\frac{3}{4}}}$ & $\lambda=\frac{1}{N}$ \\
    \hline \\
    \raisebox{10ex}{$10$}
        &\quad \includegraphics[width=0.28\textwidth]{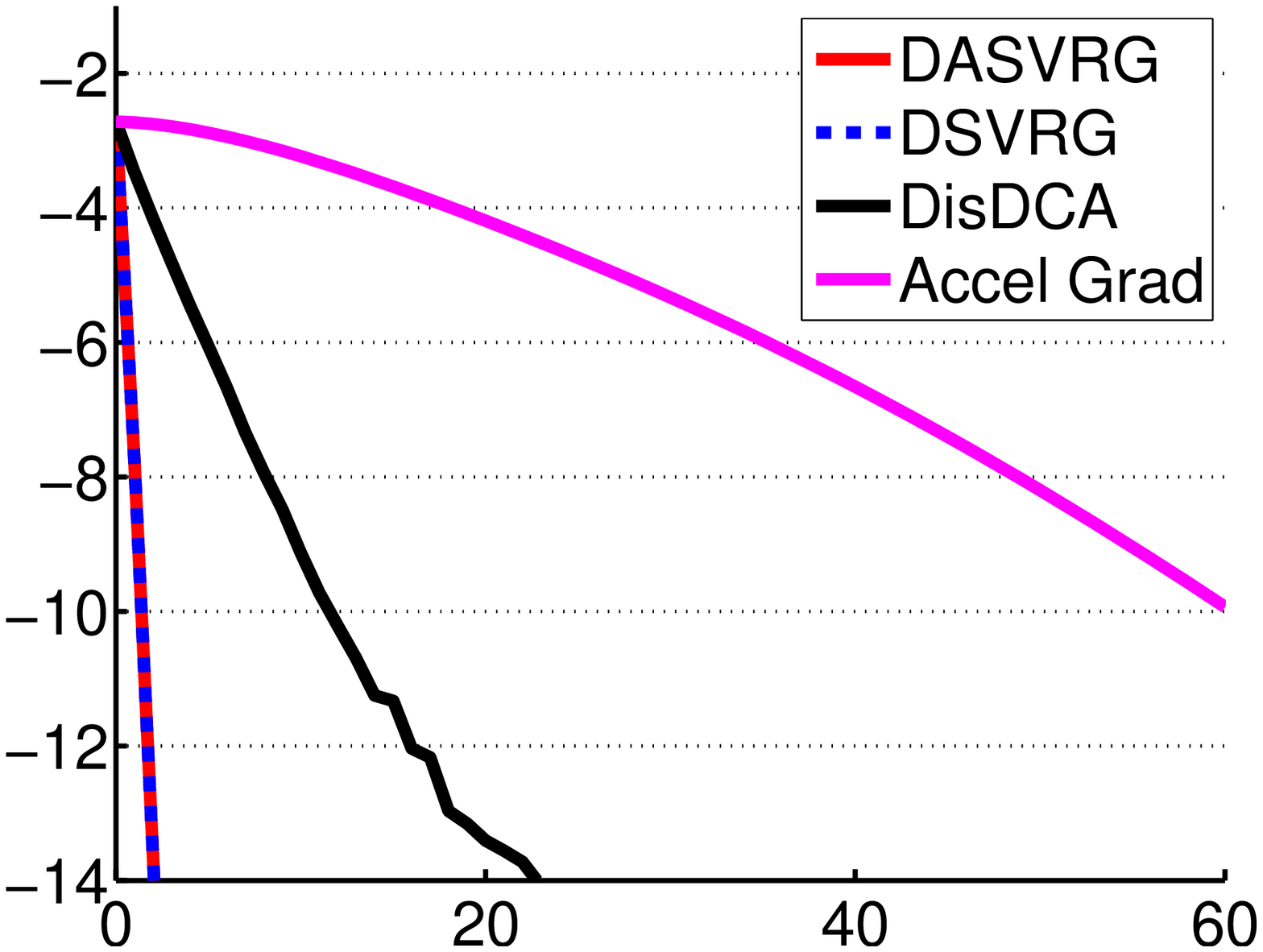}
        & \includegraphics[width=0.28\textwidth]{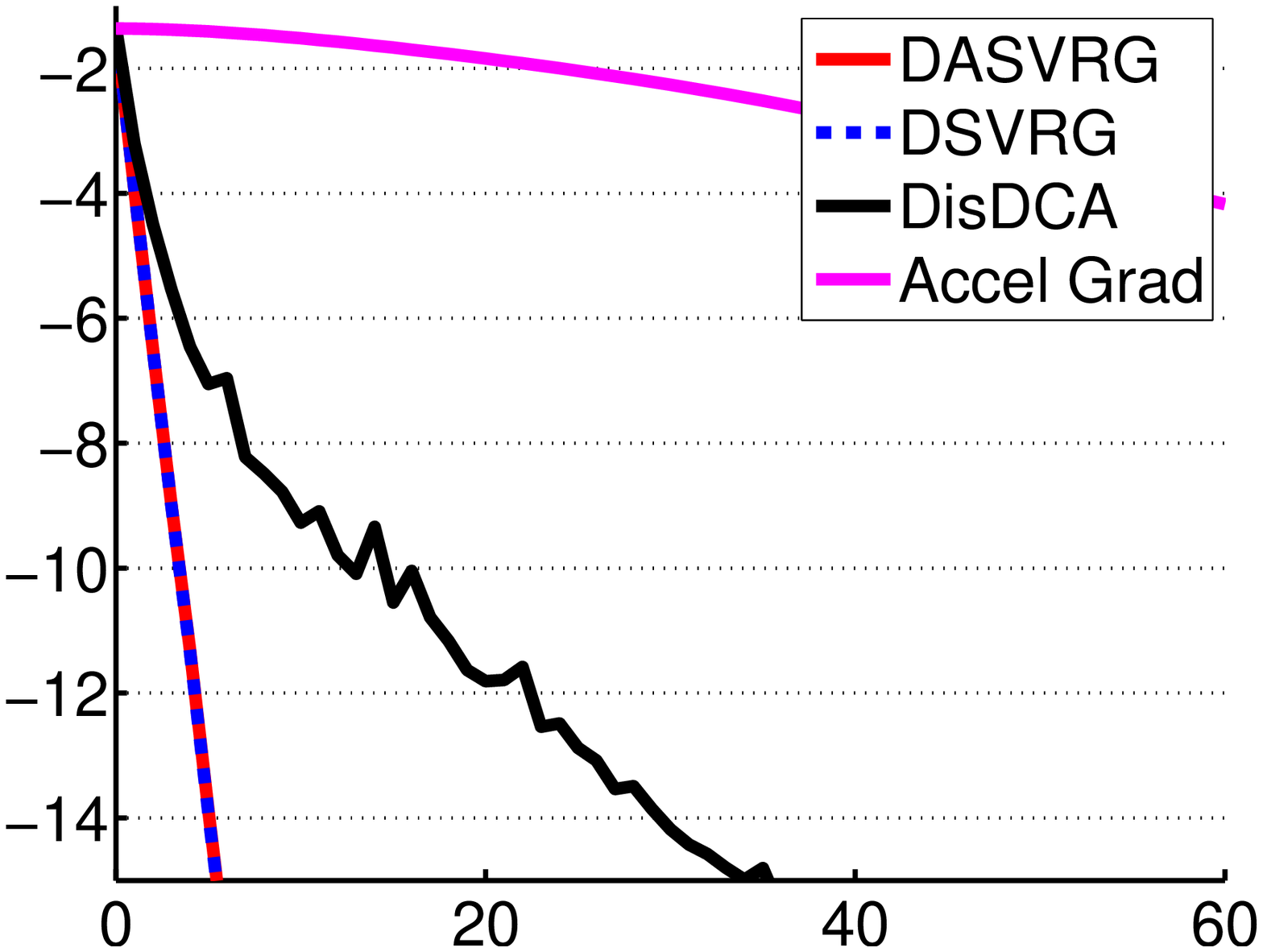}
        & \includegraphics[width=0.28\textwidth]{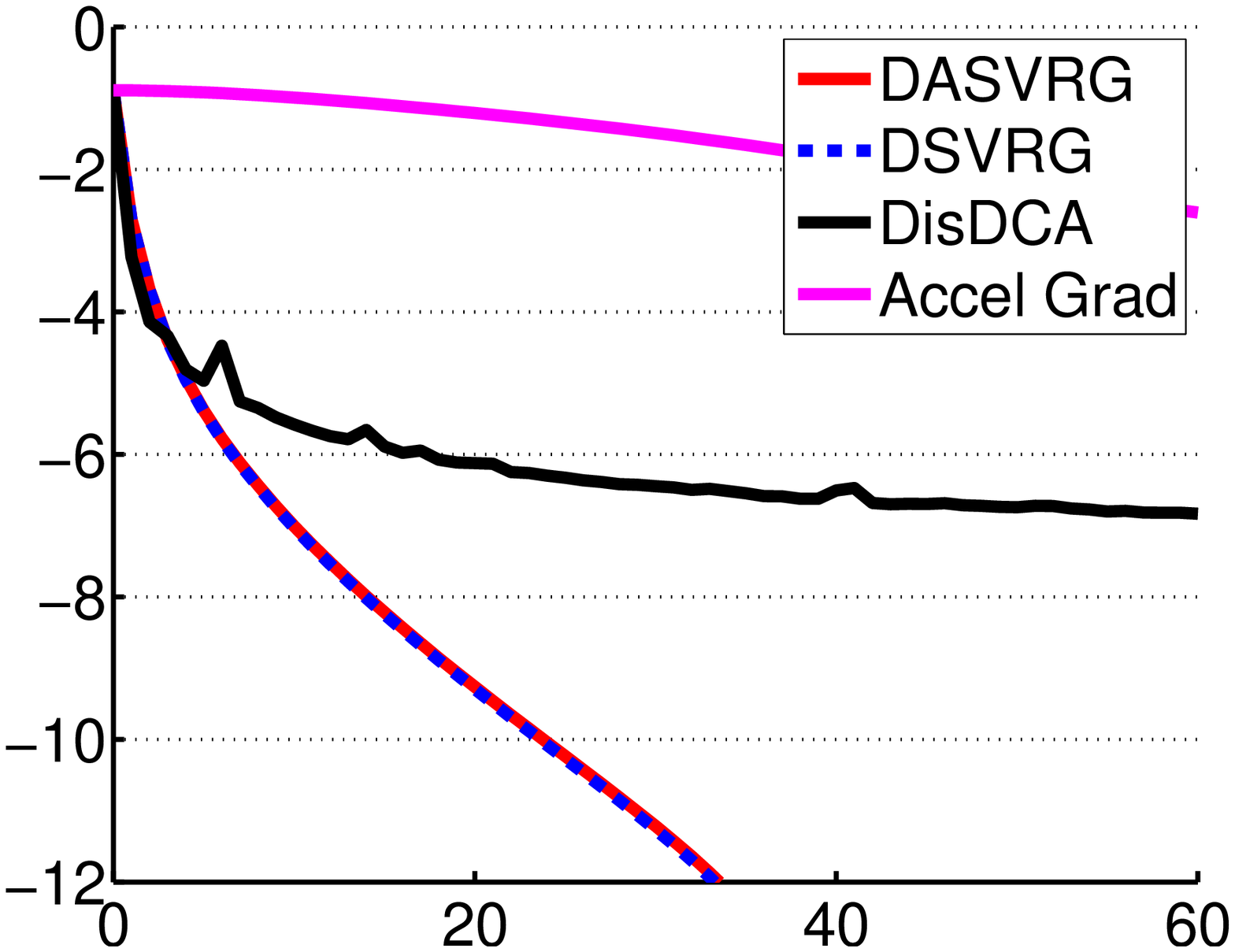}\\
    \raisebox{10ex}{$15$}
        &\quad \includegraphics[width=0.28\textwidth]{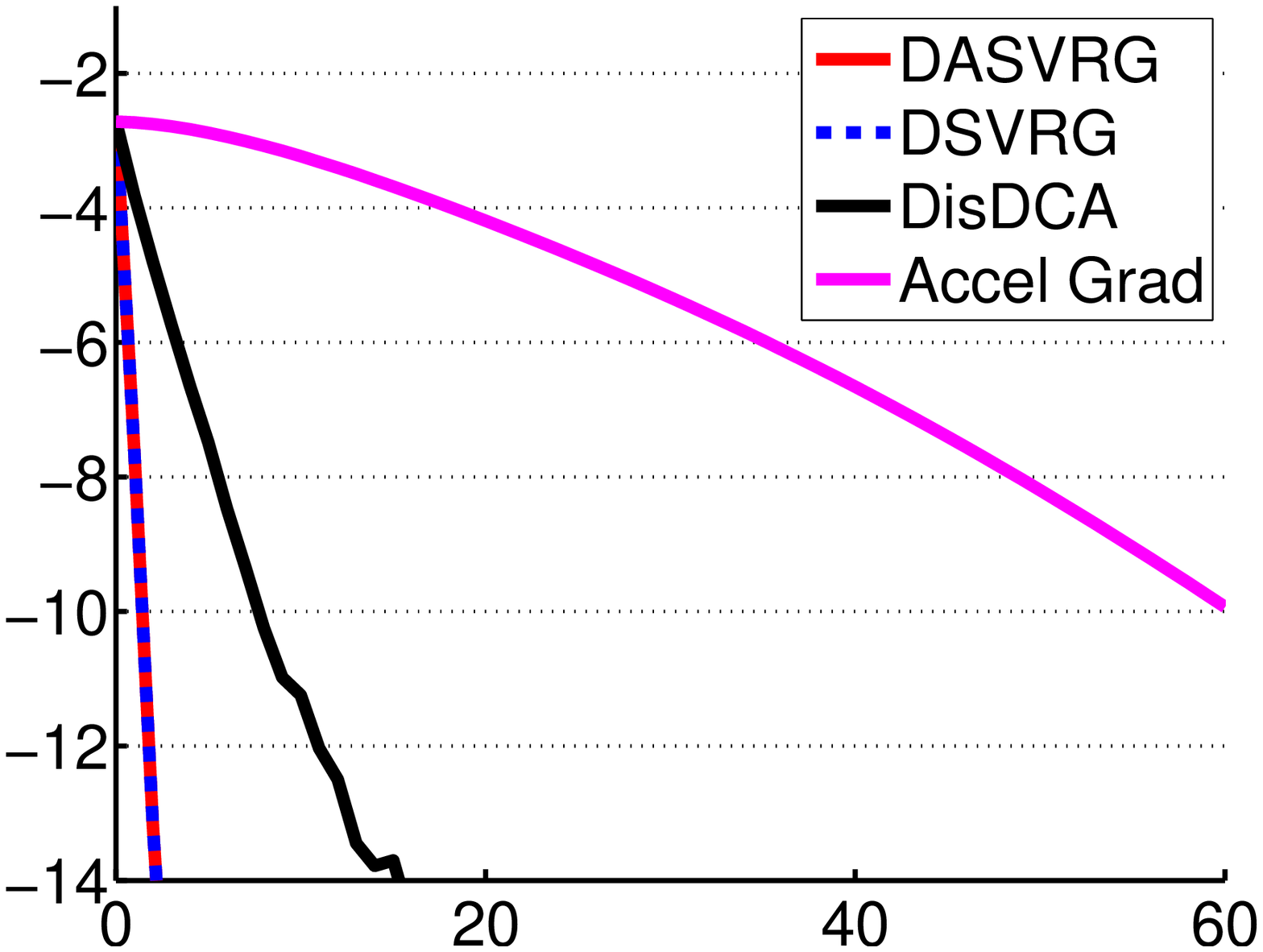}
        & \includegraphics[width=0.28\textwidth]{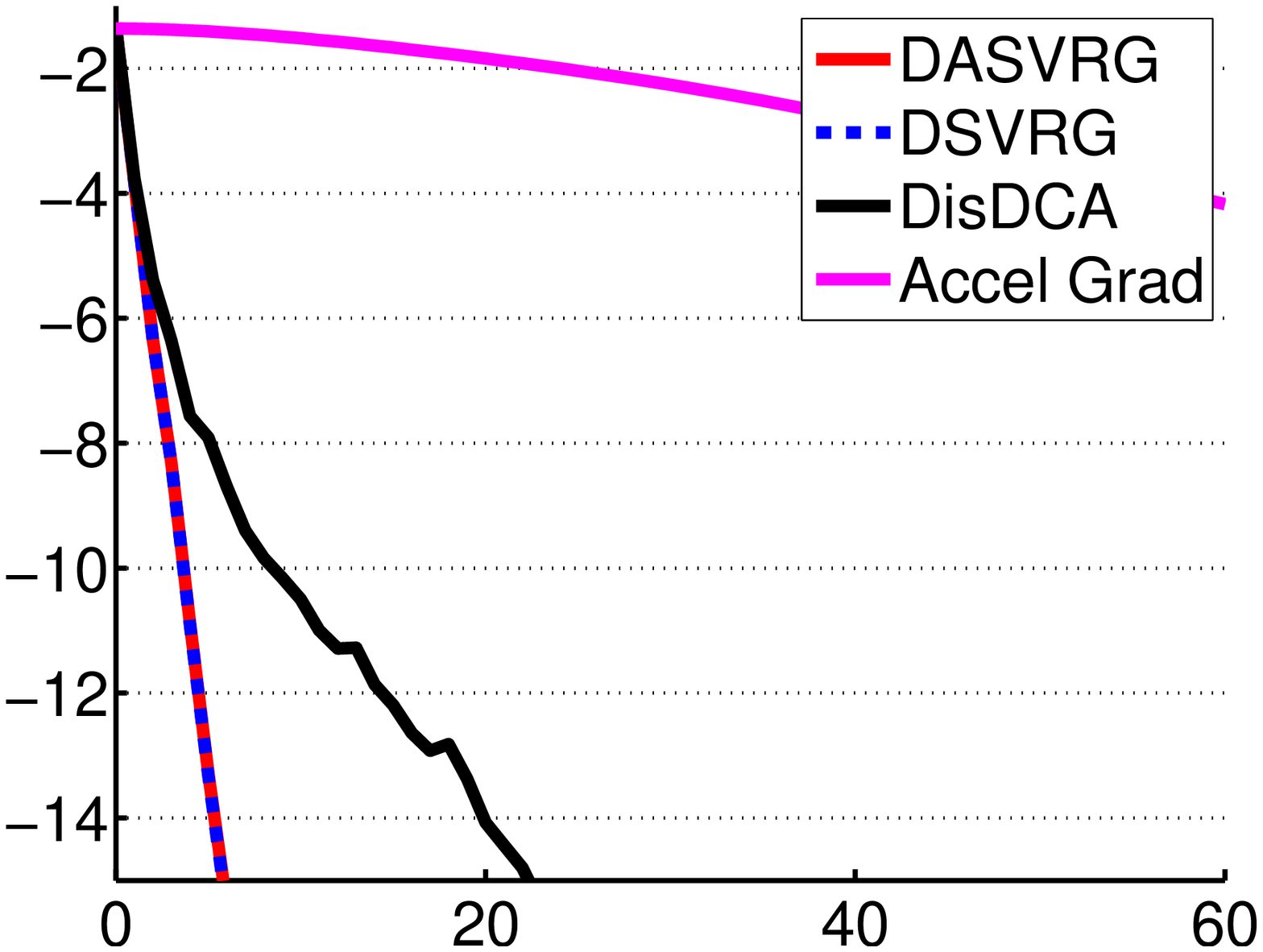}
        & \includegraphics[width=0.28\textwidth]{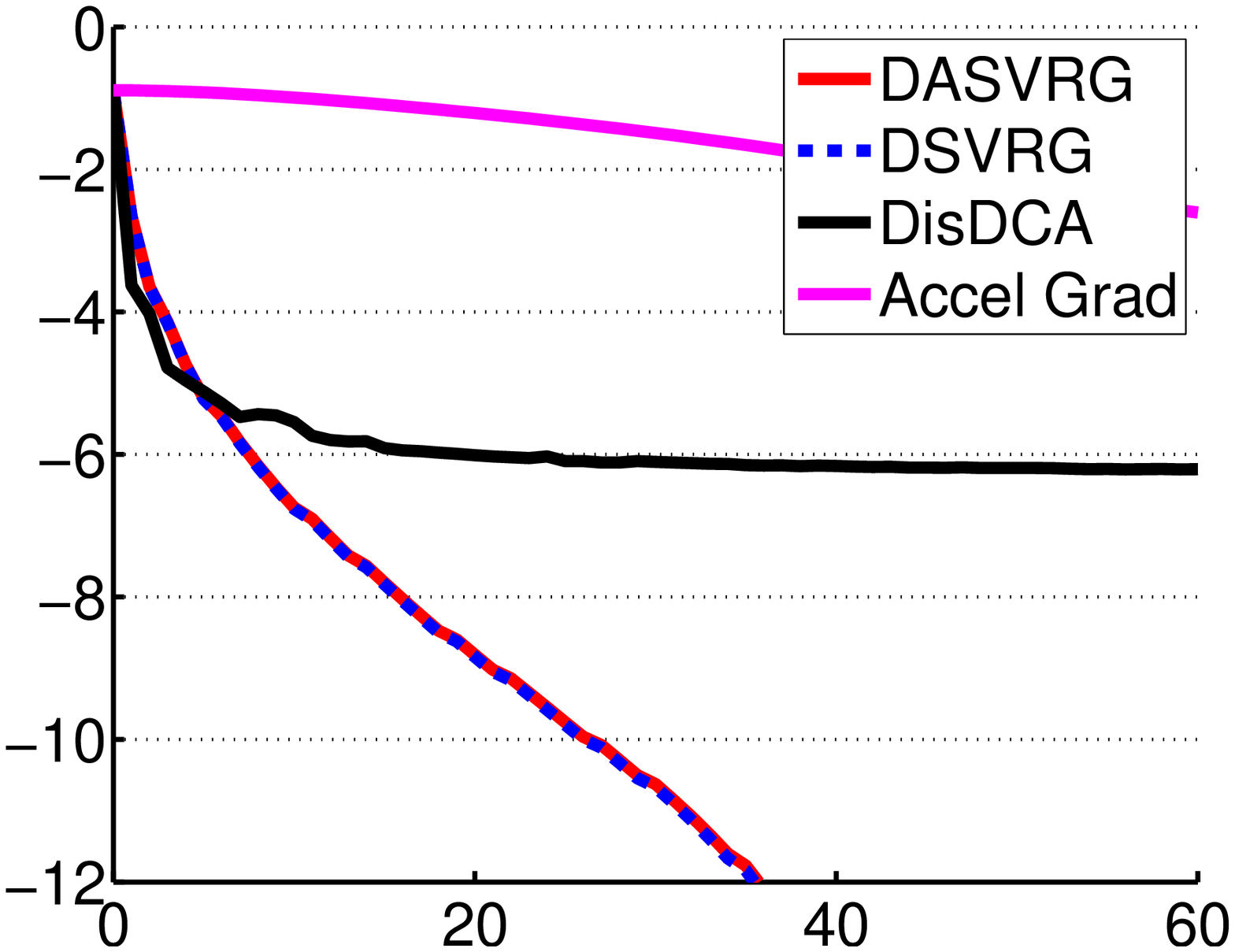}\\
\end{tabular}
\vspace{2ex}
\caption{Comparing the DSVRG and DASVRG methods with DisDCA and the accelerated gradient method (Accel Grad) in rounds.
}
\label{fig:machine_round}
\end{figure}

\begin{figure}[t]
    \begin{tabular}[h]{@{}c|ccc@{}}
    $m$ & $\lambda=\frac{1}{N^{\frac{1}{2}}}$ & $\lambda=\frac{1}{N^{\frac{3}{4}}}$ & $\lambda=\frac{1}{N}$ \\
    \hline \\
    \raisebox{10ex}{$10$}
        &\quad \includegraphics[width=0.28\textwidth]{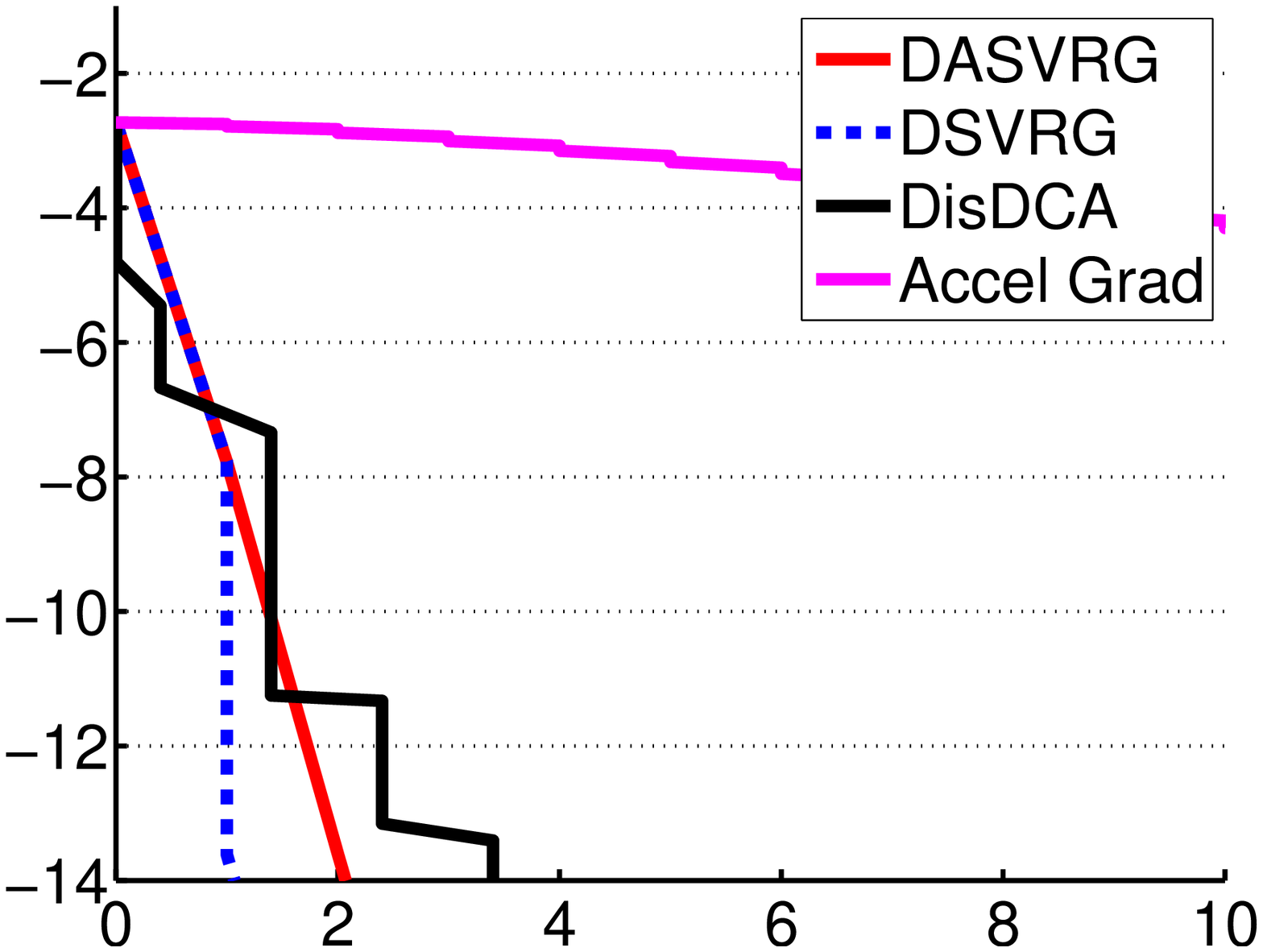}
        & \includegraphics[width=0.28\textwidth]{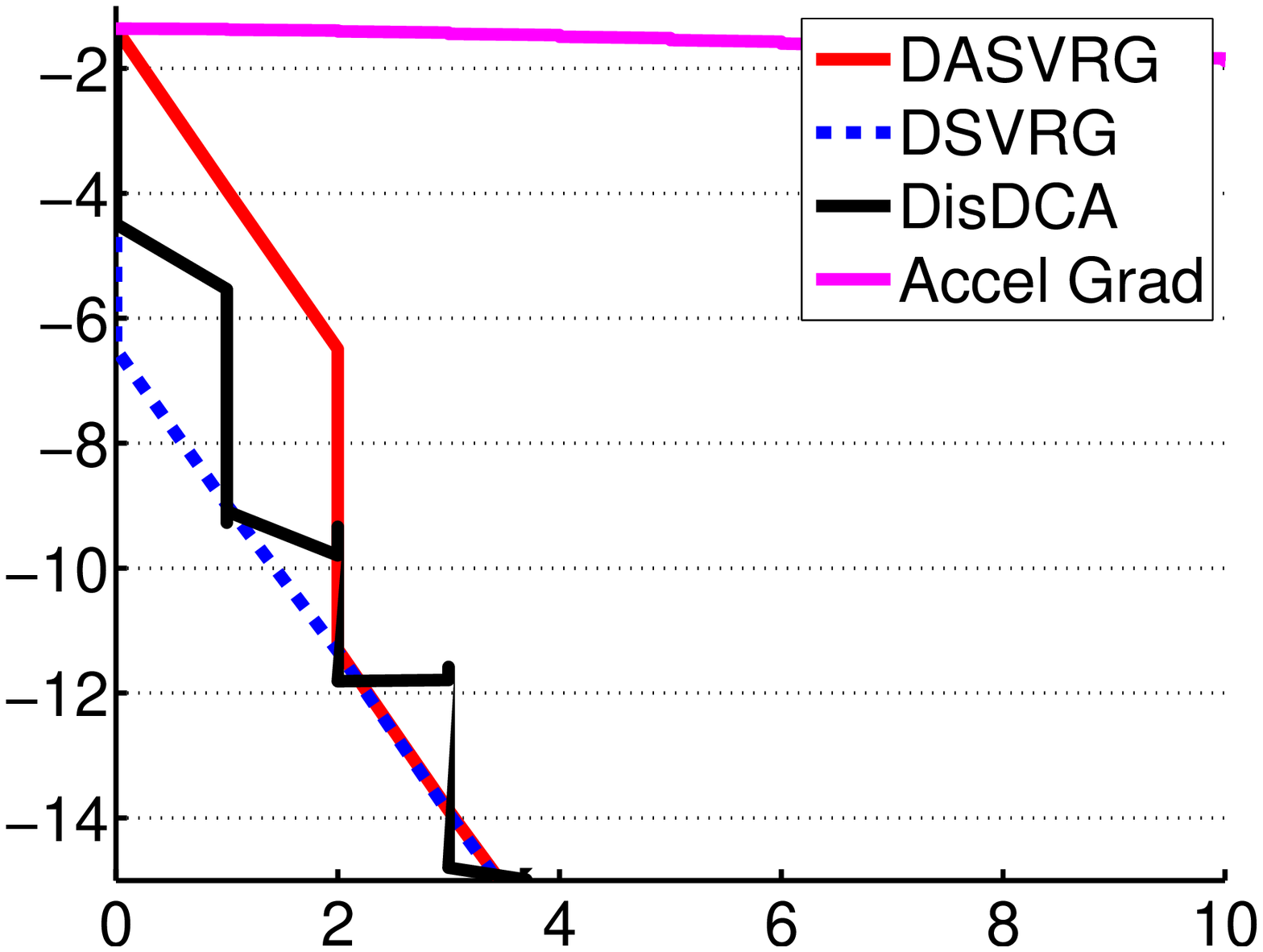}
        & \includegraphics[width=0.28\textwidth]{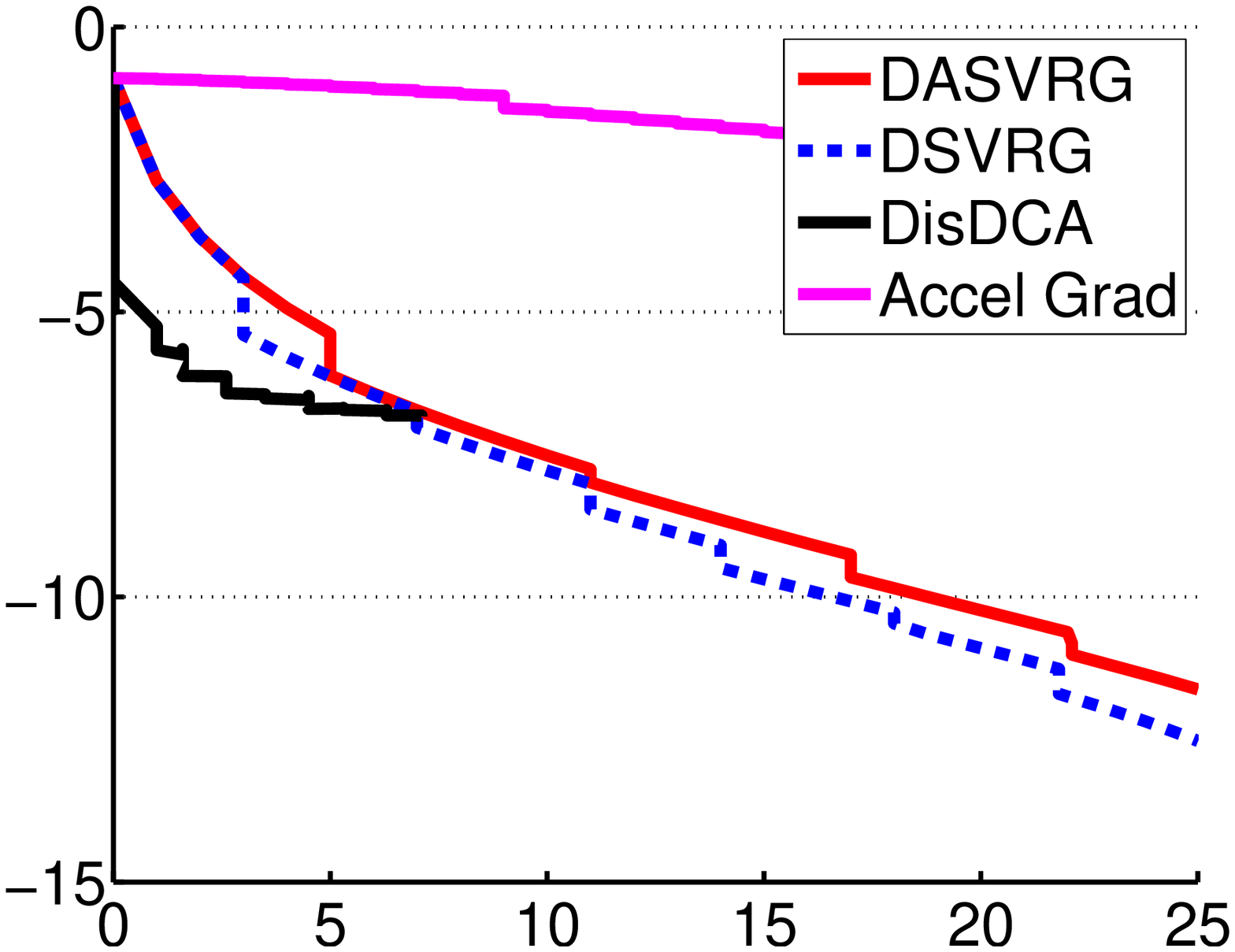}\\
    \raisebox{10ex}{$15$}
        &\quad \includegraphics[width=0.28\textwidth]{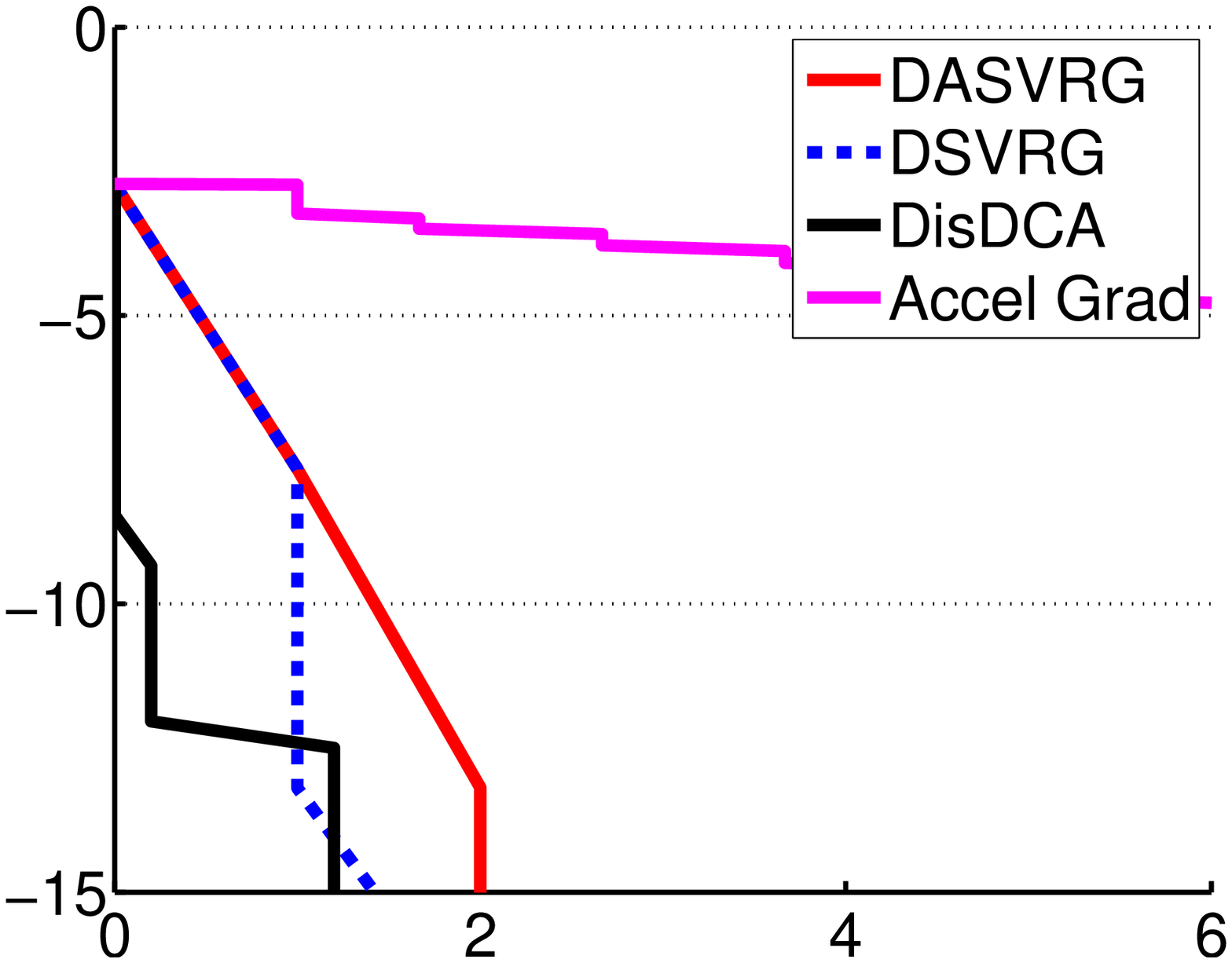}
        & \includegraphics[width=0.28\textwidth]{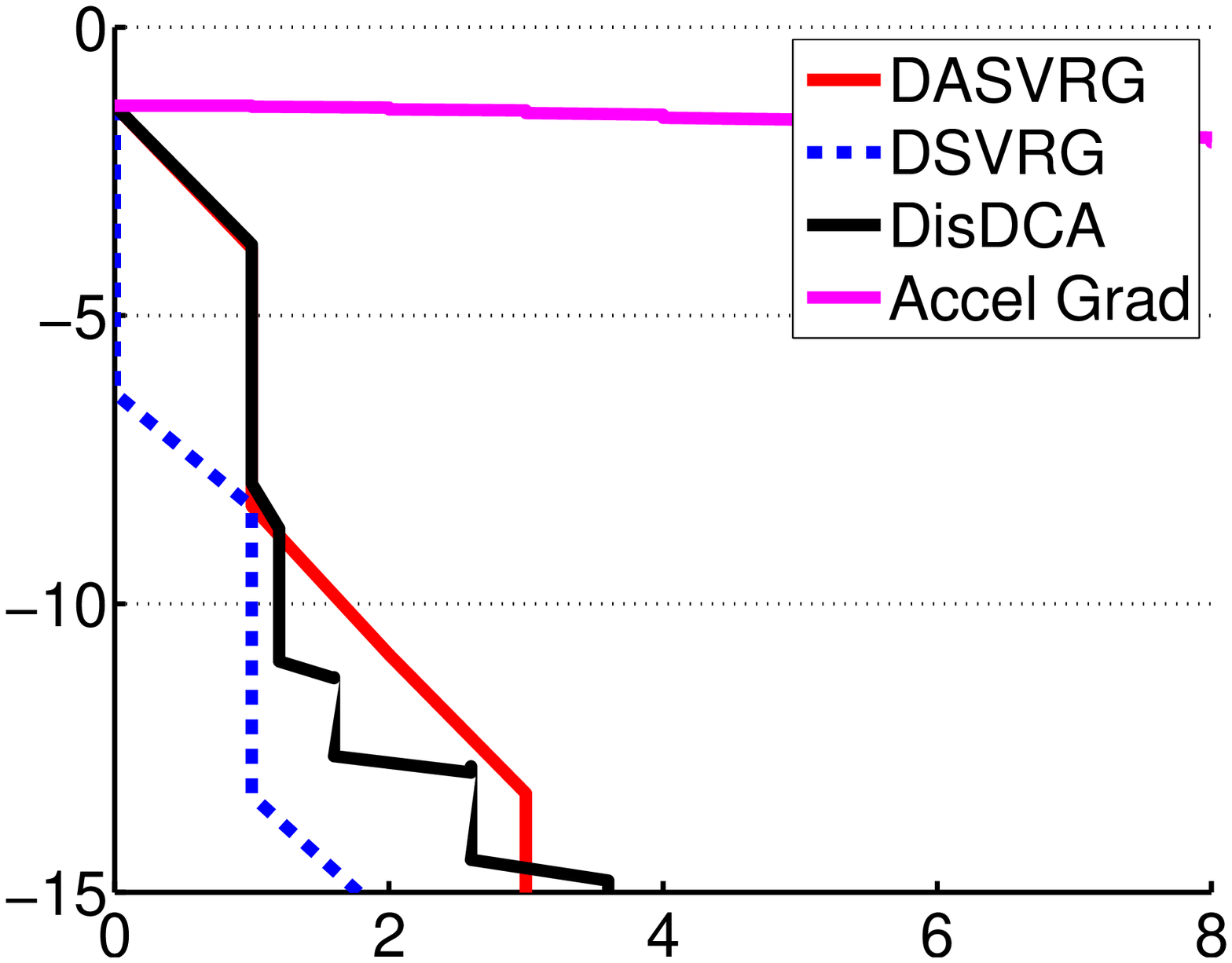}
        & \includegraphics[width=0.28\textwidth]{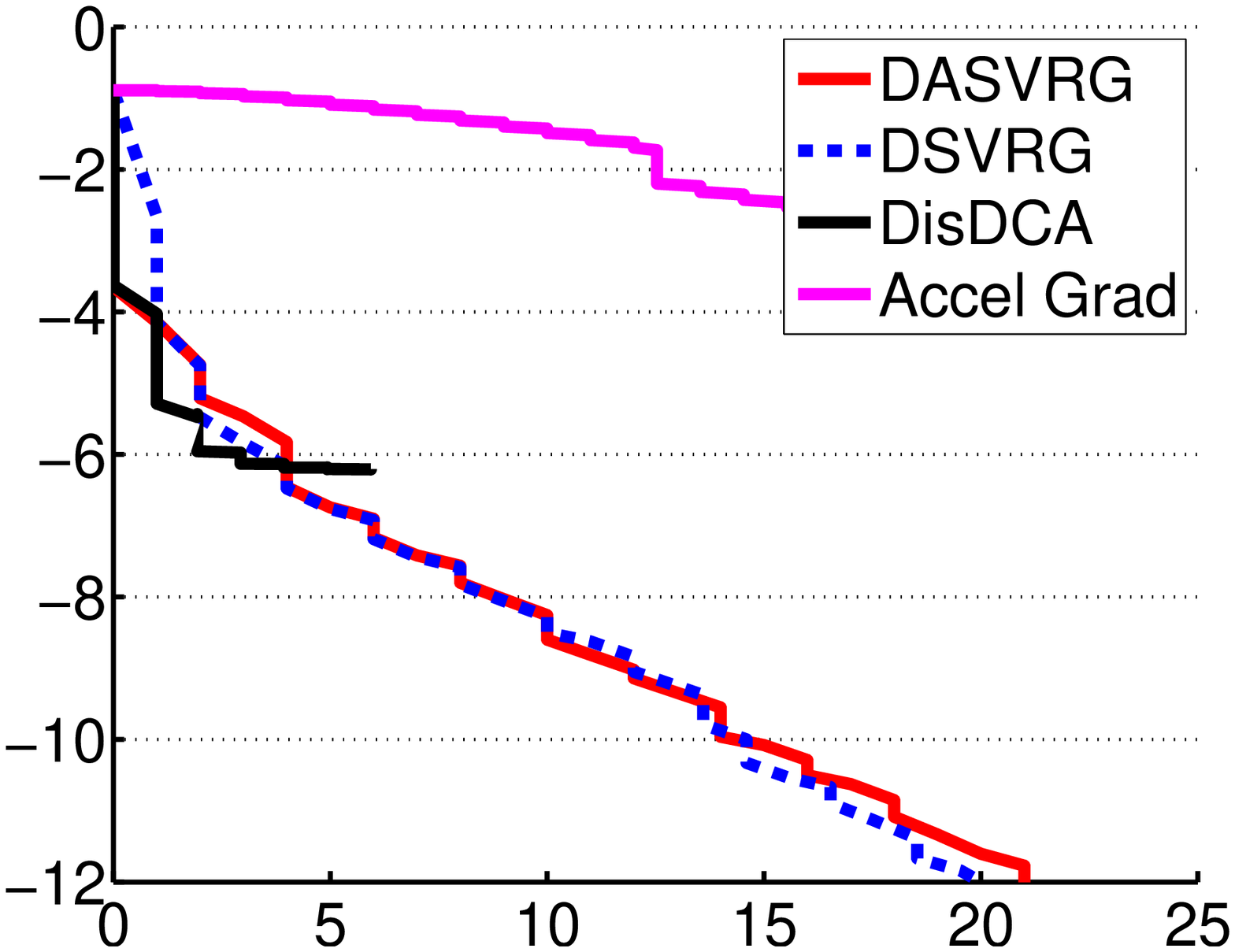}\\
\end{tabular}
\vspace{2ex}
\caption{Comparing the DSVRG and DASVRG methods with DisDCA and the accelerated gradient method (Accel Grad) in runtime.
}
\label{fig:machine_time}
\end{figure}

\section{Conclusion}
\label{sec:conclusion}
We propose a DSVRG algorithm for minimizing the average of $N$ convex functions which are stored in $m$ machines. Our algorithm is a distributed extension of the existing SVRG algorithm, where we compute the batch gradients in parallel while let machines perform iterative updates in serial. Assuming sufficient memory in each machine, we develop an efficient data allocation scheme to store extra functions in each machine to construct the unbiased stochastic gradient in each iterative update. We provide theoretical analysis on the parallel runtime, the amount and the rounds of communication needed by DSVRG to find an $\epsilon$-optimal solution, showing that it is optimal under all of these three metrics under some practical scenario. Moreover, we proposed a DASVRG algorithm that requires even fewer rounds of communication than DSVRG and almost all existing distributed algorithms
using an acceleration strategy by~\cite{frostigicml15} and~\cite{lin2015universal}.

\begin{acknowledgements}
We would like to thank Roy Frostig, Hongzhou Lin, Lin Xiao, and Yuchen Zhang for numerous helpful
discussions throughout various stages of this work.
\end{acknowledgements}

\bibliography{pdcg}
\bibliographystyle{plain}


\end{document}